\documentclass{gtart}

\usepackage[inner=20mm, outer=20mm, textheight=235mm]{geometry}

\usepackage[margin=1cm]{caption}
\usepackage{amsmath, amssymb, latexsym, amscd, graphicx, epstopdf}
\usepackage{mathabx}
\usepackage{euscript, wrapfig, setspace}
\usepackage{subfigure}
\usepackage{epsfig}
\usepackage{psfrag}
\usepackage{rotating, url}
\usepackage{makeidx}

\usepackage{tikz}
\usepackage{tikz-cd}
\usetikzlibrary{calc,matrix}

\usepackage[backref=page]{hyperref}
\hypersetup{
pdftoolbar=true,
pdfmenubar=true,
colorlinks=true,
linkcolor=blue,
citecolor=blue,
filecolor=blue,
urlcolor=blue
}

\DeclareMathAlphabet{\mathpzc}{OT1}{pzc}{m}{it}

\def\bkC{{\rm \kern.24em \vrule width.05em height1.4ex depth-.05ex 
\kern-.26em C}}
\def\C{\bkC}
\def\bksC{{\rm \kern.24em \vrule width.05em height1ex depth-.05ex 
\kern-.26em C}}
\def\sC{\bksC}
\def\bkE{{\rm I\kern-.22em E}}
 
\def\bkH{{\rm I\kern-.22em H}}
\def\H{\bkH} 
\def\bkN{{\rm I\kern-.17em N}}
\def\NN{\bkN}
\def\bkQ{{\rm \kern.24em \vrule width.05em height1.4ex depth-.05ex 
\kern-.26em Q}}

\def\bkR{{\rm I\kern-.17em R}}
\def\RR{\bkR}
\def\bkZ{{\rm Z\kern-.32em Z}}
\def\Z{\bkZ}
\def\bksZ{{\rm Z\kern-.22em Z}}

\def\whl{\mathcal{W}} 

\DeclareMathOperator{\sph}{sph}
\DeclareMathOperator{\val}{val}
\DeclareMathOperator{\Newt}{Newt}

\def\SL{\text{SL}_2(\C)}
\def\PSL{\text{PSL}_2(\C)}

\def\a{\mathfrak{a}}
\def\b{\mathfrak{b}}
\def\c{\mathfrak{c}}
\def\d{\mathfrak{d}}
\def\D{\mathfrak{D}}
\def\Ei{\mathfrak{E}}
\def\PEi{\overline{\Ei}}

\def\X{\mathfrak{X}}
\def\PX{\overline{\mathfrak{X}}}
\def\R{\mathfrak{R}}
\def\PR{\overline{\mathfrak{R}}}
\def\Red{\mathfrak{Red}}
\def\BC{\mathfrak{BC}}
\def\tree{\mathfrak{T}}
\def\Ta{\mathfrak{C}}
\def\PTa{\overline{\mathfrak{C}}}

\def\s{\overline{s}}
\def\u{\overline{u}}
\def\pX{\overline{X}}
\def\pY{\overline{Y}}
\def\pZ{\overline{Z}}
\def\pF{\overline{F}}

\def\P{\mathcal{P}}

\def\tri{\mathcal{T}}

\def\l{\mathcal{L}}
\def\m{\mathcal{M}}
\def\N{\mathcal{PF}}

\def\FH{\mathcal{PL}(\whl)}

\def\prho{\overline{\rho}}
\def\eps{\epsilon}
\def\G{\Gamma}

\def\chidefo{\chi}
\def\edefo{\mathpzc{r}}

\DeclareMathOperator{\Hom}{Hom}
\DeclareMathOperator{\tr}{tr}

\DeclareMathOperator{\im}{im}

\DeclareMathOperator{\te}{t}
\DeclareMathOperator{\pte}{\overline{t}}

\DeclareMathOperator{\ee}{e}
\DeclareMathOperator{\pee}{\overline{e}}

\DeclareMathOperator{\qe}{q}

\DeclareMathOperator{\dev}{dev}
\DeclareMathOperator{\Sph}{Sph}
\DeclareMathOperator{\wt}{wt}

\theoremstyle{plain}

\newtheorem{thm}{Theorem}
\newtheorem*{thm*}{Theorem}
\newtheorem{lem}[thm]{Lemma}
\newtheorem*{lem*}{Lemma}
\newtheorem{cor}[thm]{Corollary}
\newtheorem*{cor*}{Corollary}

\newtheorem*{cla*}{Claim}
\newtheorem{pro}[thm]{Proposition}
\newtheorem*{pro*}{Proposition}
\newtheorem{obs}[thm]{Observation}
\newtheorem*{obs*}{Observation}
\newtheorem{defn}[thm]{Definition}
\newtheorem*{defn*}{Definition}

\theoremstyle{remark}

\addtocounter{MaxMatrixCols}{4}

\begin{document}

\title[The Whitehead link]{Tropical varieties associated to ideal triangulations:}
\subtitle{The Whitehead link complement} 

\author[Stephan Tillmann]{Stephan Tillmann}

\begin{abstract}
This paper illustrates a computational approach to Culler-Morgan-Shalen theory using ideal triangulations, spun-normal surfaces and tropical geometry. Certain affine algebraic sets associated to the Whitehead link complement as well as their logarithmic limit sets are computed. The projective solution space of spun-normal surface theory is related to the space of incompressible surfaces and to the unit ball of the Thurston norm. It is shown that all boundary curves of the Whitehead link complement are strongly detected by its character variety. The specific results obtained can be used to study the geometry and topology of the Whitehead link complement and its Dehn surgeries. The methods can be applied to any cusped hyperbolic 3--manifold of finite volume.
\end{abstract}
\primaryclass{57M25, 57N10}
\secondaryclass{14A25}
\keywords{Whitehead link, spun-normal surface, representation variety, logarithmic limit set}


\makeshorttitle


\section{Introduction}

The study of the topology and the geometry of 3--manifolds via affine algebraic sets, which are related to representations of the fundamental group into $\SL$ and $\PSL$, has led to many deep results about 3--manifolds including the proofs of the Smith conjecture \cite{Mo}, the cyclic surgery theorem \cite{CGLS} and the finite surgery theorem  \cite{BZ96}. At the heart is a construction due to Culler and Shalen~\cite{cs}, which associates essential surfaces in the manifold to the finitely many ideal points of a curve in the character variety via actions on trees. Work of Morgan and Shalen \cite{ms1} associates, more generally, transversely measured codimension-one laminations in the manifold to the ideal points of a component of the character variety. 

For any ideal triangulation of a non-compact, orientable, topologically finite 3--manifold there is a so-called deformation variety~\cite{defo}, which is equipped with an algebraic map into (but possibly not onto) the character variety. The deformation variety is a version of Thurston's parameter space~\cite{t}; both depend on the choice of ideal triangulation and are therefore not unique (not even up to birational equivalence). The deformation variety is compactified in \cite{defo} using the \emph{logarithmic limit set} due to Bergman \cite{gb}, and it is shown that each point in the logarithmic limit set is associated with a transversely measured singular codimension-one foliation of $M,$ formalising a relationship originally observed by Thurston.
The logarithmic limit set of the deformation variety can be computed with the software package \tt{gfan}\rm, which resulted from work of Bogart, Jensen, Speyer, Sturmfels and Thomas~\cite{bjsst}. 
The ramification of this is an algorithmic approach to the character variety techniques of Culler, Morgan and Shalen in general, and to computing boundary curves strongly detected by the character variety in particular.

This paper illustrates this algorithmic approach to the character variety techniques of Culler, Morgan and Shalen by applying it to the Whitehead link complement (with its standard triangulation). 
To give a complete and unified picture, all relevant varieties and the maps between them are computed for the Whitehead link complement, and the spun-normal surfaces and ideal points are analysed in detail. The resulting information recovers the computation of the space of incompressible surfaces by Floyd and Hatcher~\cite{fh}, as well as the boundary curve space computed by Lash~\cite{la} and Hoste and Shanahan~\cite{hs}. It is shown that all non-compact spun-normal surfaces are essential and dual to ideal points of the character variety. However, it is also shown that there are \emph{fake ideal points} of the deformation variety: there is a trivial, closed normal surface associated to an ideal point of the deformation variety of the Whitehead link complement.
  
The contents of this paper is as follows:


{\bf Section \ref{sec:prel_varieties} (Character varieties, ideal points and essential surfaces)} This section summarises background material pertaining to various varieties associated to representations of 3--manifold groups into $\SL$ and $\PSL.$ It serves as a quick trip through Culler-Shalen theory and discusses Bergman's logarithmic limit set~\cite{gb} and the eigenvalue variety~\cite{tillus_ei}. The eigenvalue variety is a multi--cusped analogue of the $A$--polynomial of Cooper, Culler, Gillet, Long and Shalen~\cite{ccgls}.


{\bf Section \ref{sec:tillus_defo} (The deformation variety and normal surfaces)} This section describes the relationships between the deformation variety, the character and the eigenvalue varieties exhibited in \cite{defo}. The required elements of normal surface theory of an ideally triangulated 3--manifold are outlined, and the projective admissible solution space is defined and related to the logarithmic limit set of the deformation variety.


None of the material up to this point is new; it is mostly collated from \cite{gb, sh1, tillus_ei, tillmann08-finite, defo} and included for convenience of the reader.


{\bf Section \ref{sec:whl} (The Whitehead link) }
The ideal triangulation $\tri_\whl$ of the complement $\whl$ of the Whitehead link that is used for all computations in this paper is described, and differences between the right-handed and the left-handed Whitehead link, as well as different choices of meridians and longitudes are discussed.


{\bf Section \ref{whl:associated varieties} (Deformation and Eigenvalue varieties) }
The deformation variety $\D(\tri_\whl)$ and the Dehn surgery component $\PEi_0(\whl)$ of the $\PSL$--eigenvalue variety are computed. It turns out that $\D(\tri_\whl)$ and $\PEi_0(\whl)$ are birationally equivalent; this is shown via a variety $\PTa(\whl)$, which parameterises certain representations of $\pi_1(\whl)$ into $\PSL$ and which can be viewed as a de-singularisation of $\D(\tri_\whl).$ These computations were used by Indurskis~\cite{gi} to determine Culler-Shalen semi-norms. The $\SL$-- and $\PSL$--character varieties of $\whl$, and in particular their respective Dehn surgery components $\X_0 (\whl )$ and $\PX_0 (\whl )$, are computed in Appendix \ref{whl:Character varieties}.


All maps constructed are summarised in below commutative diagram. The horizontal maps are birational isomorphisms, and all other maps are generically 4--to--1 and correspond to actions by the Klein four group. The bar in the notation indicates that a variety is associated to representations into $\PSL.$

\begin{center}
\begin{tikzpicture}[ampersand replacement=\&]
  \matrix (m) [matrix of math nodes,row sep=3em,column sep=4em,minimum width=2em] {
  \& \Ta_0(\whl) \&  \Ei_0(\whl)  \\
    \D(\tri_\whl)  \& \PTa_0(\whl) \& \PEi_0(\whl)  \\
    \& \PX_0 (\whl ) \&   \\};
  \path[-stealth]
    (m-1-2) edge node [left] {$q_1$} (m-2-2)
    (m-2-2) edge node [right] {$$} (m-3-2)
    (m-1-3) edge node [right] {$q_3$} (m-2-3)
    (m-2-1) edge node [above] {$\edefo$} (m-2-2)
    (m-1-2) edge node [above] {$\ee$} (m-1-3)
    (m-2-2) edge node [above] {$\pee$} (m-2-3)
    (m-2-1) edge node [below] {$\chidefo$} (m-3-2);
\end{tikzpicture}
\end{center}


{\bf Section \ref{whl:embedded normal surfaces} (Embedded surfaces) }
The projective admissible solution space $\N(\tri_\whl)$ of spun-normal surface theory is computed, and a criterion due to Dunfield is used to determine which spun-normal surfaces are essential. The unit ball of the Thurston norm is derived using this information. Results by Walsh~\cite{wa} and Kang and Rubinstein~\cite{KR-2015} on the normalisation of essential surfaces are used to determine the space $\FH$ of all essential surfaces and the boundary curve space $\BC (\whl).$ This computation is compared with work of Floyd and Hatcher~\cite{fh}, Lash~\cite{la} and Hoste and Shanahan~\cite{hs}. 


{\bf Section \ref{whl:surfaces arising from degenerations} (Logarithmic limit sets) } The logarithmic limit set $\D_\infty(\tri_\whl)$ of the deformation variety $\D(\tri_\whl)$ is homeomorphic to a closed subset of $\N(\tri_\whl)$ according to \cite{defo}. Both $\D_\infty(\tri_\whl)$ and the logarithmic limit set of the eigenvalue variety $\PEi_0(\whl)$ are determined. The computation in particular shows that all strict boundary curves of $\whl$ are detected by the Dehn surgery component of the character variety. Since the boundary curves of fibres are detected by reducible characters, this shows that all boundary slopes of the Whitehead link are strongly detected. In addition, it is shown that $\D_\infty(\tri_\whl)$ has a fake ideal point.


{\bf Acknowledgements.}
The author thanks Debbie Yuster for the calculation in Appendix~\ref{sec:log lim debbie}, and the referee for suggestions that improved the exposition of this paper. This research has benefitted from the use of the following software packages: \tt{SnapPy }\rm by Marc Culler, Nathan Dunfield, Matthias Goerner and Jeff Weeks~\cite{snappy}, \tt{Regina }\rm by Ben Burton \cite{bab}, and \tt{gfan }\rm by Anders Jensen \cite{gfan}. The author thanks Craig Hodgson, Diane Maclagan, Hyam Rubinstein, Saul Schleimer and Bernd Sturmfels for helpful discussions. 
Research of the author is currently supported by an Australian Research Council Future Fellowship (project number FT170100316).


\section{Character varieties, ideal points and essential surfaces}
\label{sec:prel_varieties}

This section summarises background material pertaining to various varieties associated to representations of 3--manifold groups into $\SL$ and $\PSL.$ The first four subsections outline the bread and butter of Culler-Shalen theory; the reader is referred to Shalen~\cite{sh1} for details. This is followed up with a description of Bergman's logarithmic limit set~\cite{gb}  in \S\ref{combinatorics:Logarithmic limit set}, and the eigenvalue varieties~\cite{tillus_ei} in \S\S\ref{boundary:Eigenvalue variety}--\ref{boundary:PSL-Eigenvalue variety}.
 

\subsection{Character variety}
\label{sec:character variety}

Let $\G$ be a finitely generated group. The \emph{representation variety} of
$\G$ is the set $\R (\G) = \Hom (\G, \SL )$, which is regarded as an affine
algebraic set (see \cite{cs}). Two representations are
\emph{equivalent} if they differ by an inner automorphism of $\SL$. A
representation is \emph{irreducible} if the only subspaces of $\C^2$ invariant
under its image are trivial. This is equivalent to saying that the
representation cannot be conjugated to a representation by upper triangular
matrices. Otherwise a representation is \emph{reducible}.

Each $\gamma \in \G$ defines a \emph{trace function} $I_\gamma : \R (\G) \to
\C$ by $I_\gamma(\rho) = \tr\rho(\gamma)$, which is an element of the
coordinate ring $\C [\R (\G)]$.  For each $\rho \in \R (\G)$, its
\emph{character} is the function $\chi_\rho : \Gamma \to \C$ defined by
$\chi_\rho (\gamma ) = \tr\rho (\gamma )$. Irreducible representations are
determined by characters up to equivalence, and the reducible representations
form a closed subset $\Red (\G)$ of $\R (\G)$. The collection of characters
$\X (\G)$ can be regarded as an affine algebraic set, which is called the
\emph{character variety}. There is a regular map $\te : \R (\G) \to \X (\G)$
taking representations to characters.  If $\G$ is the fundamental group of a
topological space $M$, then $\R (M)$ and $\X (M)$ denote $\R (\G)$ and $\X
(\G)$ respectively.

There also is a character variety arising from representations into $\PSL$,
see Boyer and Zhang \cite{bozh}, and the relevant objects are denoted by
placing a bar over the previous notation. The natural map $\qe: \X (\G) \to
\PX (\G)$ is finite--to--one, but in general not onto.  As with the
$\SL$--character variety, there is a surjection $\pte : \PR (\G) \to \PX
(\G)$, which is constant on conjugacy classes, and with the property that if
$\prho$ is an irreducible representation, then $\pte^{-1}(\pte (\prho ))$ is
the orbit of $\prho$ under conjugation. 


\subsection{Dehn surgery component}
\label{Dehn surgery component}

Suppose $M$ is the interior of an orientable, compact, connected 3--manifold $\overline{M}$ with non-empty boundary consisting of a pairwise disjoint union of tori. If $M$ admits a complete hyperbolic structure of finite volume,
then there is a discrete and faithful representation $\pi_1(M) \to \PSL$. This
representation is necessarily irreducible, as hyperbolic geometry otherwise
implies that $M$ has infinite volume. 
The following result (see
\cite{t, sh1}) will be relied on:

\begin{thm}[Thurston]  \label{thurston thm}
  Let $M$ be a complete hyperbolic 3--manifold of finite volume with $h$
  cusps, and let $\prho_0 : \pi_1(M) \to \PSL$ be a discrete and faithful
  representation associated to the complete hyperbolic structure. Then
  ${\prho_0}$ admits a lift $\rho_0$ into $\SL$ which is still discrete and
  faithful. The (unique) irreducible component $\X_0$ in the $\SL$--character
  variety containing the character $\chi_0$ of $\rho_0$ has (complex)
  dimension $h$.
  
  Furthermore, if $T_1, \ldots , T_h$ are the boundary tori of a compact core
  of $M$, and if $\gamma_i$ is a non--trivial element in $\pi_1(M)$ which is
  carried by $T_i$, then $\chi_0 ( \gamma_i) = \pm 2$ and $\chi_0$ is an
  isolated point of the set
  \begin{equation*}
    \X^* = \{ \chi \in \X_0 \mid I^2_{\gamma_1} = \ldots = I^2_{\gamma_h} = 4 \}.
  \end{equation*}
\end{thm}

The respective irreducible components containing the so--called \emph{complete
  representations} $\rho_0$ and ${\prho_0}$ are denoted by $\R_0(M)$ and
$\PR_0(M)$ respectively.  In particular, $\te (\R_0) = \X_0$ and $\pte (\PR_0)
= \PX_0$ are called the respective \emph{Dehn surgery components} of the
character varieties of $M$, since the holonomy representations of hyperbolic
manifolds or orbifolds obtained by performing high order Dehn surgeries on $M$
are near $\prho_0$ (see \cite{t}). We remark that Dehn surgery components are generally
not unique, even if one considers representations into $\PSL$ (see \cite{CaLuTi} for examples).
The point is that there are two conjugacy classes of discrete and faithful representations into $\PSL.$ They correspond to the two different choices of orientation of $M$ and are related by complex conjugation.


\subsection{Essential surfaces}

An \emph{(embedded) surface} $S$ in the topologically finite 3--manifold $M = \text{int} (\overline{M})$ will always mean a 2--dimensional PL submanifold of $M$ with the property that its closure $\overline{S}$ in $\overline{M}$ is \emph{properly embedded} in $\overline{M},$ that is, a closed subset of $\overline{M}$ with $\partial \overline{S} = \overline{S} \cap \partial \overline{M}.$ A surface $S$ in $M$ is said to be \emph{essential} if its closure $\overline{S}$ is essential in $\overline{M}$ as described in the following definition:
\begin{defn*} \cite{sh1}
A surface $\overline{S}$ in a compact, irreducible, orientable 3--manifold $\overline{M}$ is said to be essential if it has the following five properties:
    \begin{enumerate}
    \item $\overline{S}$ is bicollared;
    \item the inclusion $\pi_1(\overline{S}_i) \to \pi_1(\overline{M})$ is injective for every
      component $\overline{S}_i$ of $\overline{S}$;
    \item no component of $\overline{S}$ is a 2--sphere;
    \item no component of $\overline{S}$ is boundary parallel;
    \item $\overline{S}$ is nonempty.
    \end{enumerate}
\end{defn*}
The boundary curves of $\overline{S},$ $\overline{S} \cap \partial\overline{M},$ are also called the \emph{boundary curves} of $S.$


\subsection{Ideal points, and valuations}
\label{sec:cs-theory}

In case that $V$ is a 1--dimensional
irreducible variety, there is a unique non--singular projective variety
$\tilde{V}$ which is birationally equivalent to $\overline{J(V)}$.
$\tilde{V}$ is called the \emph{smooth projective completion} of $V$, and
the \emph{ideal points} of $\tilde{V}$ are the points of $\tilde{V}$
corresponding to $\overline{J(V)} - J(V)$ under the birational equivalence.
Moreover, the function fields of $V$ and $\tilde{V}$ are isomorphic.

Let $C \subset \X(M)$ be a curve,
i.e. a 1--dimensional irreducible subvariety, and denote its smooth projective
completion by $\tilde{C}$. We will refer to the ideal points of $\tilde{C}$
also as ideal points of $C$.
The function fields of $C$ and $\tilde{C}$ are isomorphic. Denote them by
$K$. Any ideal point $\xi$ of $\tilde{C}$ determines a (normalised, discrete,
rank 1) valuation $ord_\xi$ of $K$, by
{\small
\begin{equation*}
   ord_\xi (f) =
   \begin{cases}k & \text{if $f$ has a zero of order $k$ at $\xi$} \\
           \infty & \text{if $f=0$} \\
   	       -k & \text{if $f$ has a pole of order $k$ at $\xi$}
   \end{cases}
\end{equation*}
}
Note that $ord_\xi (z)=0$ for all non-zero constant functions $z \in \C$.
In the language of algebraic geometry, the valuation ring
$\{ f \in K \mid ord_\xi (f) \ge 0\}$ of
$ord_\xi$ is the local ring at $\xi$.

Culler and Shalen~\cite{cs, sh1} associate essential surfaces in a 3--manifold $M$ to
\emph{ideal points of curves} in the character variety $\X (M)$. The ingredients are as
follows: 
\begin{enumerate}
\item A curve in $\X (M)$ yields a field $F$ with a discrete valuation at
each ideal point and a representation $\mathcal{P} : \pi_1(M) \to \text{SL}_2(F)$.
\item The group $\text{SL}_2(F)$ acts on the associated Bruhat-Tits tree $\tree_v$ and using $\mathcal{P}$ we can pull
this back to an action of $\pi_1(M)$ on $\tree_v$ that is non--trivial and without inversions.
\item A non--trivial action of $\pi_1(M)$ without inversions on a trees gives essential
surfaces in $M$ via a construction due to Stallings.
\end{enumerate}


\subsection{Logarithmic limit set}
\label{combinatorics:Logarithmic limit set}

Let 
$$\C[X^{\pm}]=\C[X_1^{\pm 1}, \ldots , X_n^{\pm 1}].$$
Given an ideal $I\subset \C[X^{\pm}],$ let $V = V(I)$ be the corresponding subvariety of $(\C -\{ 0\})^n.$ Bergman gives the following three descriptions of a \emph{logarithmic limit set} in \cite{gb}:

\emph{1. The tropical variety:} Define $V_{\val}$ as the set of $n$--tuples
\begin{equation} \label{log lim def2}
   (-v(X_1), \ldots , -v(X_n))
\end{equation}
as $v$ runs over all real--valued valuations on $\C[X^{\pm}]/I$ satisfying $\sum v(X_i)^2 = 1.$

\emph{2. A geometric construction:} Define the support $s(f)$ of an element $f \in \C[X^{\pm}]$ to be the set
of points $\alpha = (\alpha_1,{\ldots} ,\alpha_n) \in \Z^n$ such that
$X^\alpha = X_1^{\alpha_1} \cdots X_n^{\alpha_n}$ occurs with non--zero
coefficient in $f$. Then define $V_{\sph}$ to be the set of $\xi \in
S^{n-1}$ such that for all non--zero $f \in I$, the maximum value of the dot
product $\xi \cdot \alpha$ as $\alpha$ runs over $s(f)$ is assumed at least
twice. Geometrically, this is:
$$
V_{\sph} = \bigcap_{0\neq f \in I} \Sph(\Newt(f)),
$$
where $\Newt(f)$ is the Newton polytope and $\Sph(\Newt(f))$ its spherical dual. The latter consists of all outward pointing unit normal vectors to the support planes of $\Newt(f)$ which meet $\Newt(f)$ in more than
one point.

Bergman shows in \cite{gb} that $V_{\val}=V_{\sph}.$

\emph{3. The logarithmic limit set:} Define $V_{\log}$ as the set of limit points on $S^{n-1}$ of the set
of real $n$--tuples in the interior of $B^n$:
\begin{equation} \label{log lim def1}
  \bigg\{ \frac{ (\log |x_1|, \ldots , \log |x_n|)}
            {\sqrt{1 + \sum (\log |x_i|)^2}}
                      \mid x \in V \bigg\} .
\end{equation}

Bergman shows in \cite{gb} that we have  $V_{\log} \subset V_{\val}=V_{\sph}.$ It follows from work by Bieri and Groves \cite{BG1984} that all sets in fact coincide:

\begin{thm}[Bergman, Bieri--Groves] 
We have $V_{\log} = V_{\val}=V_{\sph}.$
\end{thm}

Let $V_\infty=V_{\log}$ be the logarithmic limit set of the variety $V$. The ideal generated by a set of polynomials $\{ f_i\}_i$ shall be denoted by $I(f_i)_i$, and the variety generated by the ideal $I$ is denoted by $V(I)$ or $V(f_i)_i$, if $I = I(f_i)_i$.
The following fact completely determines the tropical variety of a principal ideal:

\begin{pro}
  Let $f \in \C[X^\pm ]$. If $f \ne 0$, then $V(f)_\infty = \Sph(\Newt(f))$.
\end{pro}

In general, we have the following:

\begin{thm}[Bieri-Groves Polyhedrality] \label{Bergman-Bieri-Groves}
  The logarithmic limit set $V_\infty$ of an algebraic variety $V$ in $(\C -\{
  0\})^n$ is a finite union of rational convex spherical polytopes, all having the same dimension, namely
  \begin{equation*}
    (\dim_\sC V) - 1.
  \end{equation*}
\end{thm}

The following is of interest to applications:

\begin{lem}\cite{tillus_ei}\label{curve finding lemma}
  Let $V$ be an algebraic variety in $(\C -\{ 0\})^m$, and $\xi \in V_\infty$
  be a point with rational coordinate ratios. Then there is a curve $C$, i.e.
  an irreducible subvariety of complex dimension one, in $V$ such that $\xi
  \in C_\infty$.
\end{lem}


\subsection{Eigenvalue variety}
\label{boundary:Eigenvalue variety}

Let $M$ be an orientable, irreducible, compact 3--manifold with non--empty
  boundary consisting of a disjoint union of $h$ tori, denoted $T_1, \ldots, T_h.$

Given a presentation of $\pi_1(M)$ with a finite number, $n$, of generators,
$\gamma_1,{\ldots} ,\gamma_n$, introduce four affine coordinates (representing
matrix entries) for each generator, which are denoted by $g_{ij}$ for
$i=1,{\ldots} ,n$ and $j=1,2,3,4$. View $\R (M)$ as an affine algebraic set in
$\C^{4n}$ defined by an ideal $J$ in $\C[g_{11},{\ldots} ,g_{n4}]$. There are
elements $I_\gamma \in \C[g_{11},{\ldots} ,g_{n4}]$ for each $\gamma \in
\pi_1(M)$ such that $I_\gamma (\rho) = \tr \rho (\gamma)$ for each $\rho \in
\R (M)$.

Identify $\m_i$ and $\l_i$ with generators of $\im(\pi_1(T_i)\to\pi_1(M))$.
Thus, $\m_i$ and $\l_i$ are words in the generators for $\pi_1(M)$.
In the ring
$\C[g_{11},{\ldots},g_{n4}, m_1^{\pm 1},l_1^{\pm 1},{\ldots},m_h^{\pm 1}, l_h^{\pm 1}]$, we 
define the
following polynomial equations:
\begin{align}
\label{ev:cood m}  I_{\m_i}       &= m_i + m_i^{-1}, \\
\label{ev:cood l}  I_{\l_i}       &= l_i + l_i^{-1}, \\
\label{ev:cood ml} I_{\m_i\l_i} &= m_il_i + m_i^{-1}l_i^{-1},
\end{align}
for $i=1,{\ldots} ,h$. Let $\R_E(M)$ be the variety in $\C^{4n} \times
(\C-\{0\})^{2h}$ defined by $J$ together with the above equations. For each
$\rho \in \R(M)$, the equations (\ref{ev:cood m}--\ref{ev:cood ml}) have a
solution since commuting elements of $\SL$ always have a common invariant
subspace.  The natural projection $p_1: \R_E(M) \to \R(M)$ is therefore onto,
and $p_1$ is a dominating map.

If $(a_1,{\ldots} ,a_{4n}) \in \R (M)$, then there is an action of $\Z_2^{h}$
on the resulting points
\begin{equation*}
  (a_1,{\ldots} ,a_{4n}, m_1, l_1,{\ldots} ,m_h, l_h) \in \R_E(M)
\end{equation*}
by inverting both entries of a tuple $(m_i, l_i)$ to $(m_i^{-1}, l_i^{-1})$.
The group $\Z_2^{h}$ acts transitively on the fibres of $p_1$. The map $p_1$
is therefore finite--to--one with degree $\le 2^h$. The maximal degree is in
particular achieved when the interior of $M$ admits a complete hyperbolic
structure of finite volume, since Theorem~\ref{thurston thm} implies that
there are points in $\X_0(M) - \cup_{i=1}^h \{ I^2_{\m_i} = 4\}$.

The \emph{eigenvalue variety} $\Ei (M)$ is the closure of the image of
$\R_E(M)$ under projection onto the coordinates $(m_1,{\ldots} ,l_h)$. It is
therefore defined by an ideal of the ring $\C [m_1^{\pm 1}, l_1^{\pm 1},
\ldots , m_h^{\pm 1}, l_h^{\pm 1}]$ in $(\C - \{ 0\})^{2h}$.

Note that this construction factors through a variety $\X_E(M)$, which is the
character variety with its coordinate ring appropriately extended, since
coordinates of $\X (M)$ can be chosen such that $I_\gamma$ (as a function on
$\X (M)$) is an element of the coordinate ring of $\X (M)$ for each $\gamma
\in \pi_1(M)$.  Let $\te_E : \R_E(M) \to \X_E(M)$ be the natural quotient map,
which is equal to $\te: \R (M) \to \X (M)$ on the first $4n$ coordinates, and
equal to the identity on the remaining $2h$ coordinates.

There also is a restriction map $r: \X (M) \to \X(T_1) \times {\ldots} \times
\X(T_h)$, which arises from the inclusion homomorphisms $\pi_1(T_i) \to
\pi_1(M)$. Therefore denote the map $\X_E(M) \to \Ei (M)$ by $r_E$.
Denote the closure of the image of $r$ by $\X_\partial (M)$. There is the
following commuting diagram of dominating maps:
\begin{center}
$ 
\begin{CD}
      \R_E(M)    @>\te_E >>     \X_E(M)  @>r_E >> \Ei (M)   \\ 
       @Vp_1VV          @Vp_2VV         @Vp_3VV   \\
     \R(M)  @>\te >>     \X(M) @>r >> \X_\partial(M)
\end{CD}
$
\end{center}
Note that the maps $p_1,p_2,t$ and $t_E$ have the property that every point in
the closure of the image has a preimage, and that the maps $p_1, p_2$ and
$p_3$ are all finite--to--one of the same degree.

Recall the construction by Culler and Shalen. Starting with a curve $C \subset
\X (M)$ and an irreducible component $R_C$ in $\R (M)$ such that $\te (R_C) =
C$, one obtains the tautological representation $\P : \pi_1(M) \to SL_2(\C
(R_C))$. Let $R'_C$ be an irreducible component of $\R_E(M)$ with the property
that $p_1(R'_C) = R_C$. Then $\C (R'_C)$ is a finitely generated extension of
$\C(R_C)$, and $\P$ may be thought of as a representation $\P : \pi_1(M) \to
SL_2(\C (R'_C))$.

If $\overline{r_E\te_E(R'_C)}$ contains a curve $E$, then $\C (R'_C)$ is a
finitely generated extension of $\C (E)$, and essential surfaces can be
associated to each ideal point of $E$ using $\P$. Since the eigenvalue of at
least one peripheral element blows up at an ideal point of $E$, the associated
surfaces necessarily have non--empty boundary (see \cite{ccgls}).

Thus, if there is a closed essential surface associated to an ideal point of
$C$, then either $\overline{r_E\te_E(R'_C)}$ is 0--dimensional, or there is a
neighbourhood $U$ of an ideal point $\xi$ of $R'_C$ such that there are
(finite) points in $\overline{r_E\te_E(U)} - r_E\te_E(U)$. The later are
called \emph{holes in the eigenvalue variety}, examples of which are given in
\cite{tillus_kino}. 

The action of $\Z_2^{h}$
on the eigenvalue variety $\Ei (M)$ induces an action of $\Z_2^{h}$
on its
logarithmic limit set $\Ei_\infty (M)\subset S^{2h -1}$.  
The action is generated by taking $\xi = (x_1,{\ldots} ,x_{2i-1}, x_{2i}, {\ldots} ,x_{2h})
\in\Ei_\infty(M)$ to $(x_1,{\ldots} ,-x_{2i-1}, -x_{2i}, {\ldots} ,x_{2h})
\in\Ei_\infty(M)$ for $i=1,{\ldots} ,h$. Factoring by these symmetries, a
quotient of $\Ei_\infty (M)$ is obtained in the space $\RR
P^{2h-1}/\Z_2^{h-1} \cong S^{2h-1}$. The quotient map extends to a map $\varphi :
S^{2h -1} \to S^{2h -1}$ of spheres, which has degree $2^h$.  Let
\begin{equation*}
  P  = \begin{pmatrix} 0 & 1 \\ -1 & 0 \end{pmatrix}
\end{equation*}
and let $P_h$ be the block diagonal matrix with $h$ copies of $P$ on its
diagonal. Then $P_h$ is orthogonal, and its restriction to $S^{2h -1}$ is a
map of degree one.
The main result of \cite{tillus_ei} is:

\begin{lem}[Boundary slopes] \label{boundary slopes lemma}
  Let $M$ be an orientable, irreducible, compact 3--manifold with non--empty
  boundary consisting of a disjoint union of $h$ tori.  If $\xi \in \Ei_\infty
  (M)$ is a point with rational coordinate ratios, then there is an essential
  surface with boundary in $M$ whose projectivised boundary curve coordinate
  is $\varphi (P_h \xi)$. Moreover, the image $\varphi(P_h \Ei_\infty(M))$ is a closed subset of $\BC (M)$.
\end{lem}


\subsection{$\mathbf{\PSL}$--eigenvalue variety} 
\label{boundary:PSL-Eigenvalue variety}

Analogous to $\Ei(M)$, one can define a $\PSL$--eigenvalue variety $\PEi(M)$,
since the function $I^2_\gamma : \PX (M) \to \C$ defined by $I^2_\gamma
(\prho) = (\tr\prho(\gamma))^2$ is regular (i.e. polynomial) for all
$\gamma\in\pi_1(M)$ (see \cite{bozh}). Thus, a variety $\PR_E(M)$ is
constructed using the relations
\begin{align*}
  I^2_{\m_i} &= M_i + 2 + M_i^{-1}, \\
  I^2_{\l_i} &= L_i + 2+L_i^{-1}, \\
  I^2_{\m_i\l_i} &= M_iL_i + 2+M_i^{-1}L_i^{-1},
\end{align*}
for $i=1,{\ldots} ,h$.

Consider the representation of $\Z \oplus \Z$ into $\PSL$ generated by the
images of
\begin{equation} \label{four group}
    \begin{pmatrix} i&0 \\ 0&-i \end{pmatrix}
    \quad \text{and} \quad
    \begin{pmatrix} 0&1 \\ -1&0 \end{pmatrix}.
\end{equation}
In $\PSL$, the image of this representation is isomorphic to $\Z_2 \oplus
\Z_2$, but the image of any lift to $\SL$ is isomorphic to the quaternion
group $Q_8$ (in Quaternion group notation).  If $\prho \in \PR (M)$ restricts
to such an irreducible abelian representation on a boundary torus $T_i$, then
the equations $M_i + 2 + M_i^{-1}=0$, $L_i + 2+L_i^{-1}=0$, $M_iL_i +
2+M_i^{-1}L_i^{-1}=0$ have no solution. In particular, the projection
$\overline{p}: \PR_E(M) \to \PR(M)$ may not be onto.

The closure of the image of $\PR_E(M)$ onto the coordinates
$(M_1,L_1,{\ldots},M_h ,L_h)$ is denoted by $\PEi (M)$ and called the
\emph{$\PSL$-eigenvalue variety}. The relationship between ideal points of
$\PEi (M)$ and strongly detected boundary curves is the same as for the
$\SL$--version, since the proof of Lemma \ref{boundary slopes lemma} applies
without change:

\begin{lem} \label{psl boundary slopes lemma}
  Let $M$ be an orientable, irreducible, compact 3--manifold with non--empty
  boundary consisting of a disjoint union of $h$ tori.  If $\xi \in
  \PEi_\infty (M)$ is a point with rational coordinate ratios, then there is
  an essential surface with boundary in $M$ whose projectivised boundary curve
  coordinate is $\varphi (P_h \xi)$.
\end{lem}


\section{The deformation variety and normal surfacess}
\label{sec:tillus_defo}

This section summarises the approach to Culler-Morgan-Shalen theory using ideal triangulations, spun-normal surfaces and tropical geometry given in \cite{defo}.
Throughout this section, $M$ is the interior of an orientable, compact, connected 3--manifold with non-empty boundary consisting of a pairwise disjoint union of tori.


\subsection{Ideal triangulations and shape parameters}
\label{combinatorics:Deformation variety}

An ideal (topological) triangulation $\tri$ of $M$ is a triple $\tri = (\widetilde{\Delta}, \Phi, h),$ where 
\begin{enumerate}
\item $\widetilde{\Delta} = \bigcup_{k=1}^{n} \widetilde{\Delta}_k$ is a pairwise disjoint union of oriented standard Euclidean 3--simplices, 
\item $\Phi$ is a collection of orientation reversing Euclidean isometries between the 2--simplices in $\widetilde{\Delta};$ termed \emph{face pairings},
\item $h\co (\widetilde{\Delta} \setminus \widetilde{\Delta}^{(0)} )/ \Phi \to M$ is a homeomorphism, where the domain is given the quotient topology. 
\end{enumerate}
We write $p \co (\widetilde{\Delta} \setminus \widetilde{\Delta}^{(0)} ) \to M$ for the composition of the natural quotient map with $h,$ and note that $M$ is determined by $\tri.$

The associated \emph{pseudo-manifold} (or \emph{end-compactification} of $M$) is $P = \widetilde{\Delta} / \Phi$. We denote the quotient map with the same letter, $p\co \widetilde{\Delta} \to P$.
Denote $\Sigma^k$ the set of all (possibly singular) $k$--simplices in $P.$ The degree $\deg(\sigma^k)$ of a $k$--simplex in $\Sigma^k$ is the number of $k$--simplices in its preimage in $\widetilde{\Delta}.$

An \emph{ideal $k$--simplex} is a $k$--simplex with its vertices removed. The vertices of the $k$--simplex are termed the \emph{ideal vertices} of the ideal $k$--simplex. Similar for singular simplices. The standard terminology of (ideal) edges, (ideal) faces and (ideal) tetrahedra will be used for the (ideal) simplices of dimensions one, two, and three respectively.

Let $\Delta^3$ be the standard 3--simplex with a chosen orientation. Suppose the edges $\Delta^3$ are labeled by $z,$ $z'$ and $z''$ so that opposite edges have the same labeling and all labels are used. Then the cyclic order of $z,$ $z'$ and $z''$ viewed from each vertex depends only on the orientation of the 3--simplex. It follows that, up to orientation preserving symmetries, there are two possible labelings, and we fix one of these labelings. 

Suppose $\Sigma^3 = \{ \sigma_1, \ldots, \sigma_n\}.$ Since $M$ is orientable, the 3--simplices in $\Sigma^3$ may be oriented coherently. For each $\sigma_i \in \Sigma^3,$ fix an orientation preserving simplicial map $f_i\co \Delta^3 \to \sigma_i.$ An Euler characteristic argument shows that $|\Sigma^3| = |\Sigma^1|.$ Let $\Sigma^1 = \{ e_1, \ldots, e_n\},$ and let $a^{(k)}_{ij}$ be the number of edges in $f_i^{-1}(e_j)$ which have label $z^{(k)}.$

For each $i \in \{1,\ldots, n\},$ define
\begin{equation}\label{eq:para}
    p_i   = z_i (1 - z''_i) - 1,\quad
    p'_i  = z'_i (1 - z_i) - 1,\quad
    p''_i = z''_i (1 - z'_i) - 1,
\end{equation}
and for each $j \in \{1,\ldots, n\},$ let
\begin{equation}\label{eq:glue}
    g_j = \prod_{i=1}^n z_i^{a_{ij}} {(z'_i)}^{a'_{ij}} {(z''_i)}^{a''_{ij}}
    -1. 
\end{equation}
Setting $p_i=p'_i=p''_i = 0$ gives the \emph{parameter equations}, and setting $g_j=0$ gives the \emph{hyperbolic gluing equations}. For a discussion and geometric interpretation of these equations, see \cite{t, nz}. The parameter equations imply that $z_i^{(k)} \neq 0,1.$

\begin{defn}
Suppose $\tri$ is an ideal triangulation of $M.$
The \emph{deformation variety $\D (\tri)$} is the variety in $(\C-\{0 \})^{3n}$ defined by the parameter equations and the hyperbolic gluing equations.
\end{defn}


\subsection{Developing maps, characters and eigenvalues}
\label{sec: developing}
\label{sec: Holonomies and eigenvalue variety}

The following facts can be found in \cite[\S5]{y}; see also \cite{ST} for a more detailed discussion.

Given $Z \in \D (\tri),$ each ideal tetrahedron in $M$ has edge labels which can be lifted equivariantly to $\widetilde{M}.$ Following \cite{y}, we define a continuous map $\dev_Z\co \widetilde{M} \to \H^3,$ which maps every ideal tetrahedron $\sigma$ in $\widetilde{M}$ to an ideal hyperbolic 3--simplex $\Delta(\sigma),$ such that the labels carried forward to the edges of $\Delta(\sigma)$ correspond to the shape parameters of $\Delta(\sigma)$ determined by its hyperbolic structure; see \cite{t} for the geometry of hyperbolic ideal tetrahedra. If each shape parameter has positive imaginary part, then the associated
hyperbolic ideal tetrahedron is positively oriented and $\dev_Z$ is a developing map for a (possibly incomplete) hyperbolic structure on $M.$

For each $Z \in \D (\tri),$ the Yoshida map $\dev_Z$ can be used to define a representation $\rho_Z \co \pi_1(M) \to \PSL$ as follows (see \cite{y}). A representation into $\PSL$ is an action on $\H^3,$ and this is the unique representation which makes $\dev_Z$ $\pi_1(M)$--equivariant: $\dev_Z(\gamma \cdot x) = \rho_Z(\gamma) \dev_Z(x)$ for all $x \in \widetilde{M},$ $\gamma \in \pi_1(M).$  Thus, $\rho_Z$ is well--defined up to conjugation, since it only depends upon the choice of the embedding of the initial tetrahedron $\sigma.$ This yields a well--defined map $\chi_{\tri}\co \D (\tri) \to \PX (N)$ from the deformation variety to the $\PSL$--character variety. It is implicit in \cite{nr} that $\chi_{\tri}$ is algebraic; see \cite{ac} for details using a faithful representation of $\PSL \to SL(3, \C).$ Note that the image of each peripheral subgroup under $\rho_Z$ has at least one fixed point on the sphere at infinity. We summarise this discussion in the following lemma:

\begin{lem}
Let $M$ be the interior of a compact, connected, orientable 3--manifold with non-empty boundary consisting of a pairwise disjoint union of tori, and let $\tri$ be an ideal triangulation of $M.$ For each $Z \in \D(\tri),$ there exists a representation $\prho_Z \co \pi_1(M) \to \PSL$ and a $\prho_Z$--equivariant, continuous map $D_Z \co \widehat{P} \to \overline{\H}^3$ with the property that for each ideal tetrahedron $\sigma \subset \widetilde{M},$ the image $D_Z(\sigma)$ is a hyperbolic ideal tetrahedron with edge invariants determined by $Z.$ Moreover, $\prho_Z$ is well-defined up to conjugacy and the well-defined map 
$$\chi_\tri \co \D(\tri) \to \PX (M)$$
is algebraic.
\end{lem}

For cusped hyperbolic 3--manifolds, there are a number of special facts that are well-known to follow from work in \cite{t, nz}. The following highlights the fact that the deformation variety facilitates the study of the Dehn surgery components in the character variety, i.e.\thinspace the components containing the characters of discrete and faithful representations.

\begin{thm}
Let $M$ be an orientable, connected, cusped hyperbolic 3--manifold. Let $\tri$ be an ideal triangulation of $M$ with the property that all edges are essential. Then there exists $Z \in \D(\tri),$ such that $\prho_Z\co \pi_1(M)\to \PSL$ is a discrete and faithful representation. Moreover, the whole {Dehn surgery component} $\PX_0(M)$ containing the character of $\prho_Z$ is in the image of $\chi_\tri.$
\end{thm}

Denote $C$ a compact core of $M;$ this is obtained by removing a small open regular neighbourhood from each vertex in $P,$ such that $C$ inherits a decomposition into truncated tetrahedra. Each boundary torus $T_i,$ $i=1,{\ldots} ,h,$ of $C$ inherits a triangulation $\tri_i$ induced by $\tri.$ Let $\gamma$ be a closed simplicial path on $T_i.$ In
\cite{nz}, the \emph{holonomy} $\mu(\gamma)$ is defined as $(-1)^{|\gamma|}$ times the product of the moduli of the triangle vertices touching $\gamma$ on the right, where $|\gamma|$ is the number of 1--simplices of $\gamma,$ and the moduli asise from the corresponding edge labels.

At $Z \in \D (\tri),$ evaluating $\mu(\gamma)$ gives a complex number $\mu_Z(\gamma) \in \C \setminus\{0\}.$ It is stated in \cite{t, nz}, that $\mu_Z(\gamma)$ is the square of an eigenvalue; one has: 
$$(\tr \prho_Z (\gamma))^2 = \mu_Z(\gamma) + 2 + \mu_Z(\gamma)^{-1}.$$ 
This can be seen by putting a common fixed point of $\prho_Z (\pi_1(T_i))$ at infinity in the upper--half space model, and writing
$\prho_Z(\gamma)$ as a product of M\"obius transformations, each of
which fixes an edge with one endpoint at infinity and takes one face
of a tetrahedron to another. 

Choose a basis $\{\m_i, \l_i\}$ of $\pi_1(T_i)\cong H_1(T_i)$ for each boundary torus $T_i.$  Since $\mu_Z\co \pi_1(T_i) \to \C\setminus \{0\}$ is a homomorphism for each $i=1,{\ldots} ,h,$ there is a well--defined rational map:
\begin{equation*}
    \pee \co \D (\tri) \to (\C-\{0\})^{2h} \qquad
    \pee(Z) = (\mu_Z (\m_1),{\ldots} ,\mu_Z (\l_h )).
\end{equation*}
The closure of its image is contained in the \emph{$\PSL$--eigenvalue variety} $\PEi (M)$ of \cite{tillus_ei}.


\subsection{Ideal points}
\label{sec:Ideal points and normal surfaces}

Since $\D (\tri)$ is a variety in $(\C \setminus\{ 0\})^{3n},$ Bergman's construction in \cite{gb} can be used to define its set of ideal points. Let
\begin{align*}
  Z &= (z_1, z'_1, z''_1, z_2, {\ldots} ,z''_n),\\
  \log |Z| &= (\log |z_1|, \ldots , \log |z''_n|),\\
  u(Z) &= \frac{1}{\sqrt{1 + (\log |z_1|)^2 + {\ldots} + (\log |z''_n|)^2}}.
\end{align*}
The map $\D (\tri) \to B^{3n}$ defined by $Z \to u(Z)\log |Z|$ is continuous, and the \emph{logarithmic limit set} $\D_\infty (\tri)$ is the set of limit points on $S^{3n-1}$ of its image. Thus, for each 
$\xi \in \D_{\infty}(M)$ there is a sequence $\{ Z_i\}$ in $\D (M)$ such that
\begin{equation*}
   \lim_{i\to \infty} u(Z_i) \log |Z_i| = \xi.
\end{equation*}
The sequence $\{ Z_i\}$ is said to \emph{converge} to $\xi,$ written $Z_{i} \to \xi,$
and $\xi$ is called an ideal
point of $\D (\tri).$ Whenever an
edge invariant of a tetrahedron converges to one, the other two edge
invariants ``blow up''. Thus, an ideal point of $\D (\tri)$ is approached
if and only if a tetrahedron degenerates.

Since the Riemann sphere is compact, there is a subsequence, also denoted by $\{ Z_i\},$ with the property that each shape parameter converges in $\C \cup \{ \infty \}.$ In this case $\{ Z_i\}$ is said to \emph{strongly converge} to $\xi.$ If $\xi$ has rational coordinate ratios, a strongly convergent sequence may be chosen on a curve in $\D (\tri)$ according to Lemma~\ref{curve finding lemma}.


\subsection{Normal surface theory}

A \emph{normal surface} in a (possibly ideal) triangulation
$\tri$ is a properly embedded surface which intersects each
tetrahedron of $\tri$ in a pairwise disjoint collection of
\emph{triangles} and \emph{quadrilaterals}, as shown in Figure~\ref{normaldiscs}. These triangles and quadrilaterals are called
\emph{normal discs}. In an ideal triangulation of a non-compact
3--manifold, a normal surface may contain infinitely many triangles;
such a surface is called \emph{spun-normal} \cite{tillmann08-finite}.
A normal surface may be disconnected or empty.

\begin{figure}[h!]  
\centering
\includegraphics[scale=1]{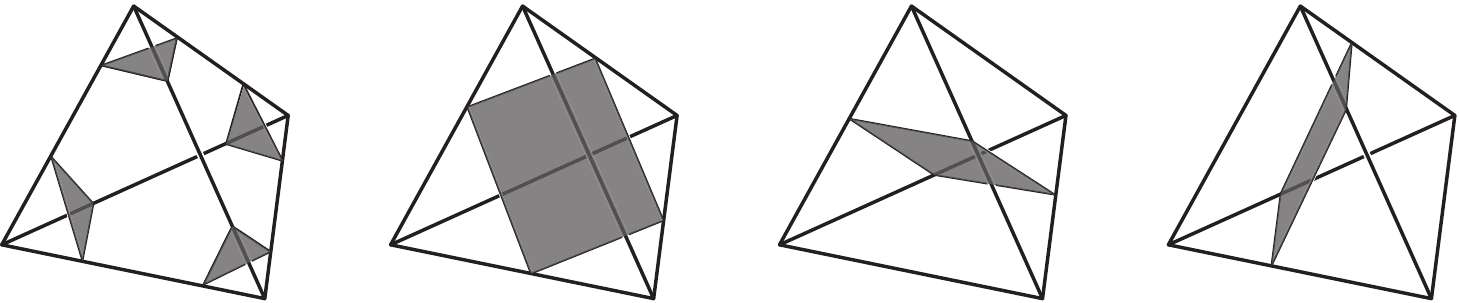}
\caption{The seven types of normal disc in a tetrahedron.} 
\label{normaldiscs}
\end{figure}

We now describe an algebraic approach to normal surfaces.
The key observation is that each normal surface contains
finitely many quadrilateral discs, and is uniquely determined
(up to normal isotopy) by these quadrilateral discs. Here a \emph{normal
isotopy} of $M$ is an isotopy that keeps all simplices of all dimensions
fixed. Let $\square$ denote the set of all normal isotopy classes
of normal quadrilateral discs in $\tri$. A normal quadrilateral disc is up to normal
isotopy uniquely determined by a pair of opposite edges of a tetrahedron.
Hence $|\square| = 3n$ where $n$ is the number of tetrahedra in $\tri$.
These normal isotopy classes are called quadrilateral \emph{types}.

We identify $\RR^\square$ with $\RR^{3n}.$
Given a normal surface $S,$ let $x(S) \in \RR^\square = \RR^{3n}$
denote the integer vector for which each coordinate $x(S)(q)$
counts the number of quadrilateral discs in $S$ of type $q \in \square$.
This \emph{normal $Q$--coordinate} $x(S)$ satisfies the
following two algebraic conditions.

First, $x(S)$ is admissible.
A vector $x \in \RR^\square$ is \emph{admissible} if
$x \ge 0$, and for each tetrahedron $x$ is non-zero
on at most one of its three quadrilateral types. 
This reflects the fact that an embedded surface cannot contain
two different types of quadrilateral in the same tetrahedron.

Second, $x(S)$ satisfies a linear equation for each interior
edge in $M,$ termed a \emph{$Q$--matching equation}.
Intuitively, these equations arise from the fact that as one
circumnavigates the earth, one crosses the equator from north to south
as often as one crosses it from south to north.
We now give the precise form of these equations.
To simplify the discussion,
we assume that $M$ is oriented and all tetrahedra are given
the induced orientation; see \cite[Section~2.9]{tillmann08-finite} for
details. These are easiest to describe with respect to a model for the neighbourhood of $e$ in $M,$ called its \emph{abstract neighbourhood}.

\begin{figure}[h]
    \centering
    \subfigure[The abstract neighbourhood $B(e)$]{%
        \label{fig:matchingquadbdry}%
        \includegraphics[scale=1.1]{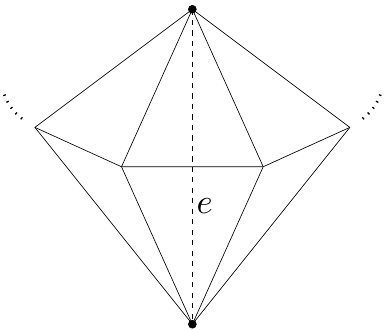}}
    \qquad
    \subfigure[Positive slope $+1$]{%
        \label{fig:matchingquadpos}%
        \includegraphics[scale=1.1]{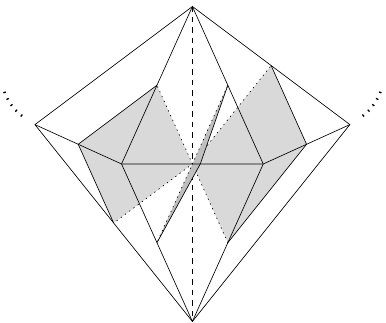}}
    \qquad
    \subfigure[Negative slope $-1$]{%
        \label{fig:matchingquadneg}%
        \includegraphics[scale=1.1]{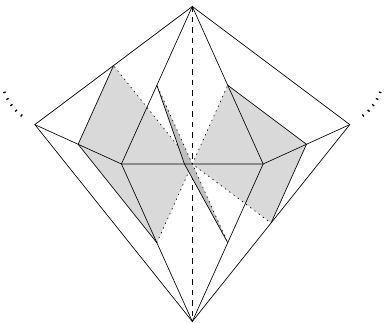}}
    \caption{Slopes of quadrilaterals}
    \label{fig:slopes}
\end{figure}

The \emph{degree of an edge} $e$ in $P,$ $\deg(e),$ is the number of 1--simplices in $\widetilde{\Delta}$ which map to $e.$ Given the edge $e$ in $P,$ there is an associated \emph{abstract neighbourhood $B(e)$} of $e.$ This is a ball triangulated by $\deg (e)$ 3--simplices, having a unique interior edge $\widetilde{e},$ and there is a well-defined simplicial quotient map $p_{e}\co B(e)\to P$ taking $\widetilde{e}$ to $e.$ This abstract neighbourhood is obtained as follows.

If $e$ has at most one pre-image in each 3--simplex in $\widetilde{\Delta},$ then $B(e)$ is obtained as the quotient of the collection $\widetilde{\Delta}_{e}$ of all 3--simplices in $\widetilde{\Delta}$ containing a pre-image of $e$ by the set $\Phi_{e}$ of all face pairings in $\Phi$ between faces containing a pre-image of $e.$ There is an obvious quotient map $b_{e}\co B(e)\to P$ which takes into account the remaining identifications on the boundary of $B(e).$

If $e$ has more than one pre-image in some 3--simplex, then multiple copies of this simplex are taken, one for each pre-image. The construction is modified accordingly, so that $B(e)$ again has a unique interior edge and there is a well defined quotient map $b_{e}\co B(e)\to P.$ Complete details can be found in \cite{tillus_normal}, Section 2.3.

Let $\sigma$ be a
tetrahedron in $B(e)$. The boundary square of a normal
quadrilateral of type $q$ in $\sigma$ meets the equator of $\partial
B(e)$ if and only it has a vertex on $e$. In this case, it has a slope
$\pm1$ of a well--defined sign on $\partial B(e)$ which is independent of the
orientation of $e$. Refer to Figures~\ref{fig:matchingquadpos} and
\ref{fig:matchingquadneg}, which show quadrilaterals with
\emph{positive} and \emph{negative slopes} respectively.

Given a quadrilateral type $q$ and an edge $e,$ there is a
\emph{total weight} $\wt_e(q)$ of $q$ at $e,$ which records the sum of
all slopes of $q$ at $e$ (we sum because $q$ might meet $e$ more than
once, if $e$ appears as multiple edges of the same tetrahedron).
If $q$ has no corner on $e,$ then we set
$\wt_e(q)=0.$ Given edge $e$ in $M,$ the $Q$--matching equation of $e$
is then defined by $0 = \sum_{q\in \square}\; \wt_e(q)\;x(q)$.

\begin{thm}\label{thm:admissible integer solution gives normal}
For each $x\in \RR^\square$ with the properties that $x$ has integral coordinates, $x$ is admissible and
$x$ satisfies the $Q$--matching equations, there is a (possibly
non-compact) normal surface $S$ such that $x = x(S).$ Moreover, $S$ is
unique up to normal isotopy and adding or removing vertex linking surfaces,
i.e., normal surfaces consisting entirely of normal triangles.
\end{thm}

This is related to Hauptsatz 2 of \cite{haken61-knot}. For a proof of Theorem \ref{thm:admissible integer solution gives normal}, see
\cite[Theorem~2.1]{kang05-spun} or \cite[Theorem~2.4]{tillmann08-finite}.
The set of all $x\in \RR^{\square}$ with the property that
(i)~$x\ge 0$ and (ii)~$x$ satisfies the $Q$--matching equations is denoted
$\mathcal{Q}(\tri).$ This
naturally is a polyhedral cone, but
the set of all admissible $x\in \RR^{\square}$
typically meets $\mathcal{Q}(\tri)$ in a non-convex set $\mathcal{F}(\tri).$

The \emph{projectivised solution space} $\mathcal{PQ}(\tri)$ and the \emph{projectivised admissible solution space} $\N(\tri)$, are obtained as the quotient spaces under the multiplication action of the positive real numbers. These are usually identified via radial projection with compact subsets of either the standard simplex or a sphere centred at the origin. For the purpose of this paper, we view $\N(\tri)$ as a compact subset of a sphere centred at the origin.


\subsection{The relationship between gluing and matching equations}
\label{subsec:relationship between gluing and matching}

The deformation variety $\D (\tri)$ is not defined by a principal ideal,
hence its logarithmic limit set $\D_\infty(\tri)$ is in general not
directly determined by its defining equations (see \cite{tillus_ei} for
details).  However, it is contained in the intersection of the spherical duals of the Newton polytopes of its defining equations. The Newton polytope of a polynomial is the convex hull of its exponent vectors. The spherical dual of a polytope is the set of unit normal vectors to its supporting hyperplanes. For a polynomial $p,$ we write $\Sph(p)$ for the spherical dual of the Newton polytope of $p.$ See \cite{gb, tillus_ei} for details. From \eqref{eq:para} and \eqref{eq:glue}, we have
\begin{equation} \label{comb:tent log lim}
      \D_\infty(\tri) \subseteq \D_{\text{pre-}\infty}(\tri) = 
\overset{n}{\underset{i=1}{\bigcap}}
      \big( \Sph(g_i) \cap \Sph(p_i) \cap \Sph(p'_i) \cap \Sph(p''_i) 
\big).
\end{equation}
The set $\D_{\text{pre-}\infty} (\tri)$ is termed a \emph{tropical pre-variety}.

\begin{pro}[\cite{defo}, Proposition~3.1] \label{comb:homeo}
Let $M$ be the interior of an orientable, connected, compact 3--manifold with non-empty boundary consisting of tori, and $\tri$ be an ideal triangulation of $M.$ The set $\D_{\text{pre-}\infty}(\tri)$ is homeomorphic with the projective admissible solution space $\N(\tri)$ of normal surface theory. In particular, $\D_\infty (\tri)$ is homeomorphic with a closed subset of $\N (\tri).$
\end{pro}

The proof of the above result in \cite{defo} gives a canonical homeomorphism $N \co \D_\infty (\tri) \to \N (\tri)$ as follows. 
At its heart is the following simple relationship between gluing equations and matching equations.
Recall the description of the hyperbolic gluing equation (\ref{eq:glue}) of
$e_j$:
\begin{equation*}
    1 = \prod_{i=1}^n z_i^{a_{ij}} {(z'_i)}^{a'_{ij}} 
{(z''_i)}^{a''_{ij}}.
\end{equation*}
Denoting $q_i^{(j)}$ the quadrilateral type separating the pair of edges with label $z_i^{(j)}$ in tetrahedron $\sigma_i$, it follows that the 
the Q--matching equation of $e_j$ is:
\begin{equation*}
    0 = \sum_{i=1}^n (a''_{ij} - a'_{ij}) q_i
                     + (a_{ij} - a''_{ij}) q'_i
+ (a'_{ij} - a_{ij}) q''_i
\end{equation*}
(see \cite{tillmann08-finite}, Section 2.9). Let \begin{equation}
    C_1 = \begin{pmatrix}
       0 & 1 & -1 \\
       -1 & 0 & 1 \\
       1 & -1 & 0
       \end{pmatrix},
\end{equation}
and let $C_n$ be the $(3n \times 3n)$ block diagonal matrix with $n$ copies of
$C_1$ on its diagonal. Hence letting $A$ be the exponent matrix of the hyperbolic gluing equations and $B$ the coefficient matrix of the 
Q--matching equations, we have $AC_n = B.$ The desired natural homeomorphism is given by letting $N(\xi)$ be the unique normal $Q$--coordinate such that $\xi = C_n^T N(\xi)$ for a given $\xi \in \D_\infty (\tri).$ 

 If $\xi$ has rational coordinate ratios, one can associate a unique spun-normal surface $S(\xi)$ to it as follows. By assumption, there is $r>0$ such that $rN(\xi)$ is an integer solution, and hence corresponds to a unique spun-normal surface without vertex linking components according to Theorem~\ref{thm:admissible integer solution gives normal}. One then requires $r$ to be minimal with respect to the condition that the surface is 2--sided.


\subsection{Essential surfaces and boundary curves}

Morgan and Shalen \cite{ms1} compactified the character variety of a 3--manifold by identifying ideal points with certain actions of $\pi_1(M)$ on $\RR$--trees (given as points in an infinite dimensional space), and dual to these are codimension--one measured laminations in $M.$ The work in~\cite{defo} takes a different but related approach by compactifying the deformation variety with certain transversely measured singular codimension--one foliations (a finite union of convex rational polytopes), and dual to these are (possibly trivial) actions on $\RR$--trees. The relationship between these compactifications is investigated in Section 6 of~\cite{defo}, with focus on the original consideration of ideal points of curves and surfaces dual to Bass--Serre trees due to Culler and Shalen \cite{cs}. A surface which can be reduced to an essential surface that is detected by the character variety in the sense of \cite{tillus_mut} will be called \emph{weakly dual to an ideal point of a curve in the character variety}. The following result combines statements of Theorem~1.2 and Proposition~4.3 in \cite{defo}:

\begin{thm}\cite{defo}\label{comb:essential prop}
Let $M$ be the interior of a compact, connected, orientable, irreducible 3--manifold with non-empty boundary consisting of a disjoint union of tori, and let $\tri$ be an ideal triangulation of $M.$  Let $\xi \in \D_\infty(\tri)$ be a point with rational coordinate ratios, and assume that the spun-normal surface $S(\xi)$ is not closed. 

Then $S(\xi)$ is non-trivial and (weakly) dual to an ideal point of a curve in the character variety of $M$. Moreover, every essential surface detected by this ideal point has the same boundary curves as $S(\xi).$ In particular, the boundary curves of $S(\xi)$ are strongly detected.
\end{thm}

It is an open question whether there is an ideal point of the deformation variety that is associated with a closed essential surface but with the property that any associated sequence of characters does not approach an ideal point of the character variety.

For a more general statement pertaining to all points in $\D_\infty(\tri)$, and thus completing the description of its relationship with the Morgan--Shalen compactification, see Theorem 1.2 in \cite{defo}. 


\section{The Whitehead link}
\label{sec:whl}

The right--handed Whitehead link is shown  in the figure on the right. Throughout this paper,  $\whl = \whl^+$ 
denotes its complement in $S^3.$ This is a hyperbolic 3--manifold  isometric with  the same  orientation\begin{wrapfigure}{r}{44mm}
\begin{center}
  \includegraphics[width=3.8cm]{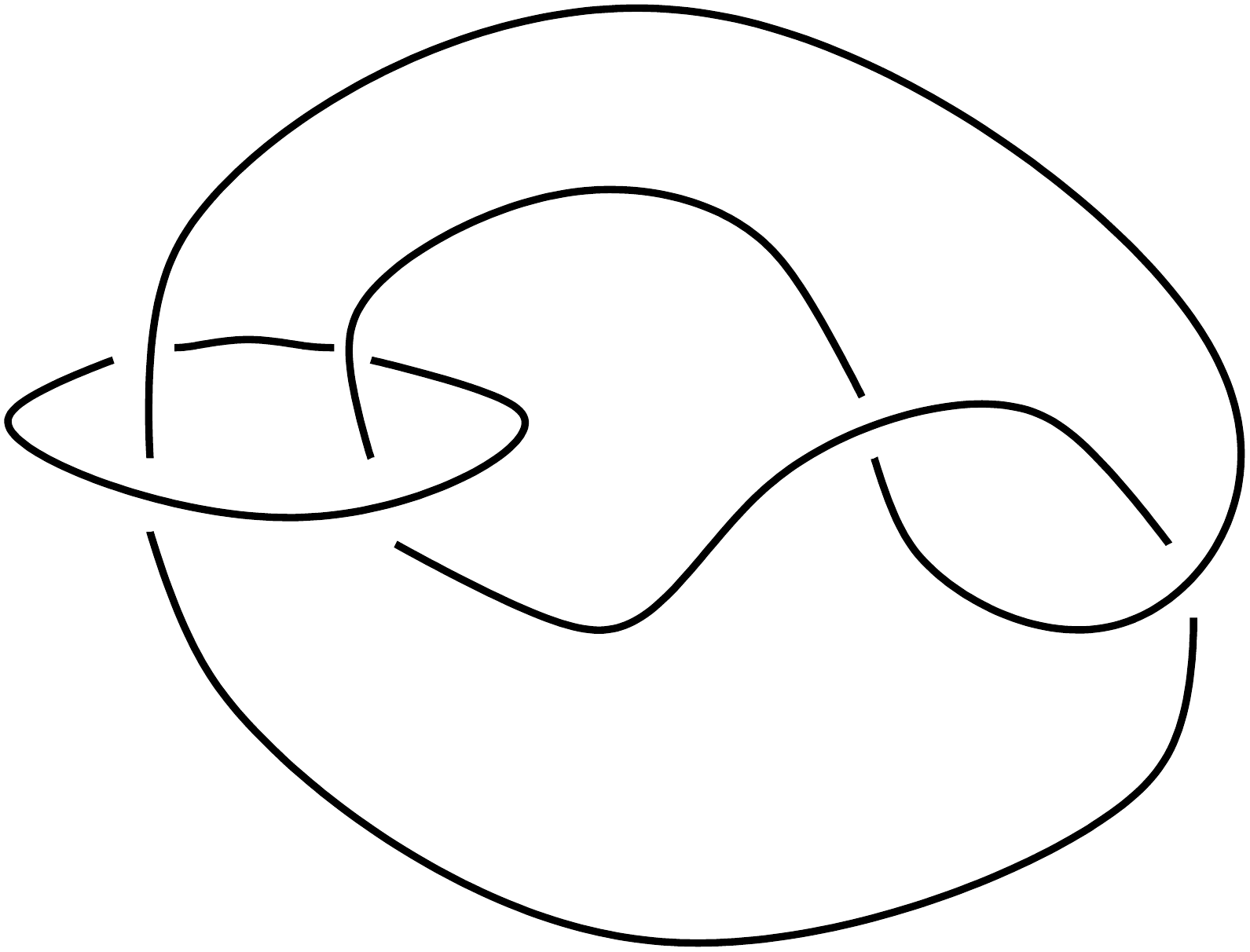}
\end{center}
\end{wrapfigure}
 to the manifold $m129$ in \tt{SnapPy }\rm's census. The complement of a left--handed Whitehead  link is oppositely oriented, and denoted by $\whl^-.$ 

The standard triangulation $\tri_\whl$ with four ideal tetrahedra of $\whl$ will be considered throughout; the triangulation is given in  Figure~\ref{fig:whl_triangulation} and Table~\ref{tab:Face pairings}, and is derived from \tt{SnapPy}\rm's manifold data (see Appendix \ref{sec:snap data}). We remark that $\whl$ has three inequivalent minimal triangulations with four tetrahedra; they arise from the three diffent ways one can choose a ``diagonal" in an octahedron; the triangulation used here is shipped as $m129 : \#2$ in \tt{Regina}\rm's cusped orientable manifold census.
There are four tetrahedra in $\tri_\whl,$ shown in Figure~\ref{fig:whl_triangulation}, and hence four edges, which will be called red, green, black and blue. The black edge is labelled with one arrow, the blue with two, the green with three and the red with four. One of the cusps corresponds to the ideal endpoints of the red edge, and the other to the ideal endpoints of the green edge. The cusps are therefore referred to as the red cusp (also:\thinspace cusp 0) and the green cusp (also:\thinspace cusp 1) respectively. The face pairings are described in Table~\ref{tab:Face pairings}, where the face of tetrahedron $\sigma_i$ opposite vertex $j$ is denoted by $F_{i,j}.$

\begin{figure}[h]
  \begin{center}
      \includegraphics[width=14cm]{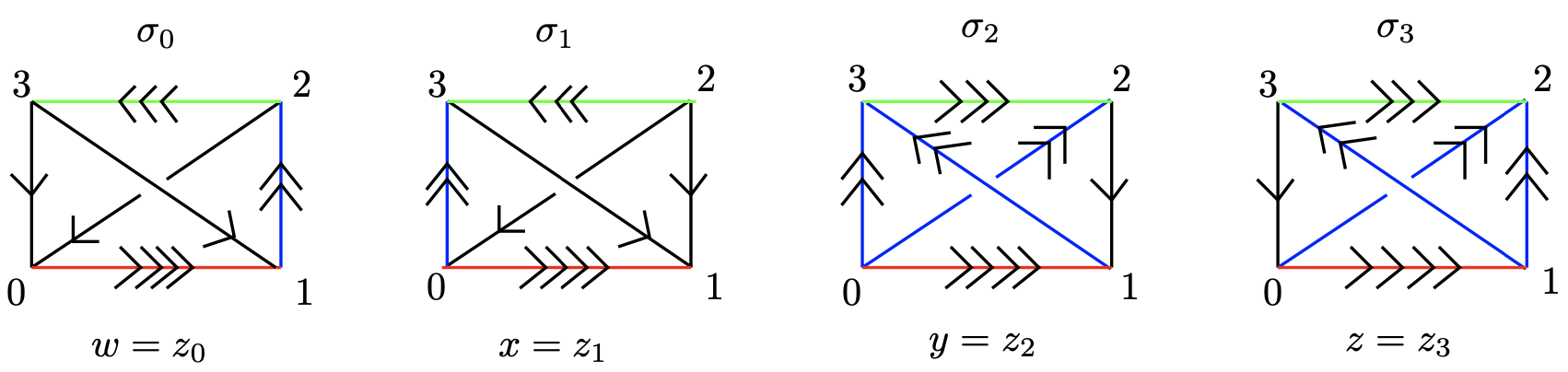}
   \end{center}
  \caption{Triangulation $\tri_\whl$ of $\whl$}
  \label{fig:whl_triangulation}
\end{figure}

\begin{table}[h]
{\small
\begin{center}
\begin{tabular}{|l|  l|  l|  l|}
\hline
$F_{0,0} \to F_{2,0}$ & $F_{1,0} \to F_{0,1}$ & $F_{2,0} \to F_{0,0}$ 
&$F_{3,0} \to F_{2,1}$ \\
$F_{0,1} \to F_{1,0}$ & $F_{1,1} \to F_{3,1}$ & $F_{2,1} \to F_{3,0}$
&$F_{3,1} \to F_{1,1}$ \\
$F_{0,2} \to F_{1,3}$ & $F_{1,2} \to F_{2,3}$ & $F_{2,2} \to F_{3,3}$
&$F_{3,2} \to F_{0,3}$\\
$F_{0,3} \to F_{3,2}$ & $F_{1,3} \to F_{0,2}$ & $F_{2,3} \to F_{1,2}$
&$F_{3,3} \to F_{2,2}$ \\ 
\hline
\end{tabular}
\end{center}
}
\caption{Face pairings}
\label{tab:Face pairings}
\end{table}

\begin{wrapfigure}{r}{58mm}
\psfrag{M}{{\small $\m$}}
\psfrag{L}{{\small $\l$}}
\begin{center}
  \includegraphics[width=5.5cm]{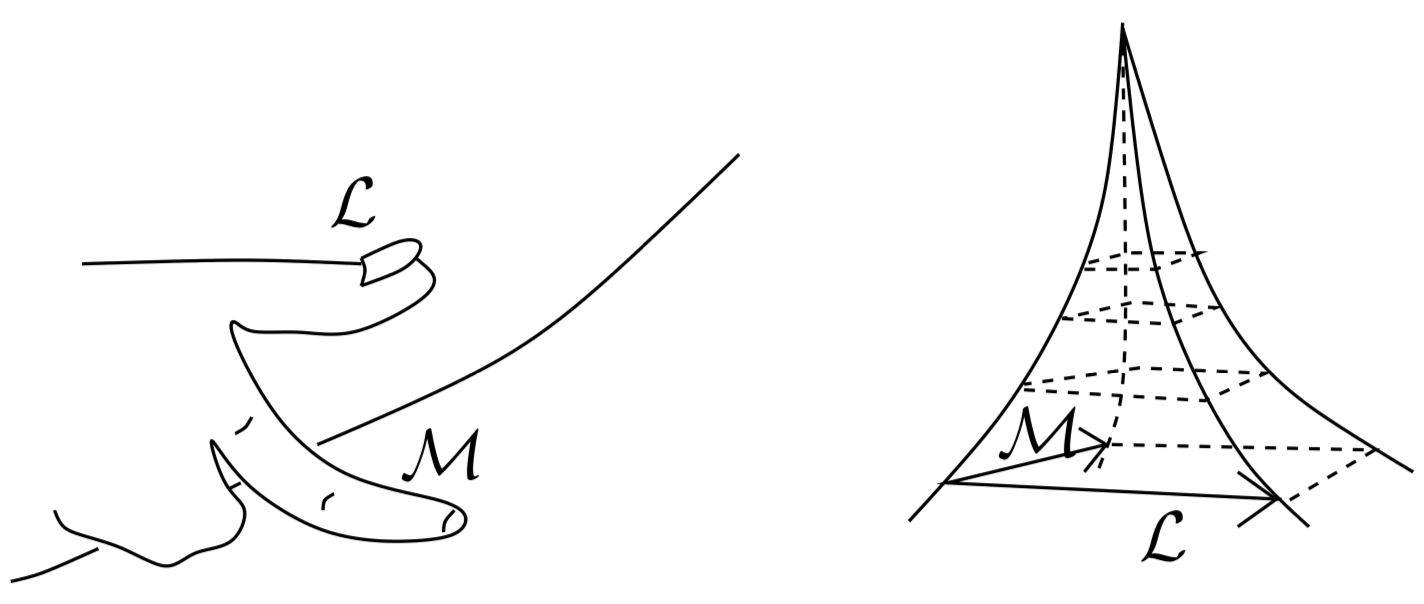}
\end{center}
\end{wrapfigure}
The convention for orientations of the peripheral elements is shown on the right, where the right--hand rule is applied to a link projection and the induced triangulation of the cusp cross--section is viewed from outside the manifold.

The induced triangulations of the cusp cross--sections of $\whl$ are given in Figure~\ref{fig:whl_cusps}. There is a standard meridian, and two choices for longitudes are considered, which are termed \emph{geometric} and \emph{topological} and denoted by $\l^g$ and $\l^t$ respectively. The geometric longitudes are chosen by \tt{SnapPy }\rm for $m129.$ They have the following property. Let $\whl (p,q)$ be the manifold obtained from hyperbolic Dehn filling on one of the cusps. The coefficients $(p,q) \in \RR^2$ are called \emph{exceptional} if $\whl (p,q)$ does not admit a complete hyperbolic metric. If the longitude is geometric, then the set of exceptional coefficients is contained in the rectangle with vertices $\pm (2, 1),$ $\pm (2, -1)$ (see \cite{hmw}). If a null--homologous longitude is chosen, the set of exceptional coefficients is contained in a parallelogram with vertices $\pm (-4,1),$ $\pm (0,1)$ for $\whl^-$ (see \cite{nr}), and $\pm (4,1),$ $\pm (0,1)$ for $\whl^+.$ 
The natural linear maps between the Dehn surgery coefficients with respect to the different peripheral systems are given in Table~\ref{tab:surg coeff}.

\begin{table}[h]
\begin{center}
\begin{tabular}{|c| c| c|}
\hline
$\{\m , \l^t\}$ for $\whl^-$ & $\{\m , \l^g\}$ for $\whl^+$ and $\whl^-$ & $\{\m , \l^t\}$ for $\whl^+$\\
\hline
$(p+2q, -q)$ & $(p,q)$      & $(p+2q,q)$\\
\hline
\end{tabular}
\end{center}
\caption{Surgery coefficients with respect to the peripheral systems}
\label{tab:surg coeff}
\end{table}

\begin{figure}[t]
  \begin{center}
                 \includegraphics[width=12cm]{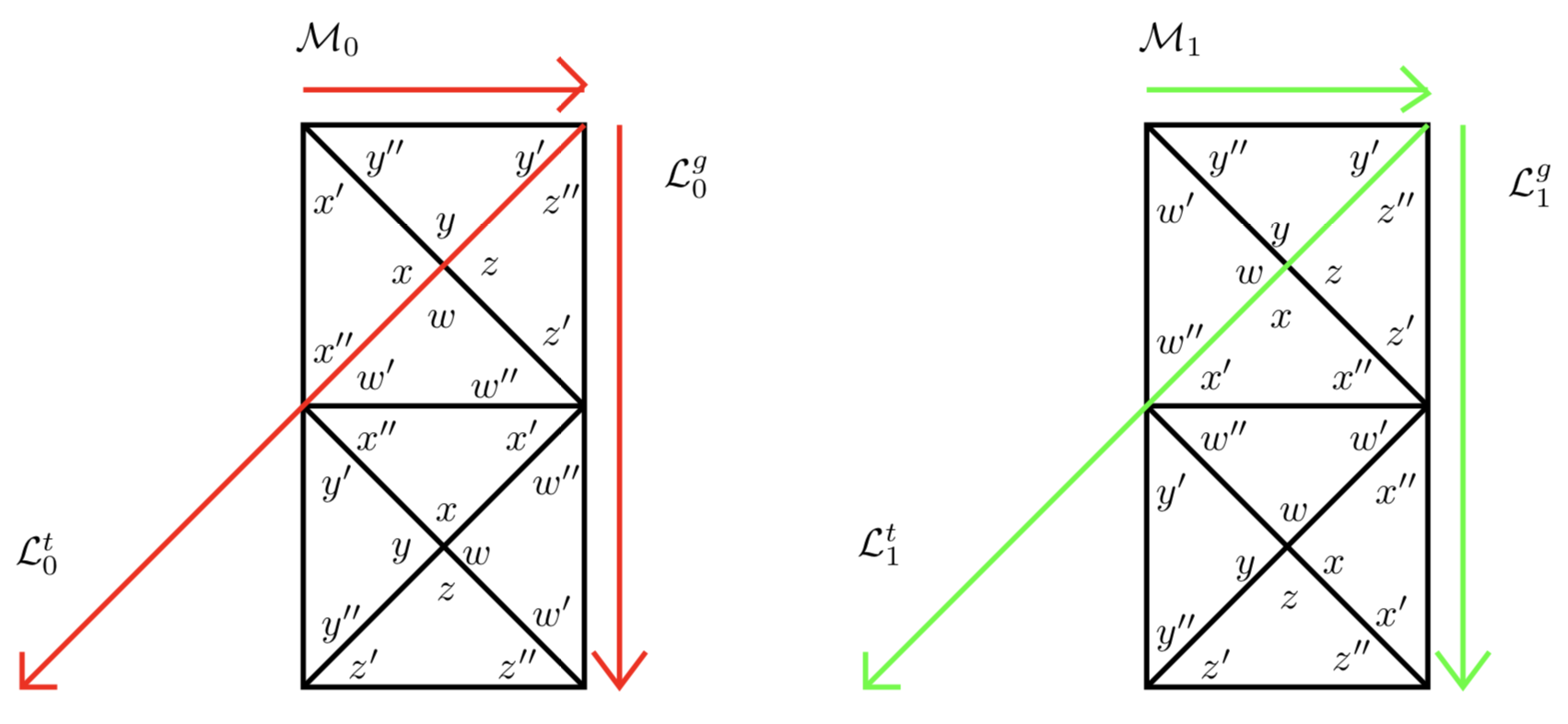}
  \end{center}
  \caption{Triangulation of cusp cross--section and peripheral elements}
  \label{fig:whl_cusps}
\end{figure}


\section{Affine algebraic sets}
\label{whl:associated varieties}

This section determines the various affine algebraic sets associated with the Whitehead link complement, as well as the maps between them. Differences between result below and results in Neumann-Reid \cite{nr} arise from the fact that the latter use a left-handed link and have a different convention regarding shape parameters.


\subsection{Deformation variety}
\label{sec:defo equations}

The shape parameters are assigned in Figure~\ref{fig:whl_triangulation}, and the following hyperbolic gluing equations can be read off from this figure: 
\begin{align}
\tag{\text{red \& green}}   1 &= w x y z \\
\tag{\text{black}} 1 &= w'x'y'z'(w'')^2 (x'')^2\\
\tag{\text{blue}}  1 &= w'x'y'z'(y'')^2 (z'')^2.
\end{align}
The deformation variety $\D(\tri_\whl)$ is defined by these equations and the parameter relations: 
\begin{align}
&1=w (1 - w''),  &&1=x (1 - x''),  &&1=y (1 - y''),  &&1=z (1 - z''),\label{eq:whl:z} \\
&1=w' (1 - w),   &&1=x' (1 - x),   &&1=y' (1 - y),   &&1=z' (1 - z), \label{eq:whl:z'}\\
&1=w'' (1 - w'), &&1=x'' (1 - x'), &&1=y'' (1 - y'), &&1=z'' (1 - z').\label{eq:whl:z''}
\end{align}
Using the fact that equations of the form $ww'w''=-1$ result from the parameter relations, it is not hard to see that $\D(\tri_\whl)$ can be defined by the parameter relations in (\ref{eq:whl:z}) and (\ref{eq:whl:z'}) together with the following two equations:
\begin{equation} \label{whl: defo simple relations}
1 = w x y z \qquad\text{and}\qquad w''x'' = y''z''.
\end{equation}
Then $\D(\tri_\whl)$ is mapped into $(\C-\{ 0\})^4$ by the projection $\varphi (w,w',{\ldots} ,z'') = (w,x,y,z).$ The closure of the resulting image is a variety defined by:
\begin{align}
\label{whl:param1} 0 &= 1 - wxyz,\\
\label{whl:param2} 0 &= wx(1 - y)(1 - z) - (1 - w)(1 - x)yz.
\end{align}
This variety is the \emph{parameter space} of \cite{t, nz}, and is denoted by $\D'(\tri_\whl).$ For any $w,x,y,z\in\C-\{0,1\}$ subject to (\ref{whl:param1}) and (\ref{whl:param2}), there is a unique point on
$\D(\tri_\whl).$ Thus, $\D(\tri_\whl)$ and $\D'(\tri_\whl)$ are birationally equivalent, and the inverse map:
\begin{align*}
&\varphi^{-1}(w,x,y,z)\\
= &\bigg(w,\frac{1}{1-w}, \frac{w-1}{w}, 
                           x,\frac{1}{1-x}, \frac{x-1}{x}, 
		   y,\frac{1}{1-y}, \frac{y-1}{y}, 
			   z,\frac{1}{1-z}, \frac{z-1}{z}\bigg)
\end{align*}
is not regular at the intersection of $\D'(\tri_\whl)$ with the collection of hyperplanes
\begin{equation*}
\{ w=1 \} \cup \{ x=1 \} \cup \{ y=1 \} \cup \{ z=1 \}
\end{equation*}
in $(\C-\{0\})^4.$ If $\D'(\tri_\whl)$ intersects any one of these hyperplanes, then it intersects either exactly two or four of them.

One of the variables, say $w,$ can be eliminated from the system of equations (\ref{whl:param1}) \& (\ref{whl:param2}), and hence there is a map $\varphi' : \D' (\whl ) \to (\C-\{0\})^3$ with the closure of its image defined by a single irreducible equation:
\begin{equation}
0= g(x,y,z) = x - xy - xz + yz - xy^2z^2 + x^2y^2z^2.
\end{equation}
Again, $\varphi'$ is a birational isomorphism onto its image, and this in particular shows that $\D(\tri_\whl)$ is irreducible. Whence $\D(\tri_\whl)$ is an irreducible variety in $(\C-\{0\})^{12}$ defined by the $10$ equations in (\ref{eq:whl:z}), (\ref{eq:whl:z'}) and  (\ref{whl: defo simple relations}), and so:

\begin{pro}
The deformation variety $\D(\tri_\whl)$ is a complete intersection variety of dimension two.
\end{pro}


\subsection{Symmetries}

There are symmetries in the defining equations of the deformation variety $\D(\tri_\whl),$ which
descend to symmetries in  (\ref{whl:param1}) and (\ref{whl:param2}).
Consider the following involutions:
\begin{align}
\label{whl:klein1}  &\tau_1 (w,x,y,z) = (z,y,x,w)  &\text{and}  &&\tau_2 (w,x,y,z) = (y,z,w,x), \\
                                 &\tau_3 (w,x,y,z) = (x,w,z,y), &                   &&\\ 
\label{whl:klein2}  &\tau_4 (w,x,y,z) = (w,x,z,y)  &\text{and}  &&\tau_5 (w,x,y,z) = (x,w,y,z).
\end{align}
Then $\tau_1\tau_2=\tau_3=\tau_4\tau_5,$ and each of the pairs (\ref{whl:klein1}) and (\ref{whl:klein2}) generates a Klein four group. The elements $\tau_6 = \tau_1 \tau_4$ and $\tau_7 = \tau_2 \tau_4$ have order four, and the group $D_4$ generated by all involutions is a dihedral group. For any $p \in \D'(\tri_\whl)$ and any $\tau \in D_4,$ $\tau p \in \D'(\tri_\whl),$ and if $\varphi^{-1} : \D'(\tri_\whl) \to \D(\tri_\whl)$ is regular at $p,$ then it is regular at $\tau p.$


\subsection{Holonomies}
\label{subsec:holonomies}

Each cusp cross section inherits a triangulation induced by $\tri_\whl.$ Let $\gamma$ be a closed simplicial path on a cusp cross section. In
\cite{nz}, the \emph{derivative of the holonomy} $H'(\gamma)$ is defined as $(-1)^{|\gamma|}$ times the product of the moduli of the triangle vertices touching $\gamma$ on the right, where $|\gamma|$ is the number of 1--simplices of $\gamma,$ and the moduli asise from the corresponding edge labels.

The derivatives of the holonomies can this be read off from Figure~\ref{fig:whl_cusps}, and simplify to:
\begin{align}
\label{whl:holo0}
H'(\m_0 ) &= \frac{z}{w'y''} =\frac{x'z''}{y}, 
&& H'(\l^t_0)= x^2 y^2, 
&& H'(\l^g_0) = \frac{(w'')^2}{(y'')^2} =\frac{(z'')^2}{(x'')^2}, \\
\label{whl:holo1}
H'(\m_1 ) &= \frac{w'z''}{y} = \frac{x}{w''z'},   
&& H'(\l^t_1)= w^2 y^2,
&& H'(\l^g_1) = \frac{(z'')^2}{(w'')^2}=\frac{(x'')^2}{(y'')^2}.
\end{align}
The completeness equations for the holonomies yield that the complete hyperbolic structure is attained at $(w,x,y,z) = (i,i,i,i).$ The action of $D_4$ on the holonomies is described in Table \ref{tab:action on hol}, where $(m_i,l_i) = (H'(\m_i ), H'(\l^t_i)).$ The relationship between elements of $D_4$ and elements of the symmetry group of $\whl$ can be deduced from this table.
\begin{table}[h]
{\small
\begin{center}
{\tiny
\begin{tabular}{ |c | c | c | c | c | c | c | c |}
\hline
& $\tau_1$ & $\tau_2$ & $\tau_3$ & $\tau_4$ & $\tau_5$
& $\tau_6$ & $\tau_7$ \\
\hline
$(m_0,l_0)$ & $(m_0,l_0)$ & 
$(m_0^{-1},l_0^{-1})$ & $(m_0^{-1},l_0^{-1})$ & $(m_1^{-1},l_1^{-1})$ & 
$(m_1,l_1)$ & $(m_1,l_1)$ & $(m_1^{-1},l_1^{-1})$\\
\hline
$(m_1,l_1)$ & $(m_1^{-1},l_1^{-1})$ 
& $(m_1,l_1)$ & $(m_1^{-1},l_1^{-1})$ & $(m_0^{-1},l_0^{-1})$ & 
$(m_0,l_0)$ & $(m_0^{-1},l_0^{-1})$ & $(m_0,l_0)$\\
\hline
\end{tabular}
}
\end{center}
}
\caption{Action of $D_4$ on holonomies}
\label{tab:action on hol}
\end{table}


\subsection{Fundamental group}
\label{subsec:Fundamental group}

An abstract presentation of $\pi_1(\whl)$ can be computed from the triangulation using the set-up in the next section:
\begin{equation*}
\pi_1(\whl )=\langle\a,\b,\c,\d \mid \d = \a\b, \b\c\d = \c\d\c, \a\c = \b\a\rangle, 
\end{equation*}
and the peripheral elements are the following words in the generators:
\begin{align*}
&\m_0 = \b^{-1},               
&&\l_0^g = \a \d^{-1} \c^{-1},
&&\l_0^t = \m_0^{-2} \l_0^g,\\
&\m_1 = \d,                  
&&\l_1^g = \a^{2}\b\c,
&&\l_1^t = \m_1^{-2} \l_1^g.
\end{align*}
Thus, $\pi_1(\whl )$ can be generated by two meridians, and one obtains a
single relation:
\begin{align}
\label{whl:fundamental group}
\pi_1(\whl ) = \langle \m_0, \m_1 \mid &
\m_1\m_0\m_1 \m_0^{-1}\m_1^{-1}\m_0^{-1}\m_1\m_0\\
&\notag\quad = \m_0\m_1 \m_0^{-1}\m_1^{-1}\m_0^{-1}\m_1\m_0\m_1 \rangle
\end{align}


\subsection{Developing map}

Given a point $Z = (w,w',{\ldots} ,z'') \in \D(\tri_\whl),$ there are (1) a continuous map $\dev_Z \co \widetilde{\whl} \to \H^3$ such that each ideal tetrahedron in $\widetilde{\whl}$ (with the lifted ideal triangulation) is mapped to an ideal hyperbolic tetrahedron of the specified shape, and (2) a unique representation $\prho_Z \co \pi_1(\whl) \to \PSL$ which makes $\dev_Z$ $\pi_1(\whl)$--equivariant (see Section 2.5 in \cite{defo}). We start with an embedding of (a lift of) $\sigma_0$ with vertices at the points $0,1,\infty, w,$ and develop around the edge $[0, \infty]$ according to the parameters given by $Z.$ This results in one lift of each ideal tetrahedron to $\H^3.$ The resulting ideal vertices of the tetrahedra are indicated in Figure~\ref{fig:whl_triang_dev}. Some tetrahedra may intersect or be ``inverted''. If all tetrahedra are positively oriented, then the resulting fundamental domain is an ideal octahedron.

Indeed, one obtains a combinatorial ideal octahedron, if one only applies the face pairings shown in Table~\ref{tab:pairings for fundom}. The Whitehead link complement is then obtained by applying the remaining face pairings, which match the boundary faces of the ideal octahedron.
\begin{table}[h]
\begin{center}
\begin{tabular}{|l|  l|  l|  l|}
\hline
$F_{0,2} \to F_{1,3}$ & $F_{1,2} \to F_{2,3}$ & $F_{2,2} \to F_{3,3}$
&$F_{3,2} \to F_{0,3}$\\
$F_{0,3} \to F_{3,2}$ & $F_{1,3} \to F_{0,2}$ & $F_{2,3} \to F_{1,2}$
&$F_{3,3} \to F_{2,2}$ \\ 
\hline
\end{tabular}
\end{center}
\caption{Face pairings used for fundamental domain}
\label{tab:pairings for fundom}
\end{table}

\begin{figure}[h]
  \begin{center}
      \includegraphics[width=16cm]{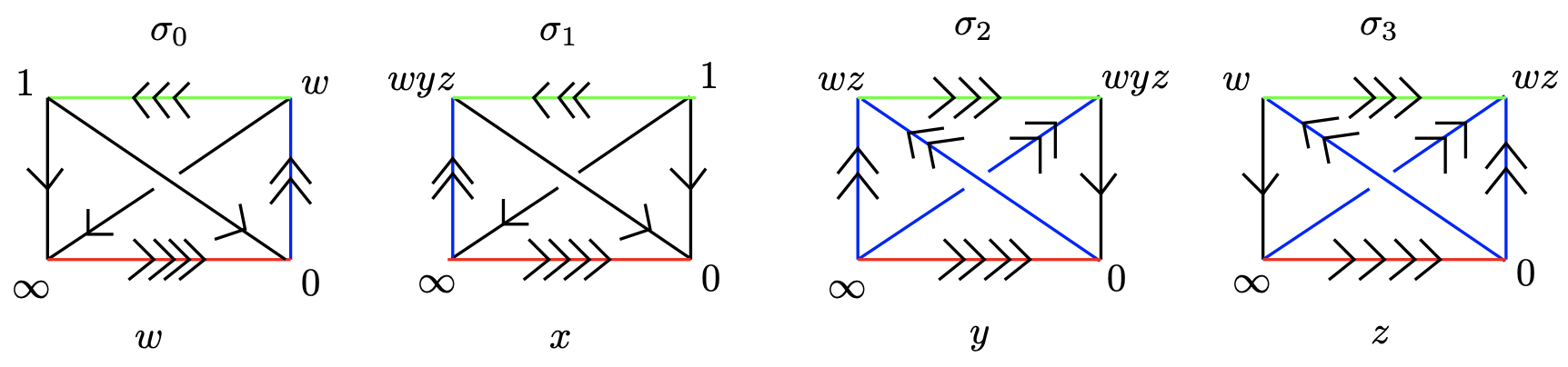}
   \end{center}
  \caption{Developing map}
  \label{fig:whl_triang_dev}
\end{figure}

Applied to the developed ideal tetrahedra in $\H^3,$ the remaining face pairings are:
\begin{align*}
\rho_Z(\a)\co & [0,1,wyz]\to [\infty, w, 1], \qquad   \rho_Z(\b)\co [w, \infty, wz] \to [1, \infty, wyz],\\
\rho_Z(\c)\co & [0,1,w] \to [0,wyz,wz],      \qquad   \rho_Z(\c\d)\co [w,0,wz] \to [wz,\infty,wyz],
\end{align*}
This determines a representation of $\pi_1(\whl)$ into $\PSL$: 
\begin{align*}
\rho_Z(\a) &= \sqrt{\frac{w'}{x'}}
    \begin{pmatrix}
          w+\frac{x'}{w'} & -\frac{x'}{w'} \\
	  1 & 0\\
    \end{pmatrix}	,
&&
\rho_Z(\b)= \sqrt{\frac{x'z''}{y}} 
    \begin{pmatrix}
          \frac{y}{x'z''} & 1 - x''z' \\
	  0 & 1\\
    \end{pmatrix}	,
\\
\rho_Z(\c) &= \sqrt{\frac{w''y'}{x}} 
    \begin{pmatrix}
          1 & 0 \\
	  xw'(1-wy)& \frac{x}{w''y'}\\
    \end{pmatrix},
&&
\rho_Z(\d) = \sqrt{\frac{y}{w'z''}} 
    \begin{pmatrix}
          (1-wx)w'             & \frac{y''}{w''} \\
	 \frac{w'}{x'} & \frac{1 - xz}{w''}\\
    \end{pmatrix}.
\end{align*}
We note:
$$\tr\rho_Z(\b) = \sqrt{\frac{x'z''}{y}}  + \sqrt{\frac{y}{x'z''}} = \sqrt{H'(\m_0 )} +  \frac{1}{\sqrt{H'(\m_0 )}}.$$
A calculation using the identity $wxyz=1$ and the standard identities for shape parameters (such as $ww'w''=-1$ and $w'(1-w)=1$) gives the following (where we have shown every second step):
\begin{align*}
\tr\rho_Z(\d) &=  \sqrt{\frac{y}{w'z''}} \Big(w'-ww'x+\frac{1}{w''}-\frac{xz}{w''}\Big)  = \sqrt{\frac{y}{w'z''}} \Big(1+\frac{w'(z-1)}{yz}\Big)
=\sqrt{\frac{y}{w'z''}} \Big(1+\frac{w'z''}{y}\Big)\\ &= \sqrt{H'(\m_1 )} +  \frac{1}{\sqrt{H'(\m_1 )}}.
\end{align*}
This algebraically confirms the consistency of the choice of meridians (up to orientation) in \S\ref{subsec:holonomies} and \S\ref{subsec:Fundamental group}.

At the complete structure, the peripheral elements have to be parabolic. This forces:
\begin{align*}
\a_0 &= \begin{pmatrix} 1+i & -1 \\ 1 & 0 \end{pmatrix}, 
&&\b_0 = \begin{pmatrix} 1 & 1-i \\ 0 & 1 \end{pmatrix},\\
\c_0 &= \begin{pmatrix} 1 & 0 \\ -1+i & 1 \end{pmatrix},
&&\d_0 = \begin{pmatrix} 1+i & 1 \\ 1 & 1-i \end{pmatrix}.
\end{align*}
There is a well--defined map $\chidefo=\chi_{\tri_\whl}\co \D(\tri_\whl)\to \PX (\whl),$ and it is shown in the next section that $\chidefo$ has degree 4. This corresponds to the observation that the manifolds $\whl (p_0,q_0,p_1,q_1),$ $\whl (-p_0,-q_0,p_1,q_1),$ $\whl (p_0,q_0,-p_1,-q_1),$ and $\whl (-p_0,-q_0,-p_1,-q_1)$ are geometrically distinguished (they ``spiral in different directions into the cusps''), whilst they have isomorphic fundamental groups.  


\subsection{Face pairings}
\label{whl:tautologicum}

A variety $\Ta (\whl )$ parameterising representations into $\SL$ is computed as follows. Putting:
\begin{equation}
\label{whl:taut}
\rho(\m_0 ) = \begin{pmatrix} s & c \\ 0 & s^{-1} \end{pmatrix} \text{  and  }
\rho(\m_1 ) = \begin{pmatrix} u & 0 \\ 1 & u^{-1} \end{pmatrix},
\end{equation}
one obtains a single equation using (\ref{whl:fundamental group}):
\begin{align}
\label{whl:taut_rel}
0&= f(s,u,c) \\
\notag&= (s-s^{-1})(u-u^{-1}) 
         + c(s^{-2}u^{-2}-u^{-2}-s^{-2}+4-s^2-u^2+s^2u^2)\\
  \notag &\qquad +c^2(2s^{-1}u^{-1}-su^{-1}-s^{-1}u+2su) + c^3 
  \in \C[s^{\pm 1}, u^{\pm 1}, c].
\end{align}
Then $\Ta (\whl )$ is an irreducible hypersurface in $(\C - \{ 0\})^2 \times \C,$
and its intersection with $c=0$ is the collection of lines
$\{ s^2 = 1, c=0\} \cup \{ u^2 = 1, c=0\},$ which parametrises reducible
representations. Moreover, $\Ta (\whl )$ is a cover of the Dehn surgery
component $\X_0(\whl)$ of the character variety, since any irreducible
representation of $\pi_1(\whl )$ into $\SL$ is conjugate to an element of
$\Ta (\whl ),$ and it is a 4--to--1 branched cover of $\X_0(\whl)$ since
$\Ta (\whl )$ is not contained in the union of hypersurfaces $s^2 = 1$ and
$u^2 = 1.$

If  $f(s,u,c)=0,$ then $f(-s,u,-c) = f(s,-u,-c)=f(-s,-u,c)=0,$ and the four
solutions correspond to the action of $\Hom(\pi_1(\whl), \Z_2)$ on
$\Ta(\whl ).$ To obtain a description of the corresponding quotient map
$\Ta (\whl ) \to \PTa (\whl ),$ and hence a parametrisation of the associated variety $\PTa
(\whl)$ of representations into $\PSL,$ note that  (\ref{whl:taut}) can be adjusted by a conjugation and
rewritten in the form:  
\begin{align}
\label{whl:good form}
&\rho(\m_0 ) = \frac{1}{s} \begin{pmatrix} s^2 & c s u\\ 0 & 1 \end{pmatrix}
\quad\text{and}\quad
\rho(\m_1 ) = \frac{1}{u} \begin{pmatrix} u^2 & 0 \\ 1 & 1 \end{pmatrix},
\end{align}
which is still subject to $0= f(s,u,c).$
The map $\qe_1: (\C-\{0\})^2\times \C \to (\C-\{0\})^2\times \C$ defined by
$\qe_1 (s,u,c) = (s^2, u^2, c s u)$ can be identified with the natural
quotient map $\Ta (\whl ) \to \PTa (\whl ),$ and the defining equation for
$\PTa (\whl )$ can be derived from this relationship. Thus, $\PTa (\whl )$ can be
viewed as a variety of representations into $GL_2(\C)$:
\begin{align}
\label{whl:ptaut}
& \prho_{GL}(\m_0 ) = \begin{pmatrix} \s & d \\ 0 & 1 \end{pmatrix}
\quad\text{and}\quad
\prho_{GL}(\m_1 ) = \begin{pmatrix} \u & 0 \\ 1 & 1 \end{pmatrix}
\quad \text{subject to}\\
0&= \s\u (\s -1)(\u-1)
     + d(1  - \s- \u + 4 \s \u - \s^2\u - \s\u^2 + \s^2\u^2)\\
\notag
&\quad + d^2 ( 2  - \s- \u + 2\s\u)+ d^3 \in \C[\s^{\pm 1}, \u^{\pm 1},d].
\end{align}
For each $\prho_{GL}$ as above, there is a unique $\PSL$--representation $\prho$ such that $\prho_{GL}$ and $\prho$ are identical as representations into the group of projective transformations of $\C P^1.$ Hence there is a bijection between the set of all of the above $GL_2(\C )$--representations of $\pi_1(\whl )$ with $d\neq 0$ and the set of all irreducible $\PSL$--representations of $\pi_1(\whl )$ (up to conjugation) which lift to $\SL.$

\begin{lem}
There is a birational isomorphism $\edefo : \D(\tri_\whl) \to \PTa (\whl ).$ 
\end{lem}

\begin{proof}
To construct the map $\edefo: \D(\tri_\whl) \to \PTa (\whl ),$ conjugate the face
pairings $\rho_Z(\a),$ \ldots, $\rho_Z(\d)$ by a \emph{suitably chosen} matrix $A$ to obtain a form
analogous to (\ref{whl:good form}), and then adjust the resulting
representation by multiplication:
\begin{align*}
\prho'_Z(\m_0) &:=
\sqrt{\frac{x}{w''y'}}\  A \ \rho_Z(\b^{-1}) \ A^{-1} =
\begin{pmatrix} \frac{w''y'}{x} & z -1 \\ 0 & 1 \end{pmatrix}
\\
\text{and}\quad
\prho'_Z(\m_1 ) &:=
\sqrt{\frac{w''z'}{x}}\ A \ \rho_Z(\d)\ A^{-1} = 
\begin{pmatrix} \frac{x}{w''z'} &  0 \\ 1 & 1 \end{pmatrix}.
\end{align*}
It can be verified that $\prho'_Z$ defines a
$GL_2(\C)$--representation for each $Z \in \D(\tri_\whl).$
This induces a map $\edefo : \D(\tri_\whl) \to \PTa (\whl )$ defined by:
\begin{align}
\edefo (w,w',{\ldots} ,z'') = 
\bigg( \frac{w''y'}{x} , \frac{x}{w''z'}, z-1 \bigg).
\end{align}
An elementary calculation shows that this map is 1--1 and that its image is
dense in $\PTa (\whl ).$ A rational inverse is given by the following map
$\PTa (\whl ) \to \D'(\tri_\whl)$: 
{\small
\begin{align}
(\s, \u, d) \to 
\bigg( 
\frac{d + d^2 - d\s + \s\u + d\s\u}{(1 + d)(d + \s\u)},
\frac{\s\u}{d + d^2 - d\s + \s\u + d\s\u}, 
1 + \frac{d}{\s\u}, 
1 + d \bigg),
\end{align}
}
which composed with $\varphi: \D'(\tri_\whl) \to \D(\tri_\whl)$
gives a map $\edefo^{-1} : \PTa (\whl ) \to \D(\tri_\whl).$
Composition of the maps $\edefo$ and $\edefo^{-1}$ induces the identity on both $\D(\tri_\whl)$ and
$\PTa (\whl ),$ and this proves the lemma.
\end{proof}

The only singularity of $\D(\tri_\whl)$ is at infinity; at the ideal point where all of $w,x,y,z$ tend to one, and the only singularity of $\D' (\whl )$ is at the corresponding point $(w,x,y,z)=(1,1,1,1).$ The variety $\PTa (\whl )$ is an irreducible variety without singularities, and the above proof shows that it parametrises developing maps when thought of as a variety in $(\C-\{0\})^3.$ It is shown in Subsection \ref{whl:eigenvalues at (1,1,1,1)} that the intersection of $\PTa (\whl )$ with $c=0$ corresponds to the ideal point of $\D(\tri_\whl)$ where $w,x,y,z \to 1.$ Thus, it is a natural de-singularisation of $\D(\tri_\whl).$

\begin{cor}
$\chidefo\co \D(\tri_\whl) \to \PX_0 (\whl )$ is generically 4--to--1 and onto.
\end{cor}

\begin{proof}
Note that $\chidefo\co \D(\tri_\whl) \to \PX (\whl )$ factors through
$\PTa (\whl),$ and hence it is enough to show that the map
$\PTa (\whl) \to \PX_0 (\whl )$ is generically 4--to--1 and onto. Firstly,
$\PTa (\whl)$ is irreducible, and hence its image in $\PX (\whl)$ is
irreducible. Secondly, the above lemma implies that $\PTa (\whl)$ contains a
discrete and faithful representation, and hence maps to the Dehn surgery
component. Since there are $\prho \in \PTa (\whl)$ such that
$(\tr\prho (\m_0))^2 \ne 4 \ne (\tr\prho (\m_1))^2,$ there are generically four
elements of each conjugacy class of representations contained in
$\PTa (\whl).$ Hence, the degree of $\chidefo$ is four.
\end{proof}


\subsection{Eigenvalue maps}
\label{whl:Eigenvalue maps}

As in \cite{tillus_ei}, denote the respective eigenvalue varieties by $\Ei (\whl )$ and $\PEi (\whl ).$ The aim of this section is to show that the subvariety $\PEi_0 (\whl)$ corresponding to $\PX_0(\whl )$ is birationally equivalent to $\D(\tri_\whl).$ The following lemma establishes suitable affine coordinates and maps needed for this. 

\begin{lem}\label{whl: eigenvalue maps lem}
There are a quotient map $\qe_3\co \Ei (\whl ) \to \PEi (\whl )$ corresponding
to the action of $\Hom(\pi_1(\whl ), \Z_2),$ and
\emph{eigenvalue maps} $\ee \co \Ta (\whl ) \to \Ei (\whl )$
and $\pee \co \PTa (\whl ) \to \PEi (\whl )$ such that the following diagram
commutes:
\begin{center}
$ 
\begin{CD}
      \Ta(\whl )    @>\ee>>     \Ei_0(\whl )   \\ 
       @V\qe_1 VV          @VV\qe_3 V\\
     \PTa (\whl )   @>\pee>>     \PEi_0 (\whl )
\end{CD}
$
\end{center}
\end{lem}

\begin{proof}
Let $\vartheta_\gamma \co \Ta (\whl ) \to \C$ be
the holomorphic map which takes $\rho$ to the upper left entry of
$\rho (\gamma).$ Then define
$\ee \co \Ta (\whl ) \to \Ei (\whl )$ to be the map
\begin{equation}\label{whl:eigenvalue map}
\ee (\rho ) = (\vartheta_{\m_0}(\rho ), \vartheta_{\l^t_0}(\rho ),
                 \vartheta_{\m_1}(\rho ), \vartheta_{\l^t_1}(\rho )).
\end{equation}
Since every $\rho \in \Ta (\whl )$ is triangular on the peripheral
subgroups, it follows that this map is well--defined. Denote the affine
coordinates of $\Ei (\whl )$ by $(s,t,u,v),$ so they correspond to
eigenvalues of $\m_0,$ $\l^t_0,$ $\m_1$ and $\l^t_1.$

Denote the map which takes
$\prho_{GL}$ to the upper left entry of $\prho_{GL}(\gamma )$ by
$\vartheta_\gamma$ as well, and let 
$\pee \co \PTa (\whl ) \to \PEi (\whl )$ be the map
\begin{equation}
\pee (\prho_{GL} ) = (\vartheta_{\m_0}(\prho_{GL} ),
\vartheta_{\l^t_0}(\prho_{GL} ),
                 \vartheta_{\m_1}(\prho_{GL} ), \vartheta_{\l^t_1}(\prho_{GL} )).
\end{equation}
To verify that this map has the right range, note that the upper left entries of $\prho_{GL} (\m_0)$ and 
$\prho_{GL} (\m_1)$ are the squares of the eigenvalues of matrices representing the associated $\PSL$--representation. The longitudes are the following words in the meridians:
\begin{align}
\l^t_0 &=  \m_0^{-1}\m_1\m_0\m_1^{-1}\m_0^{-1}\m_1^{-1}\m_0\m_1\\
\l^t_1 &=  \m_1^{-1}\m_0\m_1\m_0^{-1}\m_1^{-1}\m_0^{-1}\m_1\m_0,
\end{align}
and hence $\prho_{GL} (\l^t_i) = \prho (\l^t_i)$ for any $\prho_{GL}$ and its corresponding unique $\PSL$--representation $\prho.$ In particular, $\vartheta_{\l^t_i} (\prho_{GL})$ is an eigenvalue of
$\prho_{GL} (\l^t_i) = \prho (\l^t_i)$ and it is independent of the choice of the signs of matrices representing $\prho (\m_0)$ and  $\prho (\m_1).$ It follows that $\PEi (\whl )$ can be given the affine coordinates $(\s, t, \u, v).$

The natural quotient map which makes the above diagram commute is therefore $\qe_3 \co \Ei (\whl ) \to \PEi (\whl )$ defined by $\qe_3 (s,t,u,v) = (s^2,t,u^2,v).$ This map clearly corresponds to the action of $\Hom(\pi_1(\whl ), \Z_2)$ on $\Ei (\whl ).$

Since any irreducible $\SL$--representation is conjugate to a representation in $\Ta(\whl),$ the closure of the image of $\ee$ is the component of $\Ei(\whl)$ corresponding to the Dehn surgery component $\X_0(\whl).$ It follows that the closure of the image of $\pee$  corresponds to $\PX_0(\whl).$ In particular the composite mapping $\Psi= \pee \circ \edefo \co \D(\tri_\whl) \to \PEi_0 (\whl )$ is onto.
\end{proof}

From the face pairing $\rho_Z\co \pi_1(M)\to \PSL,$ one computes
$\vartheta_{\m_0} (\prho_Z) = \frac{x'z''}{y},$
$\vartheta_{\l^t_0} (\prho_Z) = xy,$
$\vartheta_{\m_1} (\prho_Z) = \frac{w'z''}{y}$ and
$\vartheta_{\l^t_1} (\prho_Z) = wy.$
The map $\Psi \co \D(\tri_\whl) \to \PEi (\whl)$ with respect to the chosen
coordinates is therefore given by
\begin{align} \label{whl:holo map}
\Psi (w,w',w'',x,x',x'',y,y',y'',z,z',z'') = 
\big( \frac{x'z''}{y},x y, \frac{w'z''}{y}, w y \big)
\end{align} 
We have $\Psi\varphi^{-1}(i,i,i,i) = (1,-1,1,-1)$ at the complete structure.

\begin{lem} \label{whl:degree one lemma}
The map $\Psi \co \D(\tri_\whl) \to \PEi_0 (\whl)$ is a birational isomorphism.
\end{lem}

\begin{proof}
Since $\Psi$ is a regular map and $\D(\tri_\whl)$ is irreducible, it follows that the closure of its image, which has been identified as $\PEi_0(\whl ),$ is irreducible. Thus, $\Psi \co \D(\tri_\whl) \to\PEi_0 (\whl )$ is a regular map of irreducible varieties. It remains to show that it has degree one. 

Assume that $\Psi\varphi^{-1}(w_0,x_0,y_0,z_0)=\Psi\varphi^{-1}(w_1,x_1,y_1,z_1),$ where $(w_i,x_i,y_i,z_i)$ are two regular points of $\varphi$ on $\D'(\tri_\whl).$ An elementary calculation shows that the points are either identical or we have $w_0=x_0$ and $w_1=x_1.$ Thus, all points on which $\Psi$ does not have degree one are contained on the hypersurface $w=x.$ Since for any $w \in\C-\{0,\pm 1\},$ the point $\varphi^{-1} (w,-w^{-1},w,-w^{-1})$ is contained in $\D(\tri_\whl),$ $\D(\tri_\whl)$ is not contained in this hypersurface.
\end{proof}

An inverse $\PEi_0(\whl ) \to \D(\tri_\whl)$ taking
$(\s,t,\u,v) \to (w,w',{\ldots} ,z'')$ is determined by the following map
$\PEi_0(\whl ) \to \D' (\whl )$ which can be computed from
(\ref{whl:holo map}):
\begin{equation*}
(\s,t,\u,v) \to
\Bigg(
\frac{v(\s - \u)}{\s t - \u v},
\frac{t(\s - \u)}{\s t - \u v},
\frac{\s t-\u v}{\s-\u},
\frac{\s t-\u v}{tv(\s-\u)}
\Bigg).
\end{equation*}
This map is not regular on a 1--dimensional subvariety of $\PEi_0(\whl),$
which is defined by the following three equations: 
\begin{align} \label{whl: inverse not defined}
\s = \u, && t = v, && 0 = \s - t + \s t + \s^2 t - \s^2 t^2 - \s^3 t^2 + \s^4
t^2 - \s^3 t^3.
\end{align}
See Subsection \ref{whl:compute eigen} for a computation of the eigenvalue varieties and this subvariety.


\section{Embedded surfaces}
\label{whl:embedded normal surfaces}

This section gives a complete description of the space $\FH$ of all essential surfaces in the Whitehead link complement, and deduces the associated boundary curve space $\BC (\whl)$ and the unit ball of the Thurston norm. This is compared with previous work of Floyd and Hatcher~\cite{fh}, Lash~\cite{la} and Hoste and Shanahan~\cite{hs}.

The main tools used to analyse incompressible surfaces in the projective admissible solution space $\N(\tri_\whl)$ of spun-normal surface theory are a criterion due to Dunfield that determines which spun-normal surfaces are essential, and results by Walsh~\cite{wa} and Kang and Rubinstein~\cite{KR-2015} on the normalisation of essential surfaces.


\subsection{Complete description of the projective admissible solution space}

The convention for quadrilateral coordinates for orientable manifolds in \cite{tillmann08-finite} is used. Table \ref{tab:whl:quad types} indicates the position of the quadrilaterals in the tetrahedra. For reference, we also add the shape parameters, where we write $w = z_0,$ $x = z_1,$ $y = z_2,$ and $z = z_3.$ 
The notation $ij/kl$ means that the particular quadrilateral type separates the vertices $i$ and $j$ from the vertices $k$ and $l.$ 
\begin{table}[h]
\begin{center}
\begin{tabular}{| l |c | c | c |}
\hline
Quadrilateral & $q_i$ & $q'_i$ & $q''_i$ \\
\hline
Separates     & $01/23$ & $03/12$ & $02/13$ \\
\hline
Edge label & $z_i$ & $z'_i$ & $z''_i$ \\
\hline
\end{tabular}
\end{center}
\caption{Quadrilateral types}
\label{tab:whl:quad types}
\end{table}

The $Q$--matching equations can be worked out directly from the triangulation or by using their relationship with the gluing equations described in \S\ref{subsec:relationship between gluing and matching}. There are the following three equations:
\begin{align}
\tag{\text{red \& green}}   0 &= q'_0-q''_0 + q'_1-q''_1 + q'_2-q''_2 + q'_3-q''_3,  \\
\tag{\text{black}} 0 &= q_0-2q'_0+q''_0 + q_1-2q'_1+q''_1-q_2+q''_2-q_3+q''_3
\\
\tag{\text{blue}}  1 &= -q_0+q''_0-q_1+q''_1+q_2-2q'_2+q''_2+q_3-2q'_3+q''_3
\end{align}
These are equivalent to the following two, which correspond to \eqref{whl: defo simple relations}:
\begin{align}
\label{whl:Q-match1} 0 &= q'_0-q''_0 + q'_1-q''_1 + q'_2-q''_2 + q'_3-q''_3,  \\
\label{whl:Q-match2} 0 &= q_0-q'_0 + q_1-q'_1 - q_2+q'_2 - q_3+q'_3. 
\end{align}
To obtain an admissible solution, one now sets two quadrilateral coordinates from each tetrahedron equal to zero and solves the above equations subject to this constraint. Hence the set $\N(\tri_\whl)$ can be computed by solving $3^4 = 81$ systems of linear equations. This is automated by feeding the triangulation to {\tt Regina}\rm. Alternatively, given its description as a tropical pre-variety in Proposition~\ref{comb:homeo}, one can use {\tt gfan }\rm to compute $\N(\tri_\whl)$ from the defining equation of $\D (\tri_\whl)$ using the command \texttt{tropical\_intersection}.

\begin{figure}[t]
  \begin{center}
      \includegraphics[width=12cm]{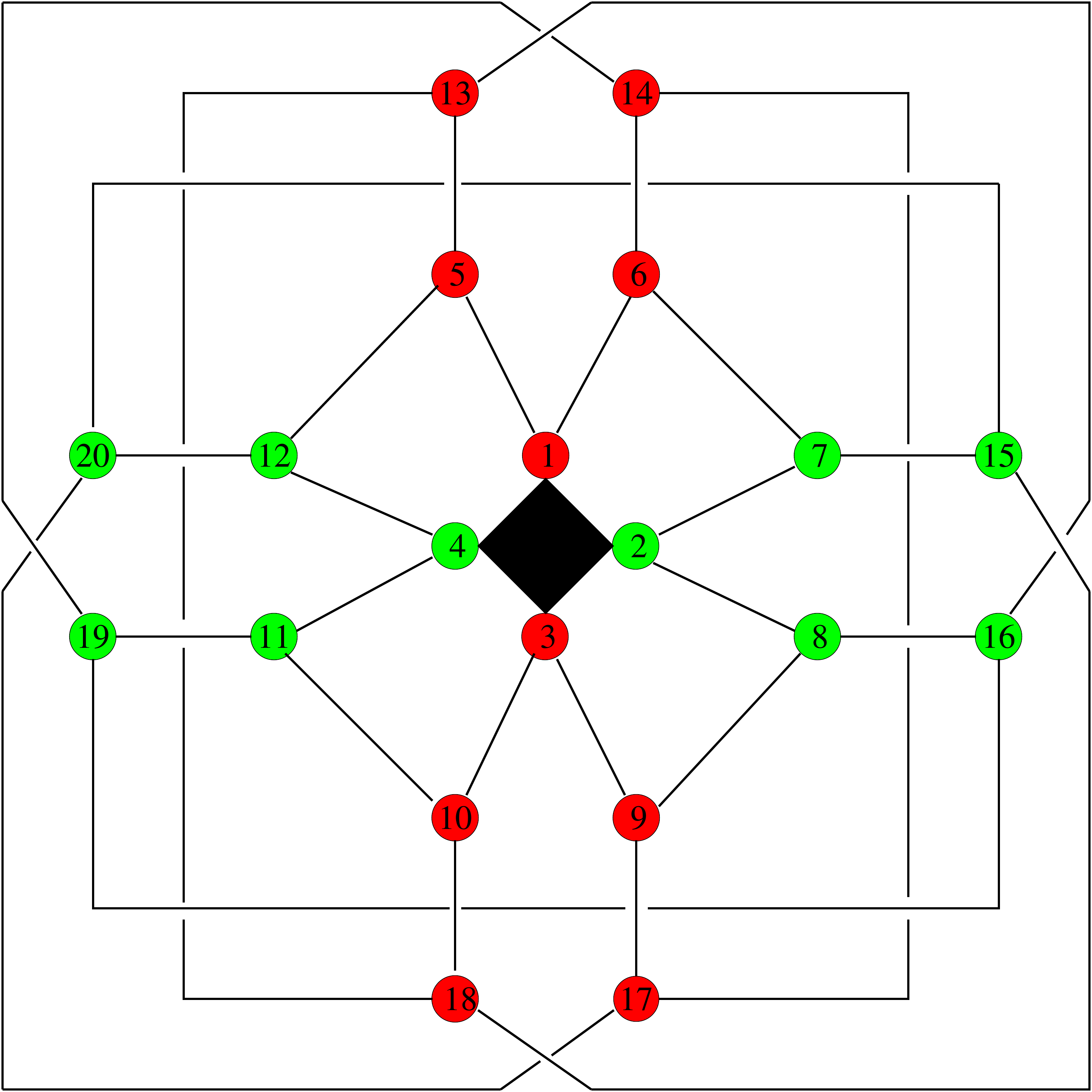}
   \end{center}
  \caption{Projective admissible solution space $\N(\tri_\whl)$: The labels of the nodes correspond to the vertex solutions; the shown PL arcs between them can be realised as geodesics in $S^{11}$ with pairwise disjoint interior, and the quadrilateral as a convex spherical quadrilateral.}
  \label{fig:whl_admissible}
\end{figure}

The set $\N(\tri_\whl)$ is a finite union of convex spherical polytopes in $S^{11}\subset \RR^{12}.$ It turns out that there are 28 geodesic arcs and one geodesic quadrilateral, spanned by a total of 20 vertices. These are indicated in Figure~\ref{fig:whl_admissible}, where geodesic arcs are represented as PL arcs.

The vertices are denote by $V_1,{\ldots} ,V_{20},$ and called \emph{vertex solutions}. These correspond to extremal solutions to the linear equations. Each vertex solution $V_i$ is rescaled by a positive real to obtain a (minimal) integer solution, and the resulting normal surface is denoted by $F_i.$ Hence the normal $Q$--coordinate $N(F_i)$ is a multiple of $V_i.$
These normal surfaces are described in Table~\ref{table:white_normal} and the information given in the table is as follows:


\begin{sidewaystable}
{\small
\begin{center}
\begin{tabular}{| r || l|l|l || l|l|l || l|l|l || l|l|l || c | l | c | r r r r|}
\hline
Vertex & $q_0$ & $q'_0$ & $q''_0$ & $q_1$ & $q'_1$ & $q''_1$
& $q_2$ & $q'_2$ & $q''_2$ & $q_3$ & $q'_3$ & $q''_3$
& type & class & $\partial$--curves & $\nu(\l^t_0),$& $-\nu(\m_0),$& $\nu(\l^t_1),$& 
$-\nu(\m_1)$ \\
\hline\hline
1 & 
1 & 0 & 0 & 0 & 0 & 0 & 1 & 0 & 0 & 0 & 0 & 0 & 
$T_1$ & $N'_1$ & $0/1$ & 0,&0,&0,&--1 \\
2 & 
0 & 0 & 0 & 1 & 0 & 0 & 1 & 0 & 0 & 0 & 0 & 0 &
$T_1$ & $N_1$ & $1/0$ & 0,&--1,&0,&0 \\
3 & 
0 & 0 & 0 & 1 & 0 & 0 & 0 & 0 & 0 & 1 & 0 & 0 &
$T_1$ & $N'_1$ & $0/1$ &0,&0,&0,&1 \\
4 & 
1 & 0 & 0 & 0 & 0 & 0 & 0 & 0 & 0 & 1 & 0 & 0 &
$T_1$ & $N_1$ & $1/0$ & 0,&1,&0,&0 \\
\hline\hline
5 & 
1 & 0 & 0 & 0 & 1 & 0 & 0 & 0 & 0 & 0 & 0 & 1 &
$S_3$ & $N'_2$ & $2/1$ &--2,&0,&0,&--1 \\
6 & 
0 & 0 & 0 & 0 & 0 & 1 & 1 & 0 & 0 & 0 & 1 & 0 &
$S_3$ & $N'_2$ & $2/1$ & 2,&0,&0,&--1 \\
7 & 
0 & 0 & 1 & 0 & 0 & 0 & 1 & 0 & 0 & 0 & 1 & 0 &
$S_3$ & $N_2$ & $1/2$ & 0,&--1,&2,&0 \\
8 & 
0 & 1 & 0 & 1 & 0 & 0 & 0 & 0 & 0 & 0 & 0 & 1 &
$S_3$ & $N_2$ & $1/2$ & 0,&--1,&--2,&0 \\
9 & 
0 & 1 & 0 & 1 & 0 & 0 & 0 & 0 & 1 & 0 & 0 & 0 &
$S_3$ & $N'_2$ & $2/1$ & 2,&0,&0,&1 \\
10 & 
0 & 0 & 1 & 0 & 0 & 0 & 0 & 1 & 0 & 1 & 0 & 0 &
$S_3$ & $N'_2$ & $2/1$ & --2,&0,&0,&1 \\
11 & 
0 & 0 & 0 & 0 & 0 & 1 & 0 & 1 & 0 & 1 & 0 & 0 &
$S_3$ & $N_2$ & $1/2$ & 0,&1,&--2,&0\\
12 & 
1 & 0 & 0 & 0 & 1 & 0 & 0 & 0 & 1 & 0 & 0 & 0 &
$S_3$ & $N_2$ & $1/2$ & 0,&1,&2,&0 \\
\hline\hline
13 & 
0 & 0 & 0 & 0 & 1 & 0 & 0 & 1 & 0 & 0 & 0 & 2 &
$R_2$ & $N'_3$ & $1/1$ & --4,&--1,&--2,&--1 \\
14 & 
0 & 1 & 0 & 0 & 0 & 2 & 0 & 0 & 0 & 0 & 1 & 0 &
$R_2$ & $N'_3$ & $1/1$ & 4,&1,&--2,&--1 \\
15 & 
0 & 0 & 2 & 0 & 1 & 0 & 0 & 0 & 0 & 0 & 1 & 0 &
$R_2$ & $N_3$ & $1/1$ & --2,&--1,&4,&1 \\
16 & 
0 & 1 & 0 & 0 & 0 & 0 & 0 & 1 & 0 & 0 & 0 & 2 &
$R_2$ & $N_3$ & $1/1$ & --2,&--1,&--4,&--1 \\
17 & 
0 & 1 & 0 & 0 & 0 & 0 & 0 & 0 & 2 & 0 & 1 & 0 &
$R_2$ & $N'_3$ & $1/1$ & 4,&1,&2,&1 \\
18 & 
0 & 0 & 2 & 0 & 1 & 0 & 0 & 1 & 0 & 0 & 0 & 0 &
$R_2$ & $N'_3$ & $1/1$ & --4,&--1,&2,&1 \\
19 & 
0 & 1 & 0 & 0 & 0 & 2 & 0 & 1 & 0 & 0 & 0 & 0 &
$R_2$ & $N_3$ & $1/1$ & 2,&1,&--4,&--1 \\
20 & 
0 & 0 & 0 & 0 & 1 & 0 & 0 & 0 & 2 & 0 & 1 & 0 &
$R_2$ & $N_3$ & $1/1$ & 2,&1,&4,&1\\
\hline
\hline
&&&&&&&&&&&&&\multicolumn{7}{|l|}{}\\
Angle & $\alpha_0$ & $\alpha'_0$ & $\alpha''_0$ & $\alpha_1$ & $\alpha'_1$ & $\alpha''_1$
& $\alpha_2$ & $\alpha'_2$ & $\alpha''_2$ & $\alpha_3$ & $\alpha'_3$ & $\alpha''_3$ &
\multicolumn{7}{|l|}{dual to vertex surfaces numbers} \\
\hline
\hline
$\alpha^+$ & 
0 & 0 & $\pi$ & $\pi$ & 0 & 0 & 0 & 0 & $\pi$ & $\pi$ & 0 & 0 &
\multicolumn{7}{|l|}{1, 5, 6, 13, 14, 16, 19} \\
\hline
$\alpha^-$ & 
0 & $\pi$ & 0 & $\pi$ & 0 & 0 & 0 & $\pi$ & 0 & $\pi$ & 0 & 0 &
\multicolumn{7}{|l|}{1, 5, 6, 7, 12, 15, 20} \\
\hline
$\beta^+$ & 
$\pi$ & 0 & 0 &  0 & 0  & $\pi$ & 0 &  0 & $\pi$ & $\pi$ & 0 & 0 &
\multicolumn{7}{|l|}{2, 7, 8, 13, 15, 16, 18} \\
\hline
$\beta^-$ & 
$\pi$ & 0 & 0 &  0 & $\pi$ & 0 & 0 & $\pi$ & 0 & $\pi$ & 0 & 0 &
\multicolumn{7}{|l|}{2, 6, 7, 8, 9, 14, 17} \\
\hline
$\gamma^+$ & 
$\pi$ & 0 & 0 &  0 &  0 &$\pi$ &  $\pi$& 0 & 0 &  0 & 0 &$\pi$ &
\multicolumn{7}{|l|}{3, 9, 10, 15, 17, 18, 20} \\
\hline
$\gamma^-$ & 
$\pi$ & 0 & 0 &  0 & $\pi$ & 0 &  $\pi$ &0 & 0 &  0 &$\pi$ & 0 &
\multicolumn{7}{|l|}{3, 8, 9, 10, 11, 16, 19} \\
\hline
$\delta^+$ & 
 0 & 0 &$\pi$ &  $\pi$ &0 &  0 &  $\pi$ &0 & 0 &  0 & 0 &$\pi$ &
\multicolumn{7}{|l|}{4, 11, 12, 14, 17, 19, 20} \\
\hline
$\delta^-$ & 
 0 &$\pi$ & 0 & $\pi$ & 0 &  0 &$\pi$ & 0  & 0 & 0 & $\pi$ & 0 &
\multicolumn{7}{|l|}{4, 5, 10, 11, 12, 13, 18} \\
\hline\hline
\end{tabular}
\end{center}
} 
\caption{Minimal representatives for vertex solutions and some dual angle structures}
\label{table:white_normal}
\end{sidewaystable}


First, the normal $Q$--coordinate $N(F_i)$ is given, then the topological type of $F_i,$ where $T_1$ stands for a once--punctured torus, $S_3$ for a thrice--punctured sphere, $R_2$ for a twice--punctured $\RR P^2.$ The abbreviations $K_i$ and $T_i$ will be used in subsequent figures for an $i$--punctured Klein bottle and an $i$--punctured torus respectively, and $G_2$ denotes a (closed) genus two surface.

The column \emph{class} specifies the equivalence class (defined in Subsection \ref{whl:Equivalence classes}) that the projective normal $Q$--coordinate $V_i$ of $F_i$ belongs to. 

The column \emph{$\partial$--curves} encodes the number of boundary components on the respective cusps; if there are $i$ boundary components on the (red) cusp 0, and $j$ on the (green) cusp 1, this is written as $i/j.$

Last, the corresponding boundary curves are given as they are computed from the chosen (oriented, topological) peripheral system. The boundary curves are determined by the signed intersection numbers of the peripheral elements with the spun-normal surface. We here summarise how this is computed; the details can be found in \cite[\S4.2]{defo} and \cite[\S3.1]{tillmann08-finite}. Let $\gamma$ be a closed simplicial path on a cusp cross section with respect to the triangulation induced by $\tri.$ The $Q$--modulus of a vertex with label $z_i$ is $q''_i-q'_i$, of a vertex with label $z'_i$ is $q_i-q''_i$, and of a vertex with label $z''_i$ is $q'_i-q_i$. The linear functional $\nu(\gamma)$ is defined to be the sum of the $Q$--moduli of all vertices of triangles touching $\gamma$ to the right. For our generators, we obtain the following linear functionals:
\begin{align*}
\nu (\m_0) &= q_1-q''_1+q'_2-q''_2-q_3+q'_3
&&\nu (\l_0^t) = -2q'_1+2q''_1 -2q'_2+2q''_2\\
\nu (\m_1) &= q_0-q'_0-q'_1+q''_1-q_3+q''_3
&&\nu (\l_1^t) = -2q'_0+2q''_0 -2q'_2+2q''_2
\end{align*}
The signs in the table respect the transverse orientations, which are relevant when the boundary curves of surfaces corresponding to linear combinations of the $N(F_i)$ are computed. 


\subsection{Some spun-normal surfaces}

Using the explicit description of normal surfaces, one can work out the topological type
and the position of normal surfaces in the manifold. The pictures of the
gluing pattern of some surfaces are shown in Figure~\ref{fig:whl_surfaces},
where the quadrilaterals and finitely many triangles are used to obtain
compact surfaces, along whose boundary components infinite normal annuli
have to be attached. The boundary components are drawn in the colour of the
corresponding cusp. Triangle coordinates are labelled by numbering the
vertices of the four tetrahedra from $0$ to $15$.
\begin{figure}[h!]
  \begin{center}
    \subfigure[]{
      \includegraphics[width=13cm]{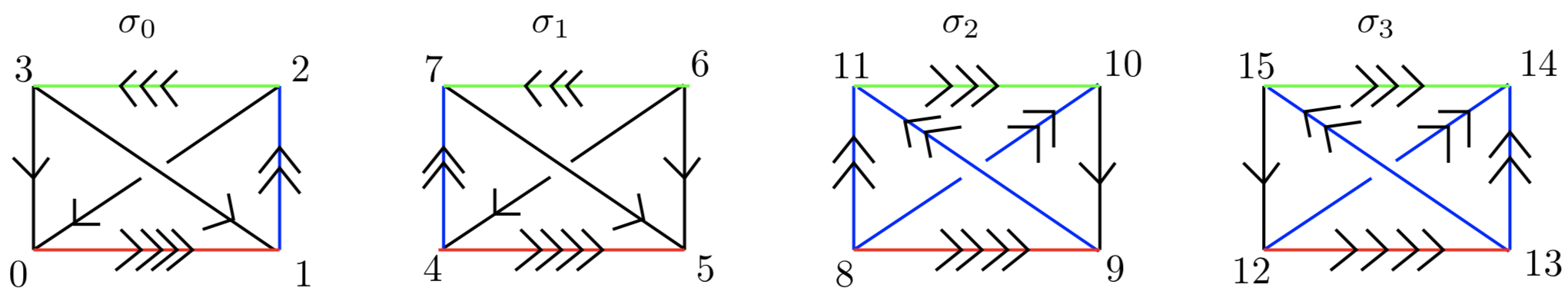}}
      \\
    \subfigure[The twice--punctured Klein bottle $N(F_{10})+N(F_{11})$]{
      \includegraphics[width=11cm]{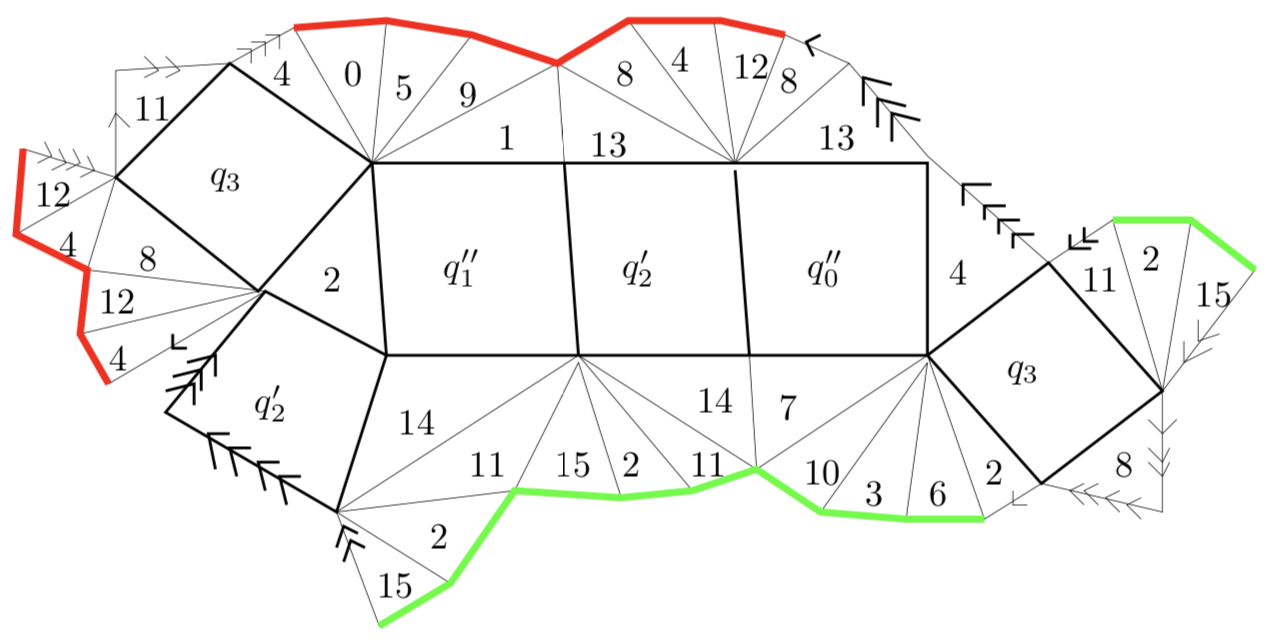}}
\\
\subfigure[The once--punctured Klein bottle $\frac{1}{2}(N(F_{13})+N(F_{18}))$]{
      \includegraphics[width=6cm]{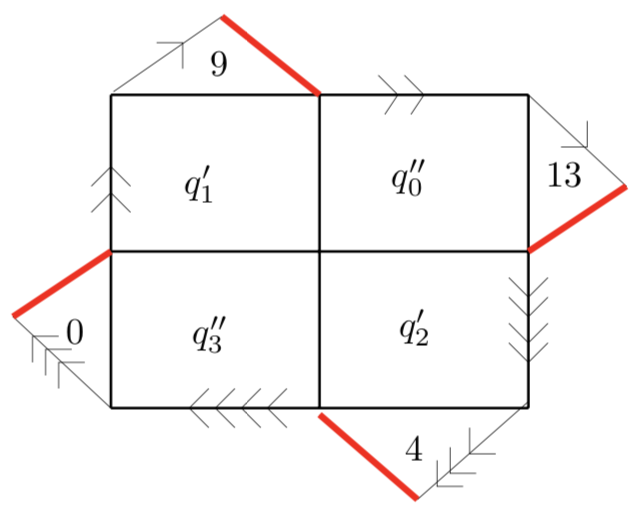}}
      \qquad
\subfigure[The twice--punctured torus $N(F_{13})+N(F_{18})$]{
      \includegraphics[width=7cm]{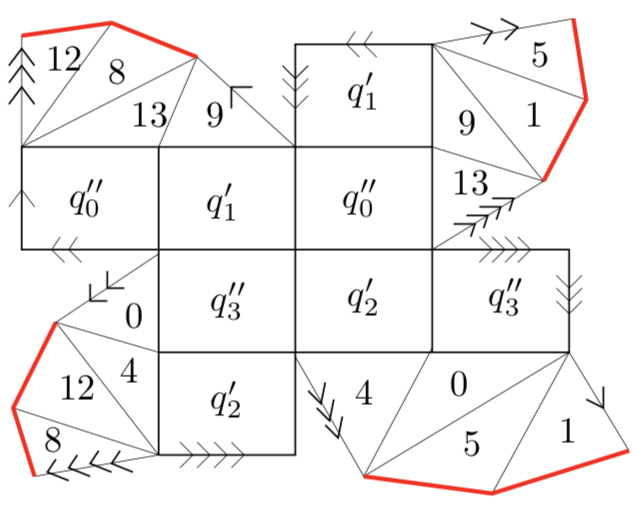}}
\end{center}
  \caption{Normal surfaces in the Whitehead link complement}
  \label{fig:whl_surfaces}
\end{figure}


\subsection{Equivalence classes}
\label{whl:Equivalence classes}

A homeomorphism $N\co \D_{\text{pre-}\infty}(\whl)\to \N(\tri_\whl)$ is given in \cite{defo} between $\N(\tri_\whl)$ and the tropical pre-variety $\D_{\text{pre-}\infty}(\whl)$ obtained from the canonical defining equations. There is an induced action of the group $D_4$ of symmetries of $\D(\whl)$ on $\D_{\text{pre-}\infty}(\whl).$ Since the elements of $D_4$ interchange coordinates, the induced action corresponds to interchanging coordinate triples of elements in $\D_\infty(\whl).$ Moreover, there is an induced action of $D_4$ on $\N(\tri_\whl)$ via the homeomorphism $N$ which again corresponds to interchanging coordinate triples of elements.

Indeed, $D_4$ can be identified with a group of symmetries of the triangulation. There is a Klein four group $K_f,$ identified with $\langle\tau_1, \tau_2\rangle,$
which stabilises the cusps, and orbits of the action of $K_f$ on $\N(\tri_\whl)$ give six equivalence classes amongst the vertex solutions in $\N(\tri_\whl)$:
\begin{align*}
N_1 &= \{ V_2,V_4 \},
&& N'_1 = \{ V_1,V_3 \},\\
N_2 &= \{ V_7,V_8,V_{11},V_{12}\}, 
&& N'_2 = \{ V_5,V_6,V_{9},V_{10}\},\\
N_3 &= \{ V_{15},V_{16},V_{19},V_{20}\},
&& N'_3 = \{ V_{13},V_{14},V_{17},V_{18} \}.
\end{align*}

\begin{figure}[t]
  \begin{center}
    \subfigure[Classes]{
      \includegraphics[width=5cm]{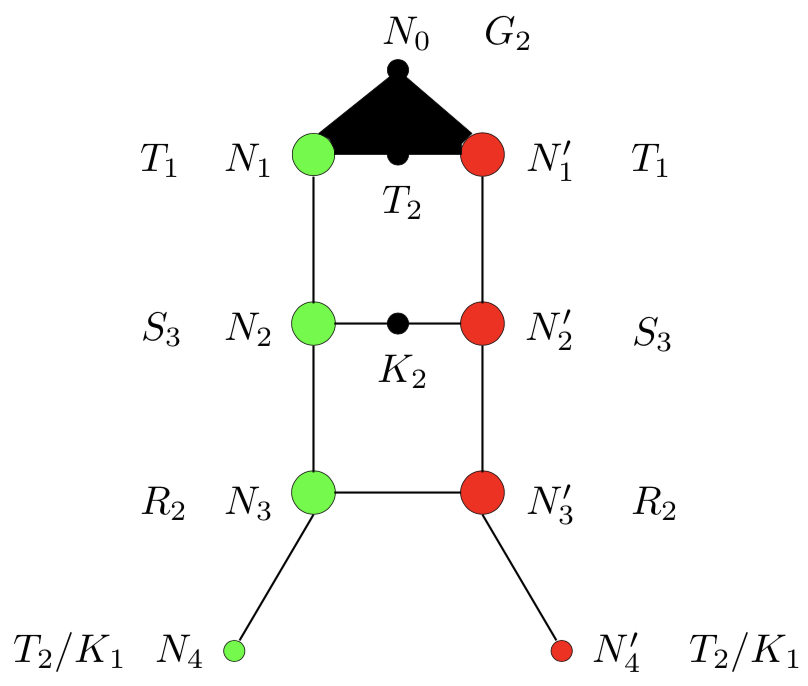}}
    \qquad
    \subfigure[$\partial$--curves]{
      \includegraphics[width=5cm]{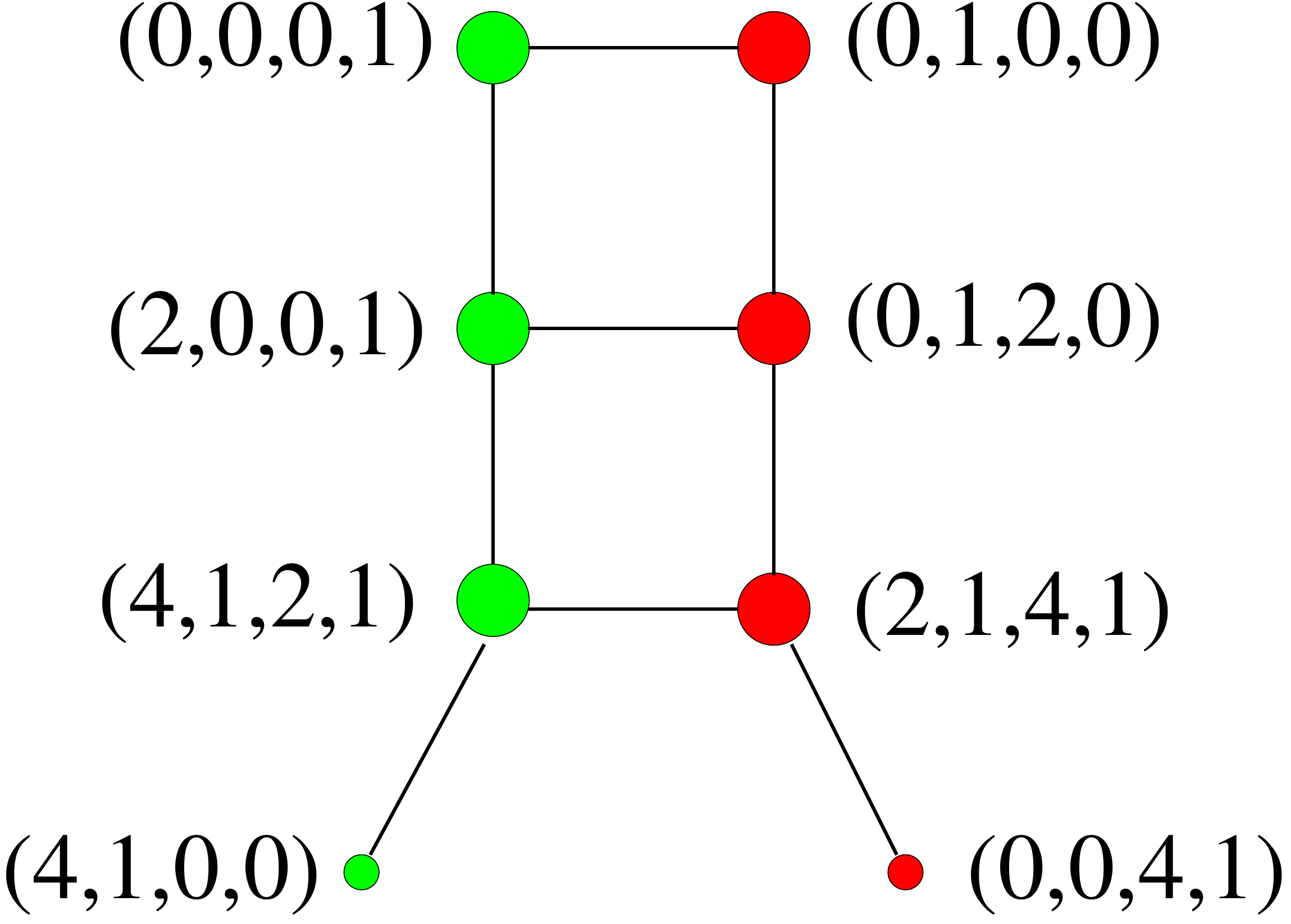}}
  \end{center}
  \caption{Surfaces in the Whitehead link complement}
  \label{fig:whl_results}
\end{figure}
The classification up to isotopy is given in \S\ref{sec:Incompressible normal surfaces} and \S\ref{incompressible after FH}.
The orbit under $K_f$ corresponds to the different ways a surface can ``spin into the cusps''. In particular, the (two or four) normal coordinates in each of the classes $N_k$, $N'_k$ correspond to isotopic surfaces. 

Members of the classes $N_k$ and $N'_k$ are interchanged by symmetries
interchanging the cusps. These symmetries correspond to the remaining
elements of $D_4.$ One can visualise the action of $D_4$ on
$\N(\tri_\whl)$ by considering the action of the dihedral group on Figure
\ref{fig:whl_admissible} induced by its standard action on a square. The
quotient by the action of $K_f$ is
pictured in Figure~\ref{fig:whl_results}(a), where the topological types of minimal representatives for some points are indicated. 

Note that there are arcs in $\N(\tri_\whl)$ connecting elements of $N_3,$ e.g.\thinspace the vertices $V_{16}$ and $V_{19}.$ The geometric sum $F_{16}+F_{19}$ is a twice--punctured torus. However, $\frac{1}{2}(N(F_{16})+N(F_{19}))$ is also an admissible integer solution, and the corresponding normal surface is a once-punctured Klein bottle. The corresponding equivalence class (i.e.\thinspace $K_f$ orbit) is denoted by $N_4.$ Similarly, for arcs in $\N(\tri_\whl)$ joining elements of $N'_3$ we obtain an equivalence class $N'_4$ whose elements are midpoints of these arcs.

The surface determined by a minimal integer solution corresponding to the point $V_0:= \frac{1}{2}(V_1+V_3) = \frac{1}{2}(V_2+V_4)$ in the ``centre" of $\N(\tri_\whl)$ is a genus two surface. This point is fixed by all symmetries, and we denote its equivalence class by $N_0.$ The quadrilateral spanned by $V_1, V_2, V_3, V_4$ is called the \emph{centre square}.

The $D_4$ orbit of a point $P \in \N(\tri_\whl)$ is now analysed. If $P=V_0,$ then its equivalence class only contains one element. If $P$ is contained in the square $[V_1,V_2,V_3,V_4]$ but not equal to its centre, then its $D_4$ orbit contains exactly four elements. Note that different elements in the square can have the same \emph{$\partial$--coordinate}
\begin{equation*}
(\nu_P(\l_0),-\nu_P(\m_0),\nu_P(\l_1),-\nu_P(\m_1)).
\end{equation*}
This is true for instance for $V_1$ and $\frac{3}{4}V_1 + \frac{1}{4}V_3.$ 

However, elements of the same $D_4$ orbit are distinguished by their $\partial$--coordinates. If $P$ is contained in $N_4$ or $N'_4,$ or is the midpoint of an arc $[V,W]$ in $\N(\tri_\whl)$ with $V \in N_i,$ $W \in N'_i$ and $i=2$ or $3,$ then its $D_4$ orbit contains exactly four elements. Moreover, the elements of the orbit are distinguished by their $\partial$--coordinates. If $P \in \N(\tri_\whl)$ is not contained in any of the sets considered above, then its $D_4$ orbit contains exactly eight elements, and all these elements are distinguished by their $\partial$--coordinates.

Moreover, the elements of an equivalence class of a point in $\N(\tri_\whl)$ have the same projectivised (i.e.\thinspace unoriented) $\partial$--coordinate. Thus, the set of projectivised $\partial$--coordinates arising from $\N(\tri_\whl)$ can be computed using the incidence structure amongst the equivalence classes, and the result is shown in Figure~\ref{fig:whl_results}(b). This shows that each equivalence class is uniquely determined by its projectivised $\partial$--coordinate unless its elements are contained in the centre square of $\N(\tri_\whl).$ 

\begin{lem} \label{whl:lem:slopes det surf}
An embedded spun-normal surface in $\whl$  is uniquely determined by its \emph{transversely oriented} boundary curves if its projectivised normal $Q$--coordinate is not contained in the centre square in $\N(\tri_\whl)$.
\end{lem}
\begin{proof}
Let $S$ be an embedded spun-normal surface. Then there is a unique $\alpha>0$ and a unique point $P \in \N(\tri_\whl)$ such that $N(S)=\alpha P.$ If $P$ is not contained in the centre square, then there is no other point in $\N(\tri_\whl)$ with the same $\partial$--coordinate. Thus, $S$ is uniquely determined by
\begin{align*}
&\alpha (\nu_P(\l_0),-\nu_P(\m_0),\nu_P(\l_1),-\nu_P(\m_1))\\
=& (\nu_{N(S)}(\l_0),-\nu_{N(S)}(\m_0),\nu_{N(S)}(\l_1),-\nu_{N(S)}(\m_1)).
\end{align*}
This proves the lemma.
\end{proof}


\begin{figure}[p]
  \begin{center}   
      \includegraphics[width=14cm]{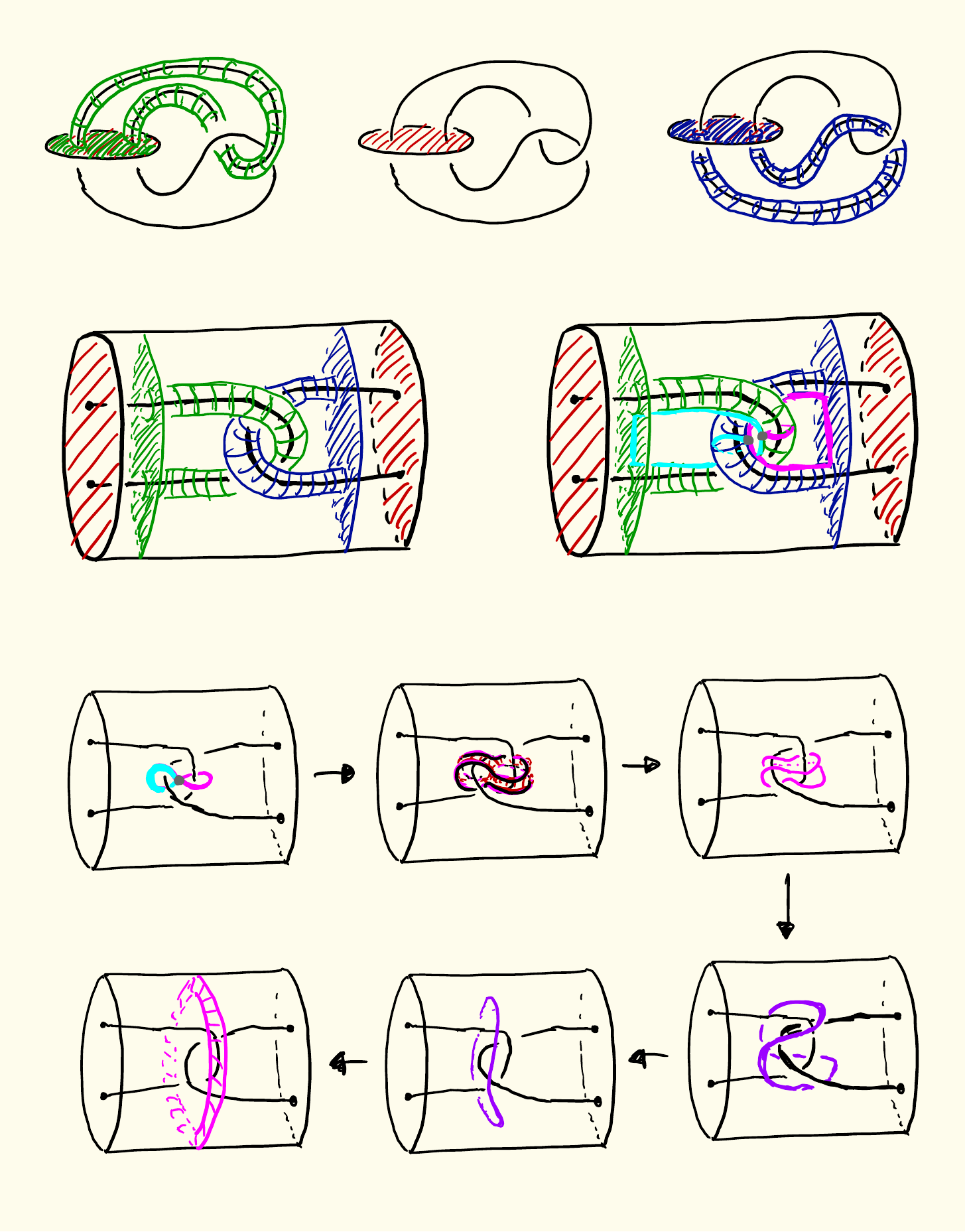}
  \end{center}
  \caption{Isotopic once-punctured tori. The reader will find it a pleasant exercise in normal surface theory to show that the once-punctured tori normalise to the same normal surface in the complement of a \emph{normal} thrice-punctured disc in $\tri_\whl.$}
  \label{fig:whl-isotopy}
\end{figure}


\subsection{Criteria for incompressibility}
\label{sec:Dunfield}

Dunfield and Garoufalidis~\cite{DuGa} give a simple criterion for a vertex surface to be incompressible, namely that it be a vertex surface with non-empty boundary which has a quadrilateral in each tetrahedron. However, this criterion does not apply to any of the vertex normal surfaces in the given ideal triangulation of the Whitehead link complement.

A certificate for incompressibility can be given using a different result, which the author learned from Nathan Dunfield. This is only summarised here, and complete details will be given in \cite{splittings}. A \emph{semi-angle structure} is an assignment $q \mapsto \alpha(q)$ of a non-negative real number (called angle) to each quadrilateral type $q$ in an ideal triangulation such that the sum of the angles associated to the three distinct quadrilateral types supported by each tetrahedron equals $\pi,$ and the sum of all angles of the quadrilateral types facing an edge in the 3--manifold equals $2\pi.$ See \cite{LuTi} for a detailed discussion of this viewpoint, We view $\alpha\in [0,\pi]^{\square} = [0,\pi]^{3n},$ where $\square$ is the set of all isotopy classes of quadrilateral discs.
We remark that if $\alpha\in (0,\pi)^{\square}$, then it is termed an \emph{angle structure}. This can be viewed as a linear hyperbolic structure and we refer the reader to the excellent exposition in \cite{FuGe} for a history and applications of angle structures. 

An embedded normal surface $S$ has normal $Q$--coordinate $N(S) \in \NN^\square = \NN^{3n}.$ Then $S$ is said to be \emph{dual to the semi-angle structure $\alpha$}, if $\alpha \cdot N(S) = 0,$ where the standard Euclidean inner product is taken. In other words, for each quadrilateral type $q$ that has non-zero weight in $S,$ we have $\alpha(q)=0$. 

\begin{thm}[Dunfield] 
Let $M$ be the interior of a compact 3--manifold with boundary a non--empty disjoint union of tori. Let $S$ be an embedded normal surface (possibly non--compact) in $(M; \tri)$ without any boundary parallel components. If $S$ is dual to a semi-angle structure of $(M; \tri)$, then it is essential.
\end{thm}

Here is a sketch of the proof. Suppose $S$ had a compression disc $D.$ Then $D$ can be put into a normal form relative to the triangulation and the normal surface, and inherits an induced semi-angle structure. The combinatorial Gau\ss-Bonnet formula then implies that $D$ has non-positive Euler characteristic, a contradiction. Hence the surface is incompressible. The proof is concluded with an observation by Hatcher that incompressible implies boundary incompressible for an embedded surface with boundary only on the torus boundary components of a 3-manifold. A proof of Dunfield's theorem is given in \cite{splittings}. 

The above theorem has the following simple corollary:

\begin{cor}\label{cor:dunfield criterion}
Let $M$ be the interior of a compact 3--manifold with boundary a non--empty disjoint union of tori, $\tri$ be an ideal triangulation of $M$, and $\alpha\in [0,\pi]^{\square}$ be a semi-angle structure. If $S$ and $F$ are compatible normal surfaces (possibly non--compact) that are both dual to $\alpha$, then each 2--sided normal surface that is a Haken sum of $S$ and $F$ is essential.
\end{cor}


\subsection{Essential normal surfaces}
\label{sec:Incompressible normal surfaces}

Listed in Table \ref{tab:whl:quad types} are a number of semi-angle structures on $\tri_\whl,$ as well as the vertex surfaces that are dual to these. In particular, each of the 20 vertex surfaces is essential. It follows from the information in the table and Corollary~\ref{cor:dunfield criterion} that each 2--sided normal surface in $\whl$ is essential except possibly those whose projectivised normal coordinates lie on the sides or the interior of the central square spanned by $V_1, V_2, V_3, V_4.$

Each surface $F_1, F_2, F_2, F_4$ is a once-punctured torus and can be given a natural orientation arising from the transverse orientation of its unique boundary curve. In this way, it represents an element $[F_i] \in H_2(\whl^c, \partial \whl^c; \Z),$ where $\whl^c$ is a compact core of $\whl.$ Note that $[F_3] = - [F_1]$ and $[F_4] = - [F_2].$
As in \cite{thu-norm}, one can now argue that the surfaces $S_i$ are norm minimising, and that each normal surface with projectivised normal coordinate along the boundary of the central square is norm minimising, and hence essential. In particular, we have the following:

\begin{obs}
The boundary of the central square naturally corresponds to the boundary of the unit ball of the Thurston norm. 
 \end{obs}

It is also shown in \cite{thu-norm} that every essential surface homologous to a fibre is isotopic to a fibre. Whence every fibre arises as a normal surface with respect to $\tri_\whl.$ This verifies a result of Kang and Rubinstein~\cite{KR-2015}.

We now claim that every normal surfaces not in the boundary of the central square is not a fibre. This can be seen either by considering their boundary slopes, or by determining their image in homology. All classes are mapped to zero except for the surfaces along segments between thrice punctured spheres and once-punctured tori. These cannot be fibres since they only evaluate non-zero on one cusp. We have thus established:

\begin{obs}
Every fibre in $\whl$ normalises in precisely two ways. A  spun-normal surface in $\whl$ is a fibre if and only if its projectivised normal coordinate lies in the interior of an edge of the central square.
\end{obs}

It is also not difficult to describe the once--punctured tori corresponding to the vertices of the central square. The surfaces $F_1$ and $F_3$ have 
identical projectivised boundary slopes, and so do $F_2$ and $F_4$. Constructing the corresponding normal surfaces explicitly, one can see
that for each of the solutions, there are four normal triangles which form an annulus on the surface and cap the quadrilaterals off at one of the cusps. The union of the two respective annuli on $F_1$ and $F_3$ is a boundary parallel torus. Similar for $F_2$ and $F_4$. It follows that 
each pair is obtained from a thrice punctured sphere that meets one of the boundary components in two meridians and the other in a longitude by tubing the two meridional boundary components together. The two choices of annulus give the two surfaces in each pair. However, these are in fact isotopic surfaces. The isotopy is indicated in Figure~\ref{fig:whl-isotopy}.

The central surface, $G_2,$ is a closed genus two surface linking the red and green edge, and hence compressible. To one side, it compresses to a torus linking the red cusp, and to the other to a torus linking the green cusp. Its complements are therefore compression bodies. The surface $G_2$ is made of up four quadrilaterals and no triangles. Any surface $F$ whose projectivised normal coordinate lies in the interior of the central square is a Haken sum of the form $F = n_1 S_i + n_2 S_j + n_3 G_2,$ where $S_i$ and $S_j$ correspond to vertices of an edge of the central square. In particular, this has the same boundary slope as the surface $n_1 S_i + n_2 S_j$. From the local structure around the green and red edges, we see that each such surface is compressible, namely there are $n_3$ compression discs, and after the compressions we have a surface in same homology class as $n_1 S_i + n_2 S_j$ and with same Euler characteristic, and hence it is a fibre and in fact isotopic to this. At this point, we may record the following:

\begin{obs}
A spun-normal surface in $\whl$ is essential if and only if its projectivised normal coordinate does not lie in the interior of the central square. 
\end{obs}

This completes the classification of projectivised essential normal surfaces up to normal isotopy.


\begin{figure}[p]
  \begin{center}
      \subfigure[$\FH$: all essential surfaces]{\label{fig:whl_fh_normal}
      \includegraphics[height=4.8cm]{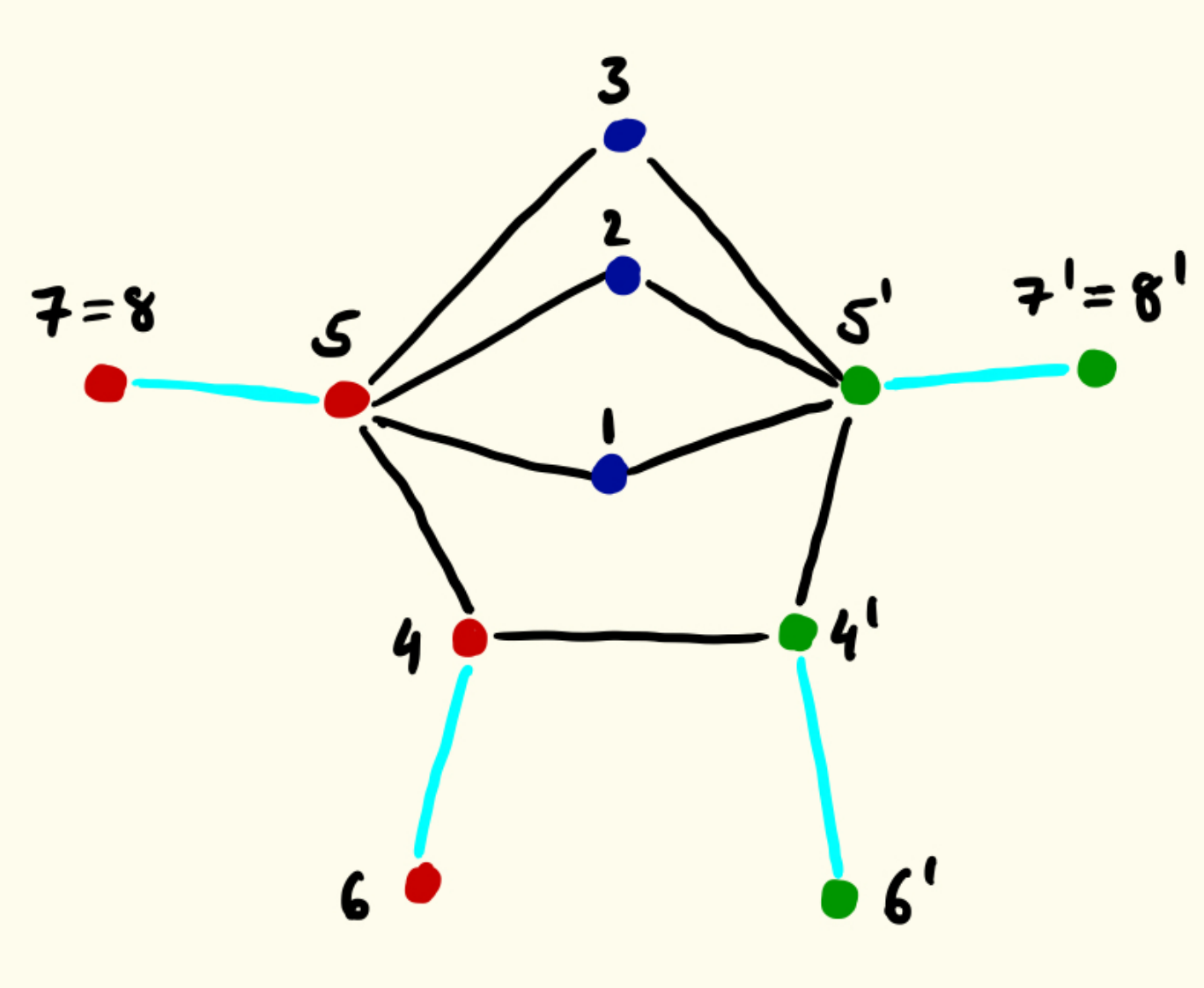}}
          \qquad
  \subfigure[Meridional incompressible surfaces]{\label{fig:whl_fh_mi}
      \includegraphics[height=4.8cm]{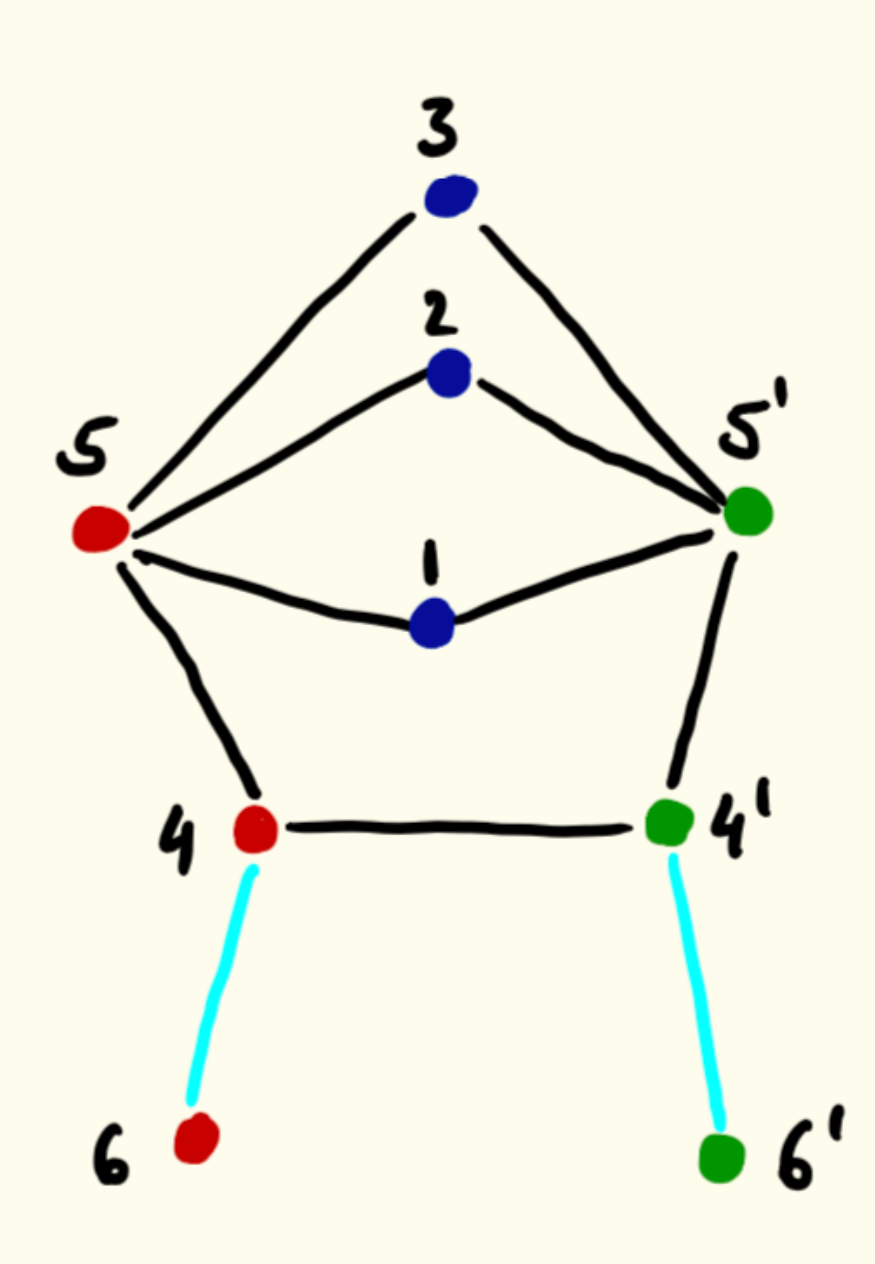}}
    \qquad
        \subfigure[Floyd-Hatcher]{\label{fig:whl_fh_ai}
      \includegraphics[height=4.8cm]{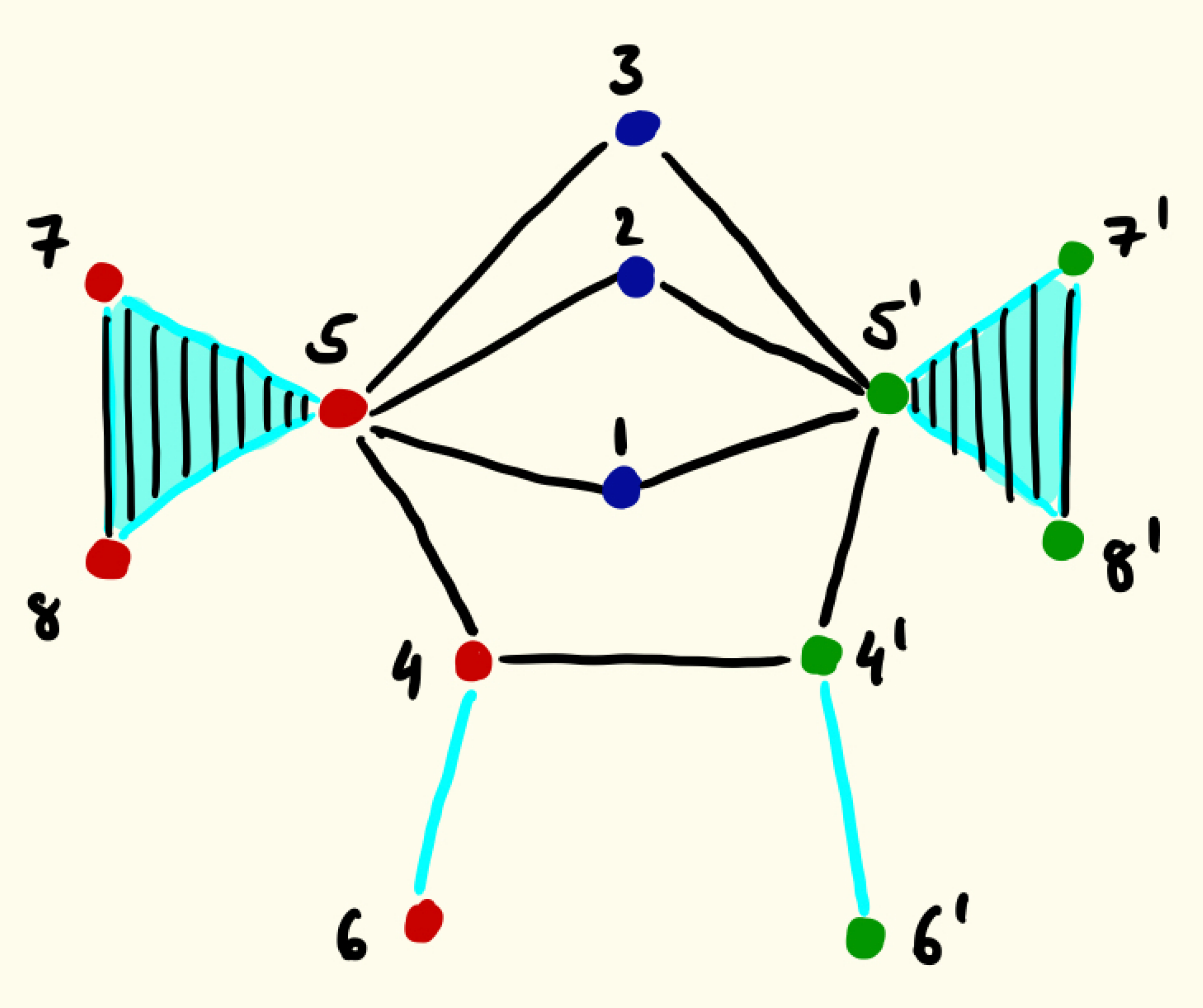}}
      \\
       \subfigure[Representative surfaces as shown in a pre-print of \cite{fh}. ]{\label{fig:whl_fh_surfs}
   \includegraphics[width=13cm]{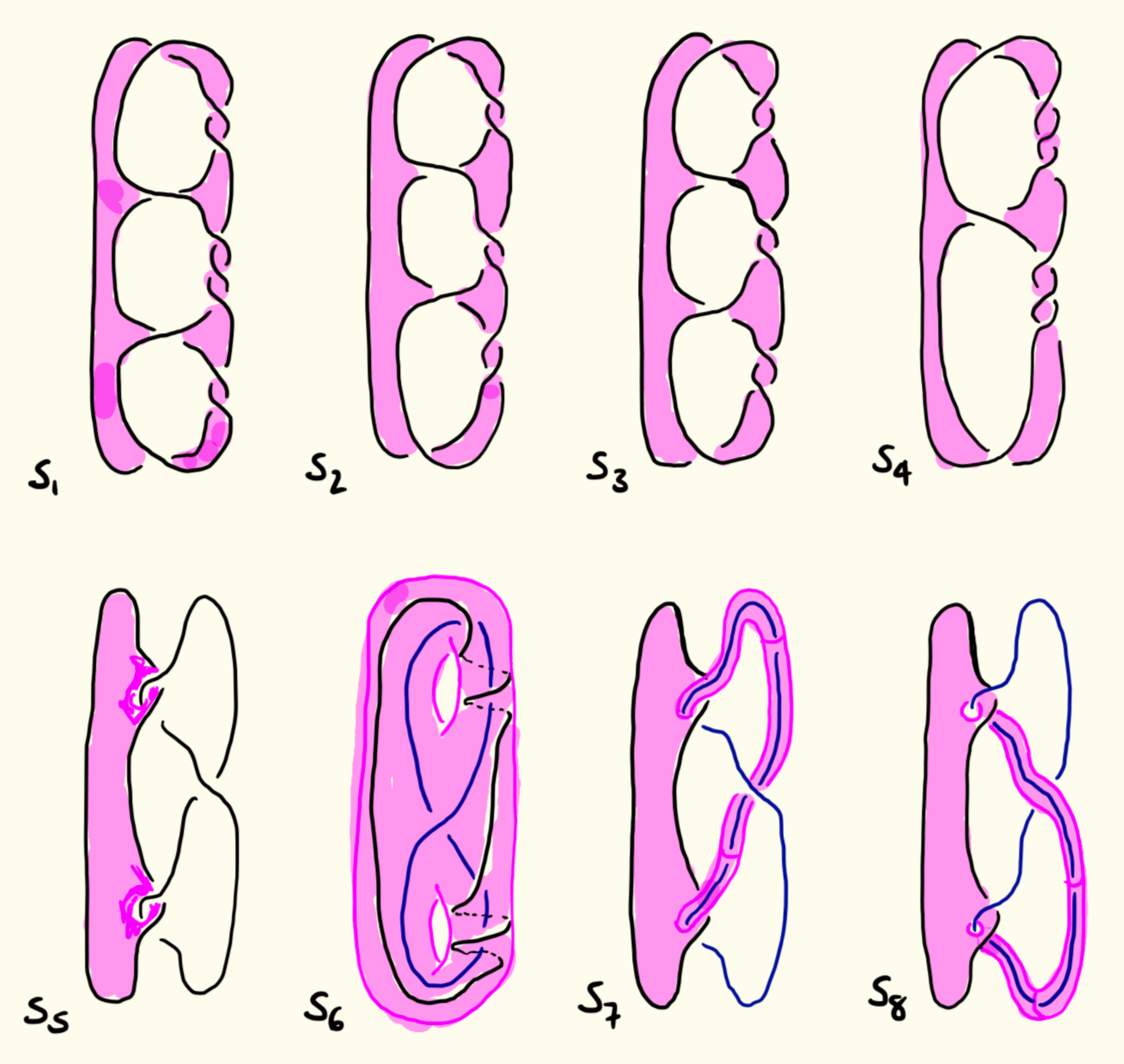}}      
      \end{center}
  \caption{Spaces of essential surfaces in the Whitehead link complement: These complexes have rational barycentric coordinates;
the bilateral symmetry corresponds to an involution exchanging the components of the Whitehead link; labels correspond to representative surfaces shown in Figure~\ref{fig:whl_fh_surfs}; and the vertical lines in the 2--simplices in \ref{fig:whl_fh_ai} indicate the same projective isotopy class; the blue 2-simplices and 1-simplices correspond to surfaces obtained from disjoint unions, whilst the black 1-simplices arise from normal sums or branched surfaces.}
  \label{fig:whl_fh}
\end{figure}


\subsection{Essential surfaces}
\label{incompressible after FH}

It follows from the main result of Walsh~\cite{wa} that every essential surface in $\whl$ that is not a fibre or a semi-fibre can be normalised with respect to $\tri_\whl.$ A semi-fibre is a separating surface that lifts to a fibre in a double cover. An inspection of the three double covers of $\whl$ shows that there is no such surface in $\whl.$ It is well known to experts that every non-compact essential surface that is not a fibre can be put into spun-normal form by choosing, for each cusp on which the surface has a cusp, a direction of spinning. A proof of this was recently announced by Kang and Rubinstein~\cite{KR-2015}. It follows that the space $\FH$ of projectivised classes of essential surfaces in $\whl$ is precisely the orbit space of the action of $K_f$ on the space of projectivised essential normal surfaces. This is shown in Figure~\ref{fig:whl_fh_normal}. Here, the labels of the vertices are with reference to Figure~\ref{fig:whl_fh_surfs}, which is taken from an earlier version of \cite{fh}, and the remaining surfaces can be determined from the symmetry interchanging the components.

Up to permuting the  link components, the surfaces in Figure~\ref{fig:whl_fh_surfs} have the following normalisations.
  The 
  pair of once-punctured tori $\{S_7, S_8\}$ normalises to the pair  $\{F_2, F_4\}$;   
  the twice-punctured tori $S_2$ and $S_3$ normalise to the pairs $\{ F_1+F_4, F_2+F_3\}$ and $\{F_1+F_2, F_3+F_4\};$
the thrice-punctured sphere $S_5$ normalises to the surfaces $\{F_7, F_8, F_{11}, F_{12}\}$;	
  the twice-punctured projective plane $S_4$ normalises to the surfaces $\{F_{15}, F_{16}, F_{19}, F_{20}\}$; 
  and the twice punctured torus $S_6$ normalises to the surfaces $\{F_{15}+F_{20}, F_{16}+F_{19}\}$.

The space of all incompressible surfaces in $\whl$ was described by Floyd and Hatcher~\cite{fh} using branched surfaces. They give an algorithm to compute a complex with rational barycentric coordinates that carries all essential surfaces that are \emph{meridional incompressible.} The condition for an essential surface $S \subset S^3 \setminus L$ to be meridional incompressible is as follows. If there is a disc $D\subset S^3$ with $D\cap S = \partial D$ and $D$ meeting $L$ transversely in one point in the interior of $D$, then there is a disc $D' \subset S\cup L$ with $\partial D' = \partial D,$ $D'$ also meeting $L$ transversely in one point. Geometrically, this corresponds to the existence of accidental parabolics. The complex carrying all meridional incompressible surfaces in the Whitehead link complement is shown in Figure~\ref{fig:whl_fh_mi} --- in \cite[Figure 5.4]{fh} this is the 1--complex with label [2,1,2] (the meaning of this label is irrelevant here).

The vertices of the 1--complex in Figure~\ref{fig:whl_fh_mi} are labelled $1, 2, 3, 4, 4', 5, 5', 6, 6'$, and we denote the 1--simplex with vertices $i$ and $j$ by $[i, j].$
It follows from \cite[Proposition 6.1]{fh} that the interior of each union $[5, 2] \cup [2, 5']$ and $[5, 3]\cup [3, 5']$ consists of fibres, and hence corresponds to a cone over a fibred face of the Thurston norm ball (after choosing coherent orientations of the surfaces representing the vertices).

An explanation of how to modify the algorithm of \cite{fh} to obtain the essential surfaces that do not satisfy the meridional incompressibility condition is given in  \S8 of \cite{fh}. The resulting complex carrying all essential surfaces in the Whitehead link complement according to \cite{fh} is the 2--complex shown in Figure~\ref{fig:whl_fh_ai}. This does not agree with our argument in \S\ref{sec:Incompressible normal surfaces} that the space $\FH$ of essential surfaces is as shown in Figure~\ref{fig:whl_fh_normal}. A possible explanation is that the isotopy between the once-punctured tori $S_7$ and $S_8$ was missed in \cite{fh}, because they are represented as non-isotopic surfaces in Figure~\ref{fig:whl_fh_ai}.


\subsection{Boundary curve space}

The boundary curve space can be computed from the complex $\FH.$ The surfaces carried by the segment $[5,3]$ in Figure~\ref{fig:whl_fh_mi} have the same boundary curves are those carried by $[5,2].$ Similar for $[3, 5']$ and $[2, 5'].$ All other surfaces are uniquely determined by their boundary curves. Moreover, the projectivised curves match up as shown in Figure~\ref{fig:whl_fh_bcn}. The dotted blue segments correspond to boundary curves  of surfaces that are not meridionally incompressible.

\begin{figure}[h]
  \begin{center}
       \subfigure[Boundary curve space]{\label{fig:whl_fh_bcn}
      \includegraphics[height=4.5cm]{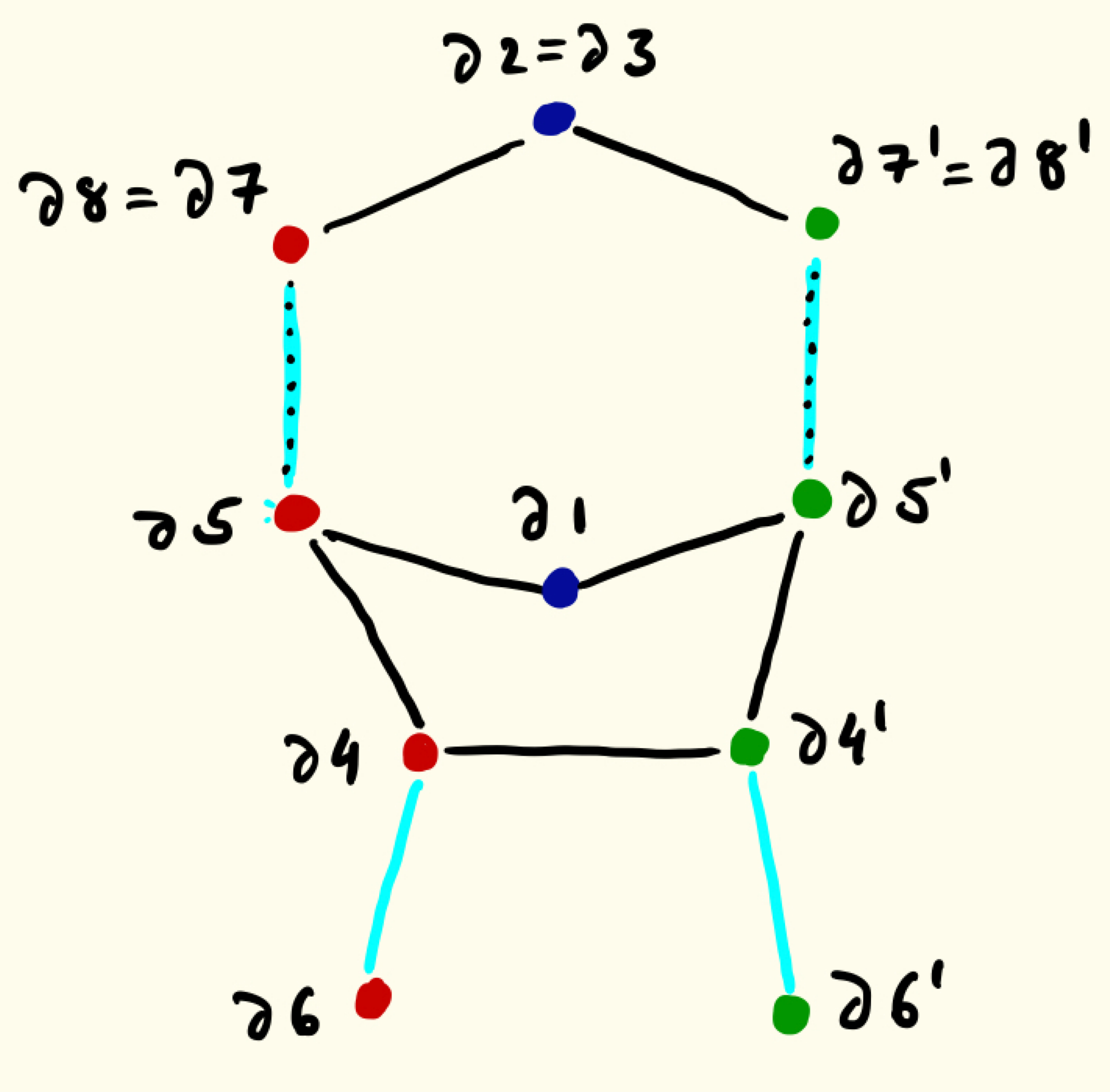}}
      \qquad
         \subfigure[meridional incompressible surfaces]{
      \includegraphics[height=4.5cm]{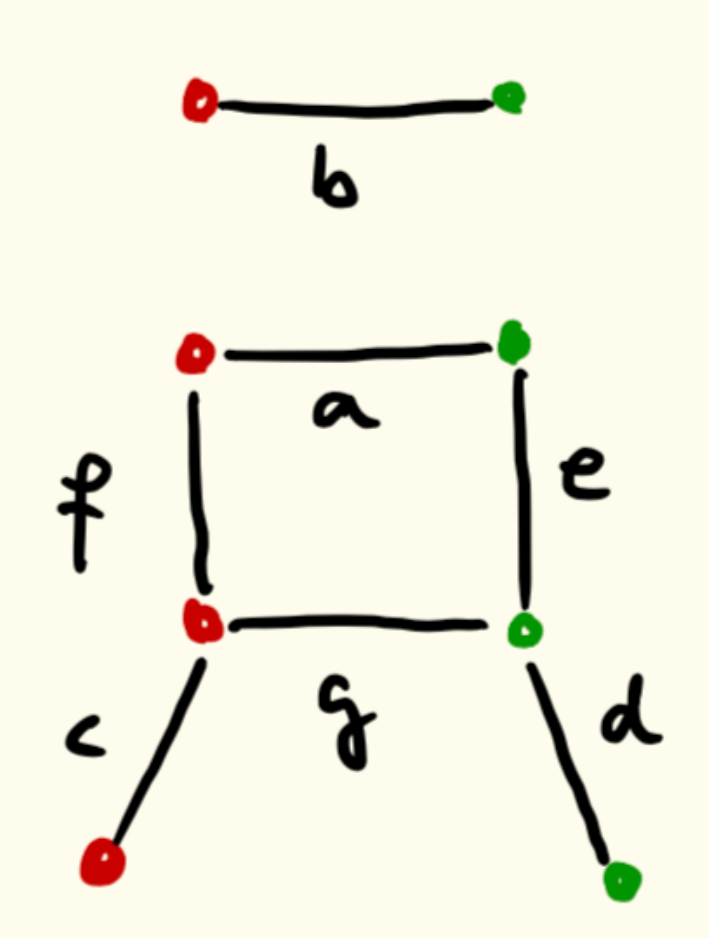}}
    \qquad
    \subfigure[all incompressible surfaces]{
      \includegraphics[height=4.5cm]{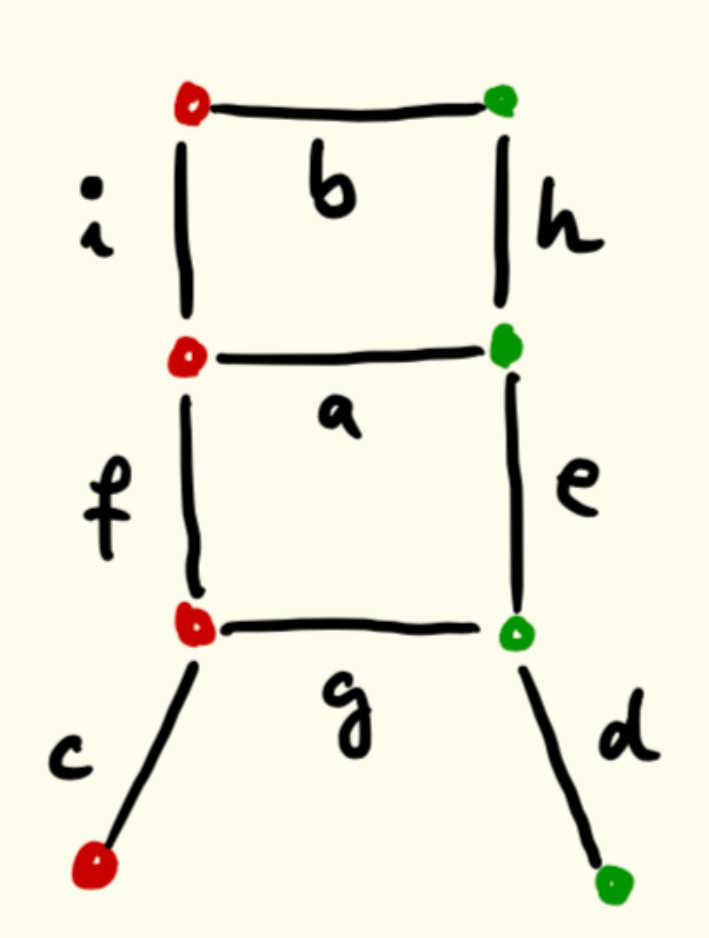}}
         \end{center}
  \caption{Boundary curve space}
  \label{fig:whl_fh_boundary}
\end{figure}

Using the notation of Lash~\cite{la}, the boundary curve space is described explicitly as:
\begin{align*}
& \{ (-2, t, -2t, 1) \mid t \in [0,\infty] \} && a \\
\cup & \{ (0, t, 0, 1) \mid t \in [0,\infty] \} && b\\
\cup & \{ (2t, t, 4, 1) \mid t \in (0,1) \} && c\\
\cup & \{ (4t, t, 2, 1) \mid t \in (1,\infty) \} && d\\
\cup & \{ (2t+2, t, 2t, 1) \mid t \in (0,1) \} && e\\
\cup & \{ (2, t, 2t+2, 1) \mid t \in (1,\infty) \} && f\\
\cup & \{ (-t+3, 1, t+3, 1) \mid t \in [-1,1 \} && g\\
\cup & \{ (-2t, 0, 0, 1) \mid t \in [0,1] \} && h\\
\cup & \{ (0, 1, -2t, 0) \mid t \in [0,1] \} && i\\
\end{align*}
Lash~\cite{la} only lists segments $a$--$g$ and we include $h$ and $i$ to account for the surfaces that are not meridional incompressible. We also changed the sign of the meridian coordinates in \cite{la} to account for a difference in conventions. We remark that 
Lash~\cite{la} and Hoste-Shanahan~\cite{hs} appear to compute only the space of projectivised boundary slopes of essential surfaces that are 
meridional incompressible, following only the main algorithm of Floyd and Hatcher. This explains why their spaces are not connected.


\section{Surfaces and boundary curves arising from degenerations}
\label{whl:surfaces arising from degenerations}

The logarithmic limit set of the deformation variety $\D(\tri_\whl)$ and the logarithmic limit set of the eigenvalue variety $\PEi_0(\whl)$ are determined in this section. Together with the results on essential surfaces of the previous section, these computations show that all strict boundary curves of $\whl$ are detected by the Dehn surgery component of the character variety. Since the boundary curves of fibres are detected by reducible characters, this shows that all boundary slopes of the Whitehead link are strongly detected. In addition, it is shown that $\D_\infty(\tri_\whl)$ has a fake ideal point, and some interesting degenerations are analysed in detail.


\subsection{The logarithmic limit sets}

The tropical variety $\D_\infty(\whl)$ is determined in Appendix~\ref{sec:log lim debbie} and shown in Figure~\ref{fig:whl_log_lim-D}. The above analysis of the rational maps and the normal surfaces together with the results concerning strongly detected boundary curves in \cite{defo} also determines the logarithmic limit set of $\PEi_0(\whl)$, which is shown in Figure~\ref{fig:whl_log_lim-E}. As discussed above, the isotopy relation is given by the action of the Klein four group $K_f,$ and the respective quotient spaces are shown in Figure~\ref{fig:whl_detected}.

\begin{figure}[h]
  \begin{center}
    \subfigure[Logarithmic limit set of $\D(\tri_\whl)$]{\label{fig:whl_log_lim-D}
      \includegraphics[width=6.1cm]{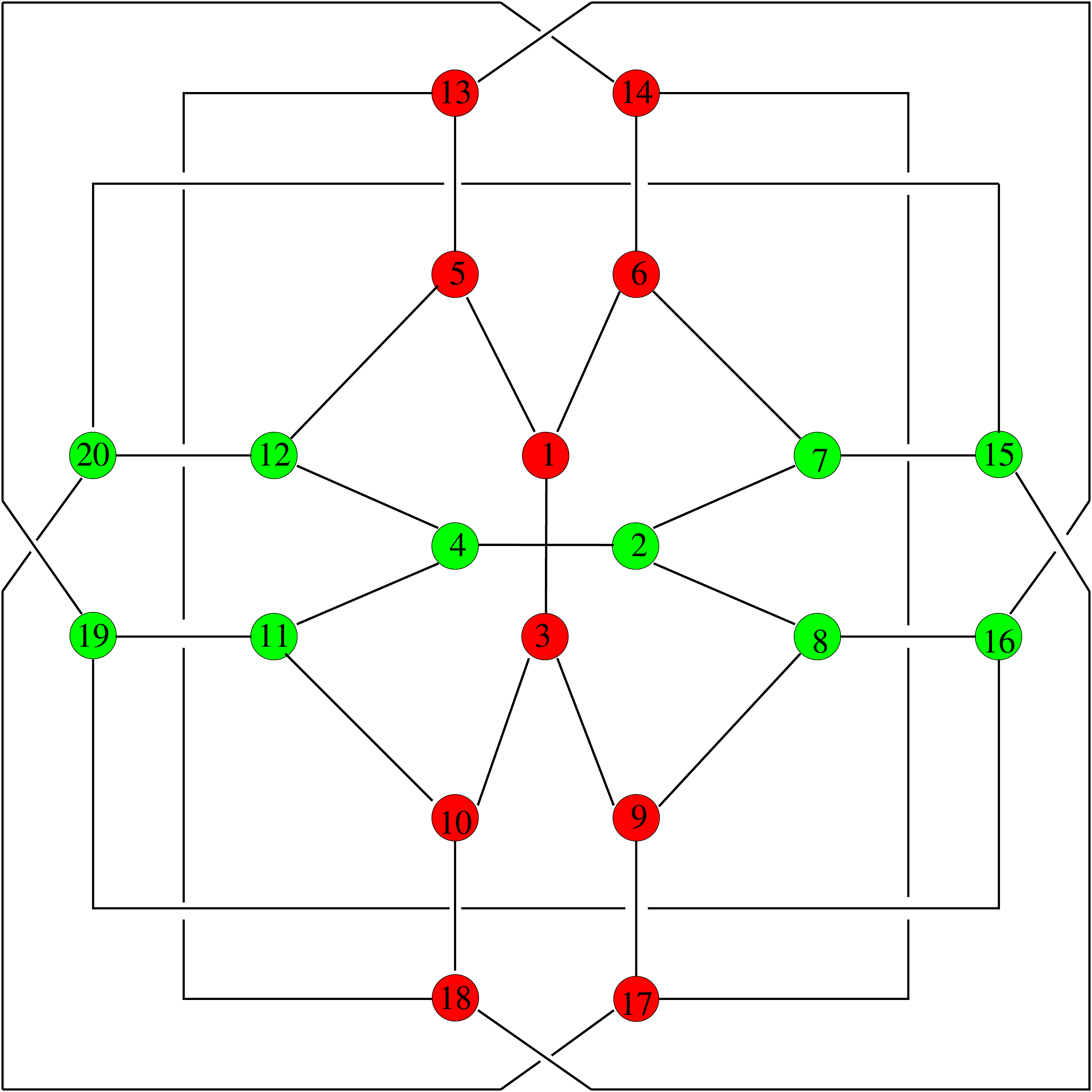}}
    \quad
    \subfigure[Logarithmic limit set of $\PEi_0(\whl )$]{\label{fig:whl_log_lim-E}
      \includegraphics[width=6.1cm]{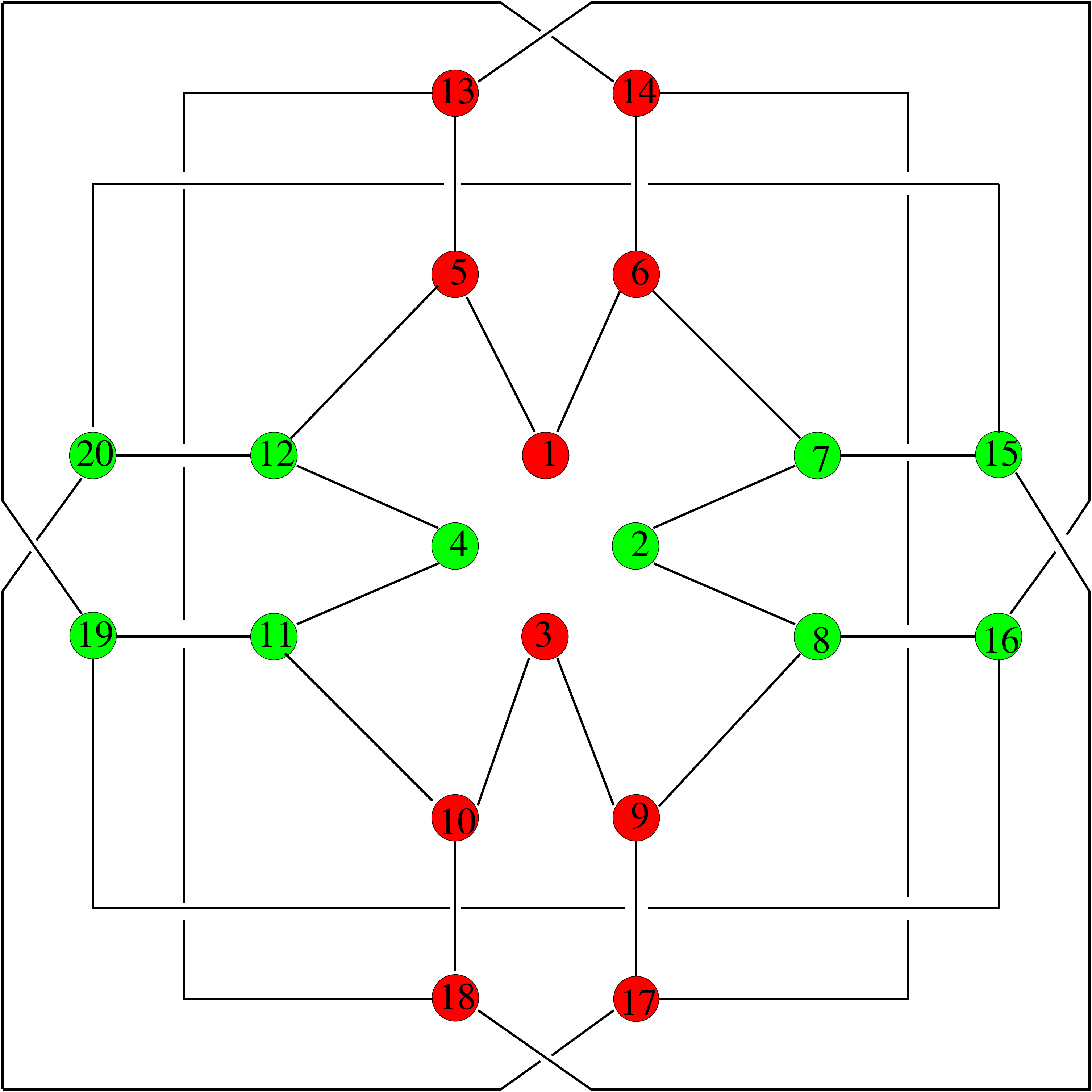}}
  \end{center}
  \caption{The logarithmic limit sets of $\D(\tri_\whl)$ and $\PEi_0 (\whl)$: The labels of the nodes correspond to the vertex solutions; the shown PL arcs between them can be realised as geodesics in $S^{11}$.}
  \label{fig:whl_log_lim}
\end{figure}

\begin{figure}[h]
  \begin{center}
    \subfigure[Detected surfaces]{
      \includegraphics[height=5cm]{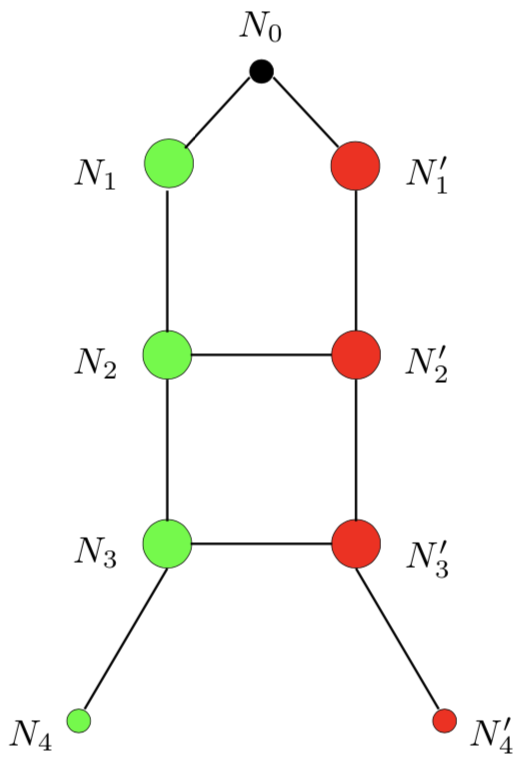}}
    \qquad
    \subfigure[Detected boundary curves]{
      \includegraphics[height=3.7cm]{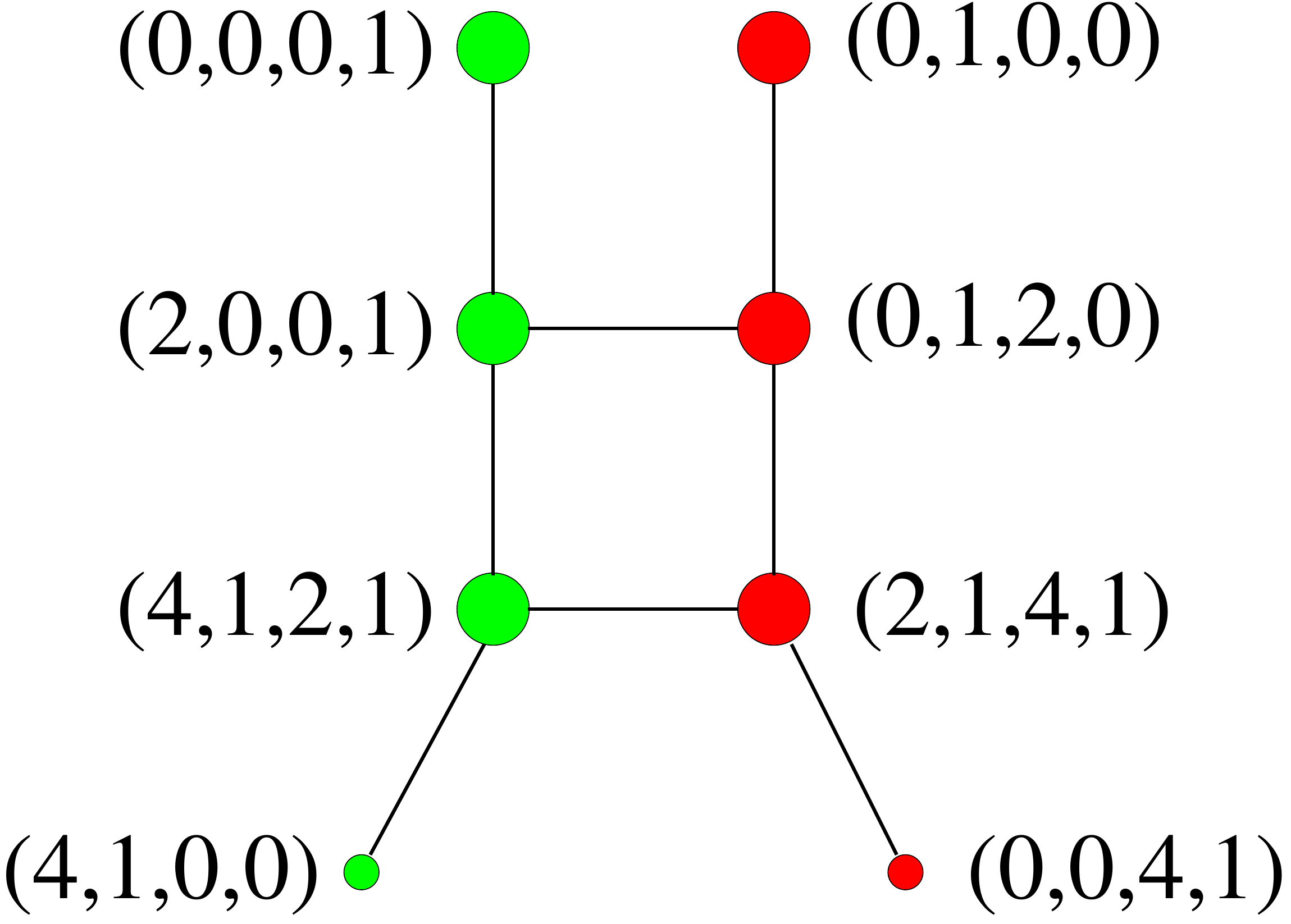}}
  \end{center}
  \caption{The logarithmic limit sets of $\D(\tri_\whl)$ and $\PEi_0 (\whl)$ modulo $\Z^2_2$}
  \label{fig:whl_detected}
\end{figure}

It follows from this calculation and Theorem~\ref{comb:essential prop} that all non-fibre essential surfaces in $\whl$ are detected by the deformation variety and the character varieties. Moreover, no fibre is detected by the deformation variety nor the geometric component in the character variety. We also note that 
\begin{enumerate}
\item there is a \emph{fake ideal point}, namely the centre $V_0$ of the central square, which corresponds to an ideal point of the deformation variety that maps to a finite point of the character variety and the associated surface compresses to two boundary parallel tori;
\item each ideal point that is in the interior of the segments $[V_1,V_0]$, $[V_2,V_0]$, $[V_3,V_0]$ or $[V_4,V_0]$ detects a surface that is not incompressible but can be compressed to an essential dual surface.
\end{enumerate}


\subsection{Detected slopes}

The relationship between ideal points of the $\PSL$--eigenvalue variety $\PEi(\whl)$ and ideal points of the deformation variety $\D(\whl)$ has already been discussed, and it follows that the logarithmic limit set of $\PEi_0(\whl)$ is as pictured in Figure~\ref{fig:whl_log_lim}(b). In particular, all these slopes arise from sequences of irreducible representations. This was already shown for the slopes of meridional incompressible surfaces by Lash~\cite{la}, and the present computation extends this to the boundary slopes of all essential surfaces that are not fibres.

\begin{obs}
Let $S \subset \whl$ be an essential surface that is not a fibre. Then the boundary slopes of $S$ are strongly detected by a sequence of irreducible representations.
\end{obs}

It is well known that the boundary slopes of fibres are strongly detected by sequences of reducible representations. The slopes detected by both reducible are irreducible representations are precisely the slopes of the once-punctured tori in $N_1$ and $N'_1.$
There are sequences of shape parameters which converge to ideal points corresponding to elements of $N_1$ and $N'_1$ giving rise to sequences of irreducible representations in the character variety since the eigenvalue of a longitude is constant equal to $-1$ throughout the degeneration. There also are sequences of reducible representations
detecting the same slopes: 
\begin{equation*}
\rho_m (\m_i) = \begin{pmatrix} 1 & 1 \\ 0 & 1 \end{pmatrix}
\qquad \text{and} \qquad
\rho_m(\m_j) = \begin{pmatrix} m & 0 \\ 0 & m^{-1} \end{pmatrix},
\end{equation*}
where $m \to 0$ or $m \to \infty.$


\subsection{Some interesting degenerations}

Recall that $\D'(\tri_\whl)$ is defined by
\begin{align} \label{whl:det det}
1 = wxyz \quad \text{and }\quad
wx(1 - y)(1 - z) = (1 - w)(1 - x)yz,
\end{align}
and that for any $w,x,y,z \in \C -\{0, 1\}$ subject to these equations,
there is a unique point on $\D (\whl )$. We will now describe sequences of
points in $\D'(\tri_\whl)$ which correspond to well--defined sequences in
$\D (\whl )$ approaching the desired ideal points.

1. The point $(w,x,y,z) = (w, -w^{-1},w,-w^{-1})$ satisfies equations
(\ref{whl:det det}) for any $w \in \C-\{0\}$, and if $w$ has positive
imaginary part, it defines a positively oriented triangulation of $\whl$,
and in particular an incomplete hyperbolic structure on $\whl$. This implies
that the triangulation can be deformed through positively oriented
tetrahedra from the complete structure where $w=i$ to the ideal points where
$w\to 0$ or $w\to \pm 1$. These points correspond to $V_1$
and $V_3$ for $w\to 1$ and $w\to -1$ respectively, and to
$\frac{1}{2}(V_{16}+V_{19})$ for $w \to 0$, since the growth rates of all
parameters are equal.

We give one example, where we scale to integer coordinates. As $w \to 0$,
the ideal point 
\begin{equation*}
\xi = (-1,0,1,1,-1,0,-1,0,1,1,-1,0)
\end{equation*}
is approached. The corresponding normal $Q$--coordinate is
\begin{gather*}
N(\xi) = (0,1,0,0,0,1,0,1,0,0,0,1),\\
\text{whilst}\quad
\frac{1}{2}(N(S_{16})+N(S_{19})) = (0,1,0,0,0,1,0,1,0,0,0,1),
\end{gather*}
which we can rescale to $\frac{1}{2}(V_{16}+V_{19})$.

Moreover, one easily obtains the limiting eigenvalues. As $w\to 0$ there are
the following power series expansions for the resulting eigenvalues: 
\begin{align*}
\s &= 1, 
&& t = -1,\\
\u &= \frac{1+w}{w(1-w)} = \frac{1}{w} + 2 \sum_{i=0}^\infty w^i, 
&& v = w^2.
\end{align*}
The first two equations reflect the fact that one cusp is complete, and
comparing the growth rates at the second cusp (remembering that $\u$ is the
square of an eigenvalue, whilst $v$ is an eigenvalue) gives a point in
boundary curve space with coordinates $[0,0,4,1]$.

2. For any $w \in \C -\{0, \pm 1\}$, the point
\begin{equation*}
(w,x,y,z) = (w,\frac{1-w}{1+w}, -\frac{1+w}{1-w} ,-w^{-1}) \in \D'(\tri_\whl)
\end{equation*}
gives a point on $\D (\whl )$.
Note that the associated triangulation involves either only flat tetrahedra
or both positively and negatively oriented ones. One may use power series
expansions to show that a detected surface for $w\to 0$ is $S_8$:
\begin{align*}
x = \frac{1-w}{1+w} = 1 + 2 \sum_{i=1}^\infty (-1)^i w^i,&&
y = -\frac{1+w}{1-w} = -1 - 2 \sum_{i=1}^\infty w^i,&&
z = -\frac{1}{w}.
\end{align*}
Thus, as $w\to 0$, the ideal point
\begin{equation*}
\xi = (-1,0,1,0,1,-1,0,0,0,1,-1,0)
\end{equation*}
is approached. The corresponding normal $Q$--coordinate is
\begin{equation*}
N(\xi) = (0,1,0,1,0,0,0,0,0,0,0,1),
\end{equation*}
which coincides with the normal $Q$--coordinate of $S_8$, and may be
rescaled to $V_8$.

3. Detecting elements of $N_3 \cup N'_3$ is a little more involved. Let
\begin{align*}
w &= \frac{-1 - \eps^3 + \sqrt{5 - 4 \eps + 4 \eps^2 - 2 \eps^3 +
\eps^6}}{2(1+\eps^2)}
&& x = \eps\\
y &= -\eps \frac{2(1+\eps^2)}{-1 - \eps^3 + \sqrt{5 - 4 \eps + 4 \eps^2 - 2
\eps^3 + \eps^6}}
&& z = -\eps^{-2}
\end{align*}
Formal substitution of these assignments in the defining equations of the
deformation variety shows that whenever the expressions are defined for any
small, non--zero $\eps$, they determine a point of $\D (\whl )$.
Since $w$ is differentiable at $\eps=0$ and takes the value
$w(0) = \frac{-1\pm \sqrt{5}}{2}$, it follows that $w$ and $y$ have
converging power series expansions at $\eps=0$, and $y$ has a zero of order
one at $\eps=0$. The first few terms are:
\begin{align*}
w(\eps) &= \frac{-1\pm \sqrt{5}}{2} - \frac{1}{\sqrt{5}}\eps + 
        \big(1 - \frac{17}{5\sqrt{5}}\big) \eps^2 
	+ \big(-3 + \frac{123}{25\sqrt{5}}\big)\eps^3
+{\ldots} \\
y(\eps) &= - \frac{2}{-1+\sqrt{5}} \eps -
\frac{8}{\sqrt{5}(-1+\sqrt{5})^2}\eps^2 +
{\ldots}
\end{align*}
Thus, as $w\to 0$, we approach
\begin{equation*}
\xi = (0,0,0,-1,0,1,-1,0,1,2,-2,0).
\end{equation*}
A detected surface is therefore $S_{13}$.

4. Let $(w,x,y,z) = (w, w^{-1},w^{-1}, w)$.
Then $w \to 1$ approaches an ideal point corresponding to $V_0$.
The chosen degeneration is through triangulations involving
positively and negatively oriented tetrahedra, or triangulations which are
entirely flat. In fact, there cannot be a degeneration through positively
oriented tetrahedra to this ideal point, since the hyperbolic gluing
equation $1 = w x y z$ would imply that throughout this degeneration
$\arg (w) + \arg (x) + \arg (y) + \arg (z) = 2 \pi$, whilst for each 
parameter the argument converges to zero.

Experimentation with {\tt SnapPy }\rm suggests that elements from all but
the fourth set can be approached through degenerations only involving
positively oriented tetrahedra. Using the projection of the right--handed
Whitehead link, the following surgery coefficients can be approached through
degenerations involving positively oriented triangulations: 
\begin{itemize}
\item $(\infty) ,(4,1)$. Four tetrahedra degenerate and the splitting
surface is a 1--sided Klein bottle. 
\item $(\infty) ,(0,1)$. Two tetrahedra degenerate and two become flat
with parameters equal to $-1$. The splitting surface is a 2--sided torus.
\item $(4,0),(0,2)$. Three tetrahedra degenerate, the remaining tetrahedron
has shape $\frac{1}{2} + \frac{1}{2} i$. The limiting orbifold has volume
approximately $0.9159$. ({\tt SnapPy }\rm does not compute a splitting surface if all cusps have
been filled.)
\item $(8,2),(4,2)$. Three tetrahedra degenerate, the remaining tetrahedron
has shape $1+i$, and the limiting orbifold has volume approximately $0.9159$. (Dito concerning the splitting surface.)
\end{itemize}

These examples will be put in a general framework in \cite{splittings}.


\appendix

\section*{Appendices} 


\section{Manifold data}
\label{sec:snap data}

The ``gluing and completeness'' data obtained from \tt{SnapPy }\rm is given in Table \ref{tab:snap data whl}. The relationship to the shape parameters used here is: $w = z_1',$ $x = z_2'',$ $y = z_3''$ and $z = z_4''.$ From this, one can verify that the holonomies of the peripheral elements given by
\tt{SnapPy }\rm are the inverses of the holonomies of the meridians and geometric longitudes given above, where \tt{SnapPy}\rm's cusp $0$ corresponds to the green cusp (here:\thinspace cusp 1), and cusp $1$ to the red cusp (here:\thinspace cusp 0).
\begin{table}[h]
\begin{center}
\begin{small}
\begin{tabular}{| l | r r r r | r r r r | r| }
\hline
   &$z_1$&$z_2$&$z_3$&$z_4$&$1-z_1$&$1-z_2$&$1-z_3$&$1-z_4$&\\
\hline
$H'(\m_0)$& 1& 0& --1& 0& --1& 0& 1& 1& 0\\ 
$H'(\m_1)$& 0& 0& 0& 1& 1& --1& 0& 0& 0\\
$H'(\l_0)$&1& 0& 0& 0& 0& 1& --1& 1& 0\\
$H'(\l_1)$&--1& 0& 0& 0& 0& --1& --1& 1& 0\\ 
\hline
$e_1$&1& 1& 1& 1& 1& --2& 0& 0& --1\\
$e_2$&0& --1& --1& --1& --1& 1& 1& 1& 1\\
$e_3$&--1& 1& 1& 1& 1& 0& --2& --2& --1\\
$e_4$&0& --1& --1& --1& --1& 1& 1& 1& 1\\
\hline
\end{tabular}
\end{small}
\end{center}
\caption{\tt{SnapPy}\rm's gluing matrix for $m129$}
\label{tab:snap data whl}
\end{table}


\section{Eigenvalue variety}
\label{sec:Eigenvalue variety}

\subsection{Eigenvalue variety}
\label{whl:eigenvalue variety}
\label{whl:compute eigen}

Denote the component in $\Ei (\whl )$ arising from reducible
representations by $\Ei^r (\whl ).$ We have
$\Ei_0 (\whl ) = \overline{\Ei (\whl ) - \Ei^r (\whl )},$ where the overline
denotes the Zariski closure. With respect to the chosen affine coordinates,
we have $\Ei^r (\whl ) = \{ t = v = 1\},$ which is 2--dimensional. 
$\Ei_0 (\whl )$ is 2--dimensional since it is the closure of the image of
the eigenvalue map.  From this, it also follows that the intersection of
$\Ei_0 (\whl )$ and $\Ei^r (\whl )$ is a union of two lines:
\begin{equation*}
\Ei^r(\whl ) \cap \Ei_0(\whl ) = \{ s = t = v = 1\} \cup \{ t = u = v = 1\}.
\end{equation*}
The following calculations use well known elimination and extension theorems which can be found in \cite{clo}. Consider $\Ta (\whl)$ and the eigenvalue map $\ee$ (see equation (\ref{whl:eigenvalue map})). The rational functions determining the eigenvalues of the longitudes are: 
\begin{align*}
t =&\vartheta_{\l^t_0}(\rho ) 
  =s^{-2}-s^{-2}u^2+u^2
   + c(2 s^{-1}u + s^{-3}u^{-1}-s^{-3}u) + c^2 s^{-2},\\
v =&\vartheta_{\l^t_1}(\rho )  
  =u^{-2}-u^{-2}s^2+s^2
   + c(2 u^{-1}s + u^{-3}s^{-1}-u^{-3}s) + c^2 u^{-2}.
\end{align*}
Then $s^2 t - u^2 v = s^2-u^2 + c(su^{-1}-s^{-1}u).$ Since $\Ta (\whl )$ is not contained in the hyperplane $s^2=u^2,$ this shows that the map $\ee: \Ta (\whl ) \to \Ei_0 (\whl)$ has degree one, and also determines an inverse mapping:
\begin{equation}\label{whl:eigen to taut}
(s,t,u,v) \to \bigg(s,u, \frac{s^2(t -1) + u^2 (1-v)}{su^{-1}-s^{-1}u}\bigg).
\end{equation}

Recall the defining equation (\ref{whl:taut_rel}) of $\Ta (\whl)$:
\begin{align*}
f_1 &= (s-s^{-1})(u-u^{-1}) 
         + c(s^{-2}u^{-2}-u^{-2}-s^{-2}+4-s^2-u^2+s^2u^2)\\
    &\qquad +c^2(2s^{-1}u^{-1}-su^{-1}-s^{-1}u+2su) + c^3.
\end{align*}
Two additional polynomials $f_2 = s^3 u t + {\ldots} $ and $f_3 = s u^3 v + {\ldots} $ are obtained from the above expressions for $t$ and $v.$ The only variable to be eliminated is $c,$ and the leading coefficients of $c$ in $f_1,$ $f_2$ and $f_3$ are monomials in $\C[s,u].$ The elimination is done using resultants (see \cite{clo} for details). Since $s$ and $u$ are units, it follows that the eigenvalue variety is defined by $Res(f_1,f_2,c) = Res(f_1,f_3,c) = Res(f_2,f_3,c)=0.$ Eliminating redundant factors from the resultants gives the following set of defining equations for $\Ei_0 (\whl )$:
\begin{align*}
h_1 &=t-s^2t+s^2t^2-s^4t^2-u^2-2s^2tu^2+s^4tu^2-t^2u^2+2s^2t^2u^2\\ 
&\qquad +s^4t^3u^2+tu^4-s^2tu^4+s^2t^2u^4-s^4t^2u^4,\\
h_2 &= s^2-v-s^4v+u^2v+2s^2u^2v+s^4u^2v-s^2u^4v+s^2v^2-u^2v^2\\ 
&\qquad -2s^2u^2v^2-s^4u^2v^2+u^4v^2+s^4u^4v^2-s^2u^4v^3,\\
h_3 &= s^4t-s^6t-s^2tu^2+s^4tu^2+s^6t^2u^2-s^2u^2v\\
&\qquad +u^4v+s^2u^4v-2s^4tu^4v-u^6v+s^2u^6v^2.
\end{align*}

The mapping (\ref{whl:eigen to taut}) is not defined when $s^2=u^2.$ One may now compute the defining equations of the subvariety of $\Ei_0(\whl)$ on which it is not defined:
\begin{align*}
s^2 = u^2, && t = v, && 
0 = -t + s^2(1+t)+s^4t(1-t)-s^6t^2(1+t)+s^8t^2.
\end{align*}

With a little more effort, one can compute the following inverse mappings defined on two open sets, which cover all but eight points of the eigenvalue variety: 
\begin{align}
\label{whl:ev: inv sl 1}\ee^{-1}_1: \Ei_0(\whl ) \to \Ta (\whl ) && 
(s,t,u,v) &\to \Bigg(s,u,\frac{(s^2-v)(1-u^2)}{su(1+v)}\Bigg)\\
\label{whl:ev: inv sl 2}\ee^{-1}_2: \Ei_0(\whl )\to \Ta (\whl ) && 
(s,t,u,v) &\to \Bigg(s,u,\frac{(u^2-t)(1-s^2)}{su(1+t)}\Bigg)
\end{align}
The points corresponding to the complete structure are always singularities of the eigenvalue variety --- here determined by the points where $t=v=-1$ and $s^2 = u^2 =1.$ The other points of $\Ei_0 (\whl )$ where neither of the above maps are defined are subject to $t=v=-1$ and $s^2 = u^2 = -1.$

The defining equations for the $\PSL$--eigenvalue variety can be worked out
from $\PTa (\whl )$ similarly to the above, or from the results of the
previous subsection using Lemma \ref{whl: eigenvalue maps lem}. In
particular, one has:
\begin{align*}
t =&\vartheta_{\l^t_0}(\prho_{GL} )  
  =\s^{-1}-\s^{-1}\u+\u
   + d(2 \s^{-1} + \s^{-2}\u^{-1}-\s^{-2}) + d^2 \s^{-2}\u^{-1},\\
v =&\vartheta_{\l^t_1}(\prho_{GL} )  
  =\u^{-1}-\u^{-1}\s+\s
   + d(2 \u^{-1} + \u^{-2}\s^{-1}-\u^{-2}) + d^2 \u^{-2}\s^{-1}.
\end{align*}
Then $\s\u^2 v - \s^2\u t  = \s\u^2-\s^2\u+d(\u-\s)$ implies that the
degree of the map $\PTa (\whl ) \to \PEi (\whl )$ is equal to one.
It follows from the above discussion that there are only two points where
neither of the following inverse maps is not defined:
\begin{align}
\label{whl:ev: inv psl 1}\pee^{-1}_1: \PEi_0(\whl ) \to \PTa (\whl ) && 
(\s,t,\u,v) &\to \Bigg(\s,\u,\frac{(\s-v)(1-\u)}{(1+v)}\Bigg)\\
\label{whl:ev: inv psl 2}\pee^{-1}_2: \PEi_0 (\whl )\to \PTa (\whl ) && 
(\s,t,\u,v) &\to \Bigg(\s,\u,\frac{(\u-t)(1-\s)}{(1+t)}\Bigg).
\end{align}

The results of this and the previous subsections are summarised in the following lemma. This provides an alternative proof of Lemma \ref{whl:degree one lemma}.
\begin{lem}
The varieties $\Ta (\whl )$ and $\Ei_0 (\whl )$ are birationally equivalent, and so are the varieties $\PTa (\whl )$ and $\PEi_0 (\whl ).$
\end{lem}


\subsection{Dehn fillings on one cusp}

If one of the cusps is assumed to be complete, then the resulting subvariety, which parametrises hyperbolic Dehn fillings on the other cusp, is defined by: 
\begin{align}
0 &= 1 - t + 4\s t - \s^2 t + \s^2 t^2 &&\text{when } \u = 1,v = -1,\\
0 &= 1 - v + 4\u v - \u^2v + \u^2v^2 &&\text{when } \s= 1, t = -1.
\end{align}
The boundary curves $(0,0,0,1),$ $(0,0,4,1)$ and $(0,1,0,0),$ $(4,1,0,0)$ are detected by these curves  respectively. The above equations may be written as the following well--defined trace relations: 
\begin{equation}
4 =  \tr\prho(\m_i^2) - \tr\prho(\m_i^2\l_i^t)
  = \tr\prho(\m_i^2) - \tr\prho(\l_i^g) \text{ where } i\in\{0,1\}.
\end{equation}
Curves in $\D(\tri_\whl)$ corresponding to these curves in the eigenvalue variety can readily be determined. Consider points in $\PEi_0 (\whl )$ of the form $(1,-1,\u, v).$ The preimage under $\edefo$ of such a point is of the form $\varphi^{-1} (w,-w^{-1},w,-w^{-1})$ for some  $w\in \C-\{0,\pm 1\},$ and hence $H'(\m_1) = \frac{1+w}{w(1-w)}$ and $H'(\l^t_1) = w^2.$


\subsection{Eigenvalues at $\mathbf{(w,x,y,z) = (1,1,1,1)}$}
\label{whl:eigenvalues at (1,1,1,1)}

Using the implicit function theorem, the eigenvalues at the ideal point 
\begin{equation*}
\xi = (0,1,-1,0,1,-1,0,1,-1,0,1,-1)
\end{equation*}
where $w,x,y,z \to 1$ can be computed.
Note that this point maps to a singularity of $\D'(\tri_\whl).$ We use the
projection $(w,x,y,z) \to (w,y,z),$ with image defined by:
\begin{equation*}
0= g(w,y,z) = w - wy - wz + yz - wy^2z^2 + w^2y^2z^2.
\end{equation*}
The holonomies of meridians take the form
\begin{equation*}
H'(\m_0) = - yz \frac{1-w}{1-y},
\qquad 
H'(\m_1) = - \frac{1-z}{yz(1-w)},
\end{equation*}
and as $\xi$ is approached, the eigenvalues of the longitudes converge to one.
Since the growth rates of all parameters are equal, let 
\begin{align*}
w = 1 + \eps, && y = 1 - \s^{-1} \eps + \eps Y && z = 1 - \u \eps,
\end{align*}
and we are seeking $Y(\eps)$ such that 
$g(w(\eps),y(\eps),z(\eps))=0$ for all
$\eps\in B_\delta(0)$ and $Y(0)=0$ for some $\s,\u \in \C-\{0\}.$
Substitution yields a polynomial
$F(\eps,Y),$ and the implicit function theorem applies if
$(\s-1)(\u-1)=0.$ We have
\begin{equation*}
H'(\m_0) = \frac{(1-\u \eps)(1-\s^{-1} \eps+\eps Y)}{\s^{-1}+Y}
\quad\text{and}\quad
H'(\m_1) = \frac{\u}{(1-\u \eps)(1+\eps)}.
\end{equation*}
Thus, for any $\s,\u \in \C-\{0\},$ one can obtain the limiting eigenvalues
$(1,1,\u,1)$ and $(\s,1,1,1)$ as $\xi$ is approached. These eigenvalues are
precisely the eigenvalues of the reducible representations in $\PTa(\whl).$


\section{Character varieties}
\label{whl:Character varieties}

A complete description of the character varieties associated to $\whl$ is now given.

Since $\Ta (\whl )$ is irreducible, it follows that the $\SL$--character variety is of the form $\X(\whl) = \X_0(\whl) \cup \X^r(\whl),$ where $\X_0(\whl)$ is the Dehn surgery component and $\X^r(\whl)$ is the set of reducible characters. A defining equation for the $\SL$--Dehn surgery component can be worked out from $\Ta (\whl ).$ Let $X = \tr\rho (\m_0 ), Y = \tr\rho (\m_1 )$ and $Z = \tr\rho (\m_0\m_1 ),$ then $\X_0(\whl )$
is defined by
\begin{equation}\label{eq:SL_Dehn surgery component}
0 = F (X,Y,Z) = XY + (2 - X^2 - Y^2)Z + XYZ^2 - Z^3.
\end{equation}
The set of reducible characters is parametrised by $\tr\rho[\m_0,\m_1]=2,$ which is equivalent to:
\begin{equation*}
4 =  X^2 + Y^2 + Z^2 - XYZ.
\end{equation*}

The form of the relator in (\ref{whl:fundamental group}) implies that a $\PSL$--representation lifts to $\SL$ if and only if the relator is equal to
the identity in $\SL$ for any assignment of matrices representing the $\PSL$--representation. The Dehn surgery component of the
$\PSL$--character variety can be worked out from (\ref{eq:SL_Dehn surgery component}). With $\pX = X^2, \pY = Y^2$ and $\pZ = Z^2$ one obtains:
\begin{align*}
0 &= \pF (\pX ,\pY ,\pZ ) \\
&= 
-\pX \pY + (4 - 4 \pX + \pX^2 - 4\pY + \pY^2)\pZ 
- (4 - 2\pX - 2\pY + \pX\pY)\pZ^2  + \pZ^3.
\end{align*}
Note that $\pF (X^2, Y^2, Z^2) = F (X,Y,Z) F(-X,Y,Z) = F (X,Y,Z) F(X,-Y,Z).$

Computation of $\PSL$--representations which do not lift to $\SL$ reveals
that there is only a finite set, parametrised by
$(\pX, \pY, \pZ) = (0,0,2 \pm \sqrt{2}).$ These points do not satisfy $\pF,$
and the corresponding representations are irreducible on the peripheral
subgroups, with image isomorphic to $\Z_2 \oplus \Z_2.$ In particular, they
do not contribute any points to the $\PSL$--eigenvalue variety. Thus, all
boundary curves which are detected by $\PEi_0(\whl)$ are also detected by
$\Ei_0(\whl).$


\section{The Logarithmic limit set}
\label{sec:log lim debbie}

The following calculation was first done by Debbie Yuster. 
The ideal formed by the hyperbolic gluing equations and parameter relations given in Section \ref{sec:defo equations} is:
\begin{align*}
I=&\langle wxyz-1, w'x'y'z'w''^2x''^2-1, w'x'y'z'y''^2z''^2-1, \\
    &w-ww''-1, w'-ww'-1, w''-w''w'-1, \\
    & x-xx''-1, x'-x'x-1, x''-x''x'-1,\\
    & y-yy''-1, y'-y'y-1, y''-y''y'-1, \\
    & z-zz''-1, z'-z'z-1, z''-z''z'-1\rangle.
\end{align*}
Solving the first three expressions for $z,$ $z',$ and $z''$ respectively gives:
\begin{equation*}
z=  \frac{1}{wxy},   \qquad
z'=\frac{1}{w'x'y'w''^2x''^2}, \qquad
z''=\frac{w''x''}{y''}.
\end{equation*}
The sign in the last equation is determined by equations (\ref{whl: defo simple relations}). Substitute these values back into the 15 ideal generators above. The first three generators become $0,$ and we obtain an ideal with 12 generators in the variables $w,w',w'',x,x',x'',y,y',$ and $y''.$ The last three of these generators, respectively, are:
\begin{align*}
&\frac{1}{wxy}-\frac{w''x''}{wxyy''}-1,\\
&\frac{1}{w'x'y'w''^2x''^2}-\frac{1}{w'x'y'w''^2x''^2wxy}-1,\\
&\frac{w''x''}{y''}-\frac{1}{w''x''y''w'x'y'}-1.
\end{align*}
We multiply each of these by the least common multiple of its denominators, thus clearing denominators. This gives the ideal:
\begin{align*}
J=&\langle w-ww''-1, w'-ww'-1, w''-w''w'-1, x-xx''-1, x'-x'x-1,  \\
& x''-x''x'-1, y-yy''-1, y'-y'y-1, y''-y''y'-1, y''-w''x''-wxyy'', \\
& w'w''^2x'x''^2y'-1-y''w''x''w'x'y',  wxy-1-w'w''^2x'x''^2y'wxy\rangle .
\end{align*}
Since $ww'w''+1,$ $xx'x''+1,$ and $yy'y''+1$ are in $J,$ we make the following substitutions:
\begin{equation*}
w''=  \frac{-1}{ww'},\qquad
x''=\frac{-1}{xx'},\qquad
y''=\frac{-1}{yy'}
\end{equation*}
and clear denominators as before. After deleting repeated elements, we are left with four generators:
\begin{equation*}
K=\langle ww'+1-w', xx'+1-x', yy'+1-y', yy'-w^2x^2yw'x'+ww'xx'\rangle.
\end{equation*}

We relabel the variables so they are alphabetically contiguous, starting from the beginning of the alphabet, as follows:
$w\rightarrow a,$ $w'\rightarrow b,$ $x\rightarrow d,$ $ x'\rightarrow e,$ $ y\rightarrow c,$ $ y'\rightarrow f.$

We now perform a few operations to make computing the tropical variety easier. Namely, 
we homogenise the ideal by first homogenising the generators and then saturating with respect to the homogenising variable. For ideals $I_0$ and $J_0,$ the saturation is the set of all elements $f$ in the ambient ring such that $J_0^Nf$ is contained in $I_0$, for some integer $N$. Saturating an ideal by a principal monomial ideal does not alter the tropical variety. 

We homogenise using the variable $g$ and call the resulting ideal $L$:
\begin{equation*}
L=\langle ab+g^2-bg, de+g^2-eg, cf+g^2-fg, cfg^5-a^2d^2cbe+abdeg^3\rangle.
\end{equation*}
We saturate $L$ with respect to $\langle g\rangle$ and find minimal generators, using \texttt{Macaulay 2}  \cite{M2}. The final ideal is:
\begin{align*}
&\langle de-eg+g^2, cf-fg+g^2, ab-bg+g^2,\\
&acd-ace+acg-bcd+bce-bcg-beg+bg^2+cdg-ceg+cg^2+eg^2-fg^2\rangle.
\end{align*}
Next we run \texttt{Gfan}'s commands \texttt{tropical\_starting\_cone} and \texttt{tropical\_traverse} on these generators, and we obtain a tropical variety isomorphic to the tropical variety of $I$ with the isomorphism determined by the above substitutions. The result is shown in Figure~\ref{fig:whl_log_lim-D}.



\address{Stephan Tillmann\\School of Mathematics and Statistics F07, The University of Sydney, NSW 2006 Australia\\{stephan.tillmann@sydney.edu.au}}

\Addresses



\begin{thebibliography}{99}

\bibitem{gb} George M. Bergman: \emph{The logarithmic limit--set of an algebraic variety}, Trans. Am. Math. Soc., 157, 459-469, (1971).

\bibitem{bjsst} Tristram Bogart, Anders Jensen, David Speyer, Bernd Sturmfels, Rekha Thomas: \emph{Computing Tropical Varieties},  J. Symbolic Comput. 42 (2007), no. 1-2, 54--73. 

\bibitem{BZ96} Steven Boyer and Xingru Zhang: \emph{Finite Dehn surgery on knots},
J. Amer. Math. Soc. 9 (1996), no. 4, 1005--1050. 

\bibitem{bozh} S. Boyer, X. Zhang: \emph{On Culler--Shalen seminorms
   and Dehn filling}, Ann. Math. 148 (1998) 737-801.

\bibitem{BG1984} Robert Bieri, John R.J.\thinspace Groves: \textit{The geometry of the set of characters induced by valuations}, Jour. reine und angew. Math. 347 (1984) 168-195.

\bibitem{bab} Benjamin A. Burton, Ryan Budney, William Pettersson, et al.,
Regina: Software for low-dimensional topology, \url{http://regina-normal.github.io/}, 1999--2017. 

\bibitem{CaLuTi} Alex Casella, Feng Luo and Stephan Tillmann: \emph{Pseudo-developing maps for ideal triangulations II: Positively oriented ideal triangulations of cone-manifolds}, Proceedings of the American Mathematical Society 145 (2017), 3543--3560. 

\bibitem{ac} Abhijit Champanerkar: \emph{A-polynomial and Bloch invariants of hyperbolic 3-manifolds}, Thesis (Ph.D.), Columbia University. 2003.

\bibitem{ccgls} Daryl Cooper, Marc Culler, Henri Gillet, Darren D. Long and Peter B. Shalen:
  \emph{Plane curves associated to character varieties of 3--manifolds},
      Invent. Math., 118, 47-84 (1994).

\bibitem{clo} D. Cox, J. Little, D. O'Shea: \emph{Ideals, Varieties,
   and Algorithms}, Springer, New York, 1992.

\bibitem{snappy} 
Marc Culler, Nathan M. Dunfield, Matthias Goerner and Jeffrey R. Weeks: \emph{Snap{P}y, a computer program for studying the geometry and topology of $3$-manifolds}. Available at \url{http://snappy.computop.org}.
    
\bibitem{CGLS}  Marc Culler, Cameron McA. Gordon, John Luecke and Peter B. Shalen: \emph{Dehn surgery on knots.} Ann. of Math. (2) 125 (1987), no. 2, 237?300.

\bibitem{cs} Marc Culler, Peter B.~Shalen: \emph{Varieties of group representations and splittings of 3-manifolds}, Ann. Math., 117, 109-146 (1983).

\bibitem{DuGa} Nathan M. Dunfield, Stavros Garoufalidis: \emph{Incompressibility criteria for spun-normal surfaces},  Trans. Amer. Math. Soc. 364 (2012), no. 11, 6109--6137.

\bibitem{fh} William J. Floyd, Allen E. Hatcher: \emph{The space of incompressible surfaces in a
2--bridge link complement}, Trans. Amer. Math. Soc., 305, 575-599 (1988).

\bibitem{FuGe}  David Futer and Francois Gu\'eritaud: \emph{From angled triangulations to hyperbolic structures.} Interactions between hyperbolic geometry, quantum topology and number theory, 159--182, Contemp. Math., 541, Amer. Math. Soc., Providence, RI, 2011.

\bibitem{M2} Daniel R. Grayson and Michael E. Stillman,  \emph{Macaulay2, a software system for research in algebraic geometry}. Available at \url{http://www.math.uiuc.edu/Macaulay2/}

\bibitem{haken61-knot}
Wolfgang Haken, \emph{Theorie der {N}ormalfl{\"a}chen}, Acta Math. \textbf{105}
  (1961), 245--375.

\bibitem{hmw} Craig D. Hodgson, G. Robert Meyerhoff, Jeffrey R. Weeks: \emph{Surgeries on
the Whitehead Link Yield Geometrically Similar Manifolds}, Topology `90,
     195-206 (1992).

\bibitem{hs}  Jim Hoste and Patrick D.  Shanahan: \emph{Computing boundary slopes of 2-bridge links.} Math. Comp. 76 (2007), no. 259, 1521--1545.

\bibitem{gi} Gabriel Indurskis; \emph{Culler-Shalen seminorms of fillings of the Whitehead link exterior}, preprint, arXiv:math.GT/0611713.

\bibitem{gfan} Anders Jensen: \emph{Gfan}, available at \url{http://home.math.au.dk/jensen/software/gfan/gfan.html}.

\bibitem{kang05-spun}
Ensil Kang, \emph{Normal surfaces in non-compact 3-manifolds}, J. Aust. Math.
  Soc. \textbf{78} (2005), no.~3, 305--321.

\bibitem{KR-2015} Ensil Kang and Hyam Rubinstein: \emph{Note on spun normal surfaces in 1-efficient ideal triangulations}, Proc. Japan Acad., 91, Ser. A (2015).

\bibitem{la} Alan E. Lash: \emph{Boundary Curve Space of the Whitehead
   Link Complement}, Ph. D. thesis, UC Santa Barbara, 1993.

\bibitem{LuTi} Feng Luo and Stephan Tillmann: \emph{Angle structures and normal surfaces}, Transactions of the American Mathematical Society 360 (2008) 2849--2866. 

\bibitem{Mo} John W. Morgan and Hyman Bass (eds.): \emph{The Smith conjecture},
Papers presented at the symposium held at Columbia University, New York, 1979. Pure and Applied Mathematics, 112. Academic Press, Inc., Orlando, FL, 1984.

\bibitem{ms1} John W. Morgan and  Peter B. Shalen: \emph{Valuations, trees, and degenerations of hyperbolic structures I}, Ann. of Math. (2) 120, no. 3, 401--476 (1984).
    
\bibitem{nr} Walter D. Neumann and Alan W. Reid: \emph{Arithmetic of Hyperbolic Manifolds}, Topology 90, 273-312 (1992).

\bibitem{nz} Walter D. Neumann and Don Zagier: \emph{Volumes of hyperbolic three--manifolds} Topology, 24, 307-332 (1985).

\bibitem{ST} Henry Segerman and Stephan Tillmann: \emph{Pseudo-developing maps for ideal triangulations I: Essential edges and generalised hyperbolic gluing equations}. Topology and geometry in dimension three, 85-102, Contemp. Math., 560, Amer. Math. Soc., Providence, RI, 2011. 

\bibitem{sh1} Peter B. Shalen: \emph{Representations of 3--manifold groups}, Handbook of geometric topology, 955--1044, North-Holland, Amsterdam, 2002.

\bibitem{t} William P. Thurston: \emph{The geometry and topology of 3--manifolds}, Princeton Univ. Math. Dept. (1978). Available at \url{http://msri.org/publications/books/gt3m/}.

\bibitem{thu-norm} William P. Thurston: \emph{A norm for the homology of 3-manifolds.}
Mem. Amer. Math. Soc. 59 (1986), no. 339, i--vi and 99--130. 

\bibitem{tillus_kino} Stephan Tillmann: \emph{On the Kinoshita--Terasaka mutants and
    generalised Conway mutation}, J. Knot Theory Ramifications 9 (2000)
  557-575.

\bibitem{tillus_mut} Stephan Tillmann: \emph{Character varieties of mutative 3-manifolds}, Algebr. Geom. Topol. 4, 133--149 (2004).

\bibitem{tillus_ei} Stephan Tillmann: \emph{Boundary slopes and the logarithmic limit set}, Topology 44, no. 1, 203--216 (2005).

\bibitem{tillmann08-finite} Stephan Tillmann: \emph{Normal surfaces in topologically finite 3--manifolds}, L'Enseignement Math\'ematique, 54 (2008) 329--380. 

\bibitem{defo} Stephan Tillmann: \emph{Degenerations of ideal hyperbolic triangulations}, Mathematische Zeitschrift 272 (2012), no. 3-4, 793--823. 

\bibitem{splittings} Stephan Tillmann: \emph{Geometric splittings at ideal points}, in preparation.

\bibitem{wa} Genevieve Walsh: \emph{Incompressible surfaces and spunnormal form}, Geom. Dedicata 151 (2011), 221--231.

\bibitem{y} Tomoyoshi Yoshida: \emph{On ideal points of deformation curves of hyperbolic 3--manifolds with one cusp}, Topology, 30, 155-170 (1991).

\end{thebibliography}
\end{document}